\documentclass{article}
\usepackage{arxiv}

\usepackage[normalem]{ulem} 

\usepackage[utf8]{inputenc} 
\usepackage[T1]{fontenc}    
\usepackage{hyperref}       
\usepackage{url}            
\usepackage{booktabs}       
\usepackage{amsfonts}   
\usepackage{amsmath}
\usepackage{microtype}      
\usepackage{color}
\usepackage{lipsum}
\usepackage{amsthm}
\usepackage{amssymb}
\usepackage{indentfirst}
\usepackage{graphicx}
\usepackage{epstopdf}
\usepackage{fourier}
\usepackage{bm}
\usepackage{bbm} 
\usepackage{mathtools}
\usepackage{todonotes}
\usepackage{latexsym,enumerate}
\usepackage{ulem}
\usepackage{subcaption}

\newcommand{\comment}[1]{}

\newcommand{\BEA}{\begin{eqnarray}}
\newcommand{\EEA}{\end{eqnarray}}
\newcommand{\bk}{{\bf k}}

\newcommand{\bx}{{\bf x}}
\newcommand{\by}{{\bf y}}

\newcommand{\BR}{\mathbb{R}}

\newtheorem{lem}{Lemma}[section]
\newtheorem{theo}{Theorem}[section]
\newtheorem{prop}{Proposition}[section]
\newtheorem{defn}{Definition}[section]

\newtheorem{assu}{Assumption}[section]
\newtheorem{remark}{Remark}

\title{Diffusion Maps Kernel Ridge Regression}

\author{
 John Harlim\thanks{Corresponding Author} \\
  Department of Mathematics\\ Institute for Computational and Data Sciences \\
  The Pennsylvania State University, University Park, PA 16802, USA\\
  \texttt{jharlim@psu.edu} 
  \And
 Daning Huang, Jiwoo Song \\
  Department of Aerospace Engineering \\
  The Pennsylvania State University, University Park, PA 16802, USA\\
  \texttt{daning@psu.edu,jzs6565@psu.edu} 
}

\begin{document}

\maketitle

\begin{abstract}
In this paper, we study kernel ridge regression using the data-driven diffusion maps (DM) kernel, which is constructed through algebraic manipulations of the diffusion maps algorithm. Under the assumptions that the data lie on a manifold and are sampled uniformly, we prove that {\color{black}a family of kernels induced by temporally scaled leading orthonormal basis of the} DM kernel converges uniformly to the heat kernel on the manifold for sufficiently large times as the dataset size increases and the kernel bandwidth is scaled appropriately. Consequently, the limiting Reproducing Kernel Hilbert Spaces (RKHSs) {\color{black}of these family of induced DM kernels} coincide with the RKHSs associated with the heat kernels on the manifold. We further show that the RKHS induced by the DM kernel is isometrically isomorphic to the RKHS of the Gaussian kernel, which is continuously embedded in the RKHS of a Mat\'ern kernel whose norm is equivalent to an appropriate Sobolev norm. This result implies that standard risk bounds for kernel ridge regression applicable to Mat\'ern kernels also apply to the DM kernel. Finally, we provide numerical results that (1) validate the convergence of the heat kernel approximation, (2) demonstrate the greater expressiveness of the DM kernel compared to the Gaussian kernel for supervised learning over a larger class of functions on manifolds with boundary, and (3) demonstrate the advantage of the DM kernel over the Gaussian kernel in learning functions with varying frequencies and co-dimensions.
\end{abstract}

\section{Introduction}

The effectiveness of kernel methods such as Kernel Ridge Regression (KRR) are fundamentally determined by the choice of kernel, because the kernel determines the hypothesis space, i.e., the Reproducing Kernel Hilbert Space (RKHS) in which learning is performed \cite{scholkopfsmola2002,christmann2008support}. In practice, this fact means the model selection problem is analogous to a kernel selection problem. Particularly, one should choose a kernel such that the induced RKHS is rich enough to approximate the unknown target function while retaining sufficient regularity for stable generalization.

An existing approach to choosing kernel is to identify the kernel by fitting the data to a class of kernels of choice, rather than committing to a particular kernel. For example, a semi-definite programming (SDP) is proposed in \cite{LanckrietEtAl2004} for identifying the kernel matrix directly from a class of admissible kernels. However, this SDP based method is not scalable for large sample size and is superseded by Multiple Kernel Learning that solves a convex dual optimization problem to identify nonnegative weights over a fixed dictionary of kernels \cite{BachLanckrietJordan2004}. Kernel flows generalizes this idea by solving a non-convex optimization problem \cite{OwhadiScovelSuleiman2019}. 
Although this class of approaches enlarges the family of admissible kernels, it still requires selecting and optimizing within a predefined (finite set of) kernel class, leaving persistent challenges in representation bias, computational scalability, and out-of-sample generalization. 

A complementary direction is to construct a kernel that explicitly encodes the geometry and regularity inherent in the input data, such that the induced RKHS can represent the structure of the target function without brute-force kernel search.
Recently, \cite{song2025learning} proposed a data-driven Diffusion Maps (DM) kernel for learning solution operators of dynamical systems. Empirically, they showed that kernel ridge regression with the DM kernel (DMKRR) consistently outperforms state-of-the-art random-feature, neural-network, and other operator-learning methods in both long-horizon prediction accuracy and data efficiency. Their numerical results suggest that the DM kernel induces a hypothesis space well suited to represent observables (i.e., functions) defined on the forward-invariant set of the dynamical system encoded by the training samples. A rigorous mathematical explanation of this success is challenging due to the fact that they tuned the KRR with a validation metric that is different from the regression loss function, and the data are not independent and identically distributed (i.i.d.). 

In this paper, we provide a rigorous mathematical formalization of DMKRR 
for learning functions, which is a simpler setup than learning dynamical systems \cite{song2025learning}. First, under the manifold assumption and uniformly distributed samples, we show that {\color{black}a family of kernels induced by the temporally scaled orthonormal basis of the} DM kernel converges to the heat kernel associated with the Laplace-Beltrami operator for sufficiently large time $t$ that is lower bounded by a positive quantity of order $o(N^{-2/d}\log N)$ in the large-sample limit with an appropriately vanishing bandwidth parameter (see Theorem~\ref{thm1}). This result suggests that for such $t>0$, the RKHS of the temporally scaled {\color{black}induced} DM kernel approximates the RKHS of the heat kernel. This is a stronger mode of convergence than the one reported in \cite{dunson2021spectral}, where convergence was only uniformly over the training dataset. Second, in the same setting, we derive a risk bound (oracle inequality) for kernel ridge regression with the diffusion maps kernel (DMKRR), see Theorem~\ref{thm2}. Our analysis is based on the general result of \cite{steinwart2009optimal}, which applies to RKHSs induced by Mat\'ern kernels on arbitrary compact metric spaces. To obtain this bound, we prove that the RKHS induced by the DM kernel is isometrically isomorphic to the RKHS induced by the Gaussian kernel that is continuously embedded in the RKHS induced by a Mat\'ern kernel, with a norm equivalent to a Sobolev norm of suitable regularity. 

Numerically, we conduct the following verifications. First, we validate the convergence of the heat kernel estimation on simple closed manifolds, circle, flat torus, and a two-dimensional disk, whose analytical heat kernels are known. Second, we design a simple test example to identify which function classes are well suited for the DM kernel. Since the DM kernel is a rescaled Gaussian RBF kernel, we also identify the classes of functions that are well approximated by each kernel. Third, we also verify the error rates of regression of smooth labels of varying frequencies and co-dimensions.

This paper is organized as follows: In Section~\ref{sec2}, we review some mathematical preliminaries to formalize the notation used in the paper and the setting. In Section~\ref{sec3}, we prove the convergence of the DM kernel to the heat kernel of the Laplace-Beltrami operator. In Section~\ref{sec4}, we discuss the convergence of the Diffusion Maps Kernel Ridge Regression (DMKRR). In Section~\ref{sec5}, we report the numerical performance of DMKRR simulations compared to the RBF Gaussian kernel. We conclude the paper with a short summary in Section~\ref{sec6}. For completeness, we provide three appendices: In Appendix~\ref{app:A}, we report the asymptotic expansion of the continuous DM kernel that is useful for the theoretical discussion. In Appendix~\ref{app:proof}, we document proofs of some intermediate results (Lemmas). In Appendix~\ref{app:num}, we document additional details and numerical results that compliment Section~\ref{sec5}. 

\section{Overview of mathematical setting}\label{sec2}


In this paper, let $\mathcal{M}$ be a $d$-dimensional closed (compact with no boundary) smooth sub-manifold of $\BR^n$ and $X = \{\bx_1,\ldots, \bx_N\} \subset \mathcal{M}$ be a set of i.i.d random samples with distribution $\mu$ that is empirically estimated by the empirical measure $\mu_N = \frac{1}{N}\sum_{i=1}^N\delta(\bx-\bx_i)$. The following basic result will be important for the analysis in this paper.  

\begin{lem}\label{samplingassumption} (Lemma B.2 in \cite{harlim2023radial}) Let the training data set $X$ be sampled i.i.d. with sampling distribution $\mu$ that is absolutely continuous with respect to the volume measure, $d\mu = q d\text{Vol}$. Assume that the sampling density is bounded away from zero and infinity, $0<q_{\min}\leq q \leq q_{\max}$. Then, with probability higher than $1-\frac{1}{N}$, the fill distance,
\BEA
h = h_{X,\mathcal{M}}:= \sup_{\bx\in \mathcal{M}}\min_{\bx_i\in X}d_g(\bx,\bx_i) \leq  C\left(\frac{\log N}{N}\right)^{\frac{1}{d}},\label{filldistance}
\EEA where 
the constant $C$ depends on the intrinsic dimension $d$ of $\mathcal{M}$ and on the lower bound $q_{\min}$, but independent of $N$.
\end{lem}

Given a Gaussian kernel, $\tilde{k}_{\epsilon}:\mathcal{M}\times \mathcal{M} \to \mathbb{R}$, defined as, 
\begin{equation}\label{eqn_rbf}
\tilde{k}_{\epsilon}(\bx,\by) = \exp\left(-\frac{\|\bx-\by\|^2}{4\epsilon}\right),
\end{equation}
the natural assumption with the fill distance is the choice of kernel bandwidth $\epsilon^{1/2} \gg h$. 

We define the \emph{Continuous Diffusion Maps (CDM) kernel}, $k_{\epsilon}:\mathcal{M}\times\mathcal{M}\to \BR$, through the following algebraic procedure,
\begin{equation}
\begin{aligned}
q_\epsilon(\bx) &= \int_{\mathcal{M}} \tilde k_{\epsilon}(\bx, \by)q(\by)\,d\text{Vol}(\by),\qquad  \hat{k}_{\epsilon}(\bx,\by) =  \frac{\tilde
  k_{\epsilon}(\bx, \by)}{q_{\epsilon}(\bx)q_{\epsilon}(\by)},\\ \hat{q}_{\epsilon}(\bx) &= \int_{\mathcal{M}} \hat{k}_{\epsilon}(\bx,\by) q(\by)d\text{Vol}(\by), \qquad 
  k_{\epsilon}(\bx,\by) = \frac{ \hat{k}_{\epsilon}(\bx,\by)}{\sqrt{\hat{q}_{\epsilon}(\bx)\hat{q}_{\epsilon}(\by)}}.  \label{analyticdmalgebra}
  \end{aligned}
\end{equation}
We note that $q$ denotes the sampling density of the dataset ($d\mu = qd\text{Vol}$).

Numerically, since the sampling measure $\mu$ is unknown, $k_\epsilon$ is not practically available. We approximate the kernel with the empirical measure $\mu\approx \mu_N$ and define:

\begin{defn} The Diffusion Maps (DM) kernel $k_{\epsilon,N}:\mathcal{M}\times \mathcal{M} \to \mathbb{R}$ is defined as,
\begin{equation}
\begin{aligned}
q_{\epsilon,N}(\bx) &= \frac{1}{N}\sum_{j = 1}^N \tilde k_\epsilon(\bx, \bx_j),\qquad
\hat k_{\epsilon,N}(\bx, \by) = \frac{\tilde k_{\epsilon}(\bx,
\by)}{q_{\epsilon,N}(\bx)  q_{\epsilon,N}(\by)},\\
\hat{q}_{\epsilon,N}(\bx) &= \frac{1}{N}\sum_{j = 1}^N \hat
k_{\epsilon,N}(\bx, \bx_j),\qquad
    k_{\epsilon,N}(\bx,\by) = \frac{\hat{k}_{\epsilon,N}(\bx,\by)}{\sqrt{\hat{q}_{\epsilon,N}(\bx)\hat{q}_{\epsilon,N}(\by)}} .\label{scalarkernel}
\end{aligned}
\end{equation}
We note that this kernel is data-driven (as it depends on $N$ in addition to the kernel bandwidth $\epsilon$), and is a discrete estimator of the CDM kernel in \eqref{analyticdmalgebra}. 
\end{defn}

\comment{\color{black} Define 
\[
\overline{q}_{\epsilon,N} = c q_{\epsilon,N},
\]
for any constant c>0. Then, 
\[
\hat{\overline{k}}_{\epsilon,N}(\bx, \by) = \frac{\tilde k_{\epsilon}(\bx,
\by)}{q_{\epsilon,N}(\bx)  q_{\epsilon,N}(\by)} = \frac{1}{c^2} \hat k_{\epsilon,N}(\bx, \by) 
\]
and
\[
\hat{\overline{q}}_{\epsilon,N}(\bx) = \frac{1}{N}\sum_{j = 1}^N \hat{
\overline{k}}_{\epsilon,N}(\bx, \bx_j) = \frac{1}{c^2} \hat{q}_{\epsilon,N}(\bx). 
\]
so
\[
\hat{\overline{k}}_{\epsilon,N}(\bx,\by) = \frac{\hat{k}_{\epsilon,N}(\bx,\by)}{\sqrt{\hat{q}_{\epsilon,N}(\bx)\hat{q}_{\epsilon,N}(\by)}} = k_{\epsilon,N}(\bx,\by).
\]
}

\begin{remark}\label{DMconnection}
    We should point out that the DM kernel is not defined in the original Diffusion Maps algorithm \cite{coifman2006diffusion,dunson2021spectral}, but is related as follows. Conventionally, the DM algorithm does not use the factor of $\frac{1}{N}$ in both normalizations with the intention of constructing a Markov matrix, with a unit column sum. Particularly, if we define  
\begin{equation}
q_{\epsilon,N}^{DM}(\bx) = \sum_{j = 1}^N \tilde k_\epsilon^{DM}(\bx, \bx_j),\quad
\hat k_{\epsilon,N}^{DM}(\bx, \by) = \frac{\tilde k_{\epsilon}(\bx,
\by)}{q_{\epsilon,N}^{DM}(\bx) q_{\epsilon,N}^{DM}(\by)},\quad
\hat{q}_{\epsilon,N}^{DM}(\bx) = \sum_{j = 1}^N \hat
k_{\epsilon,N}^{DM}(\bx, \bx_j), \nonumber
\end{equation}
the DM algorithm constructs the following Markov transition kernel,
\[
P^{DM}_{\epsilon,N}(\bx,\by) = \frac{\hat k_{\epsilon,N}^{DM}(\bx, \by)}{\hat{q}_{\epsilon,N}^{DM}(\bx)},
\]
of a reversible Markov chain on $\mathcal{M}$.
One can verify that
\BEA \sqrt{\hat{q}_{\epsilon,N}^{DM}(\bx)} \frac{P_{\epsilon,N}^{DM} (\bx,\by)}{\sqrt{\hat{q}_{\epsilon,N}^{DM}(\by)}} =  \frac{\hat{k}_{\epsilon,N}^{DM}(\bx,
\by)}{\sqrt{\hat{q}_{\epsilon,N}^{DM}(\bx) \hat{q}_{\epsilon,N}^{DM}(\by)}}= \frac{1}{N}k_{\epsilon,N}(\bx,\by).\label{DMsymmetrickernel}\EEA
This suggests that the Gram matrix $\mathbf{P}_{\epsilon,N}^{DM}$ of $P_{\epsilon,N}^{DM}$ is diagonally conjugate to the Gram matrix $\frac{1}{N}\mathbf{K}_{\epsilon,N}$ of $\frac{1}{N}k_{\epsilon,N}$. So, eigenvalues of $\mathbf{P}_{\epsilon,N}^{DM}$ and $\frac{1}{N}\mathbf{K}_{\epsilon,N}$ are the same. In the remainder of this paper, we denote the eigenvalues of these matrices as $\left\{\lambda_j^{\epsilon,N}\right\}$. Moreover, one can verify that if $\varphi_j^{DM}$ denotes an eigenvector of $\mathbf{P}_{\epsilon,N}^{DM}$ associated to eigenvalue $\lambda_j^{\epsilon,N}$, then $\varphi_j^{\epsilon,N}= \sqrt{\hat{q}_{\epsilon,N}^{DM}} \varphi_j^{DM}$ is an eigenvector of $\frac{1}{N}\mathbf{K}_{\epsilon,N}$ corresponding to the same eigenvalue. 
\end{remark}

In this remainder of this section, we first identify the Reproducing Kernel Hilbert Space (RKHS) corresponding to the Diffusion Maps kernel in \eqref{scalarkernel} for fixed $\epsilon>0$ and $N\in \mathbb{N}^+$. Subsequently, we give a brief review on the spectral convergence result of the DM operator.

\subsection{RKHS of the DM kernel}\label{sec21}
To characterize the RKHS corresponding to $k_{\epsilon, N}$, we 
define $S_{\epsilon,N}: L^2(\mu_N) \to \mathcal{H}_{\epsilon,N}$, as, 
\begin{eqnarray}
S_{\epsilon,N} f = \int_{\mathcal{M}} k_{\epsilon,N}(\cdot,\by) f(\by) d\mu_N(\by).\label{sk2}   
\end{eqnarray}
One can show that its adjoint $S_{\epsilon,N}^*: \mathcal{H}_{\epsilon,N} \to L^2(\mu_N)$ is an identity operator. 

\comment{\color{black}
\[
\langle f, \text{id}g\rangle_{L^2(\mu_N)} = \int_{\mathcal{M}} f(\bx) g(\bx)d\mu_N(\bx) = \int_{\mathcal{M}} f(\bx) \langle g, k_{\epsilon,N}(\cdot, \bx)\rangle_{\mathcal{H}_{\epsilon,N}} d\mu_N(\bx) = \langle S_{\epsilon,N} f , g\rangle_{\mathcal{H}_{\epsilon,N}}
\]
}

  Define $T_{\epsilon,N} = S_{\epsilon,N}^*S_{\epsilon,N}$ which is the same integral operator as above, except that it is mapping from $L^2(\mu_N)$ to $L^2(\mu_N)$.
For any bounded kernel, $\|k_{\epsilon,N}\|_{L^2(\mu_N)} <\infty$, the spectral theory (see Theorem 4.9.19 in \cite{debnath2005introduction}) states that $T_{\epsilon,N}$ has real eigenvalues $\lambda_0^{\epsilon,N}\geq  \lambda_1^{\epsilon,N} \geq \ldots \searrow 0$. For Markov kernel, $\lambda_0^{\epsilon,N}=1$. We should point out that eigenvalues of $T_{\epsilon,N}$ are identical to eigenvalues of the Gram matrix $\frac{1}{N}\mathbf{K}_{\epsilon,N}$, which are exactly the eigenvalues of the matrix constructed by the DM algorithm (see \eqref{DMsymmetrickernel}).
The corresponding eigenfunctions, $\{\varphi_i^{\epsilon,N}\}$, form a Riesz basis for $L^2(\mu_N)$. The Mercer's theorem suggests that, $k_{\epsilon,N}$ at the interpolation points can be written as
\begin{eqnarray}
k_{\epsilon,N}(\bx_i,\bx_j) = \sum_{i= 0}^{N-1} \lambda_i^{\epsilon,N} \varphi_i^{\epsilon,N}(\bx_i)\varphi_i^{\epsilon,N}(\bx_j),\qquad \forall \bx_i,\bx_j \in X,\notag
\end{eqnarray}
where $k_{\epsilon,N}$ is a finite-dimensional kernel on $\mathcal M \times \mathcal M$ given by 
\begin{eqnarray}
k_{\epsilon,N}(\bx,\by) = \sum_{i= 0}^{N-1} \lambda_i^{\epsilon,N} \psi_i^{\epsilon,N}(\bx)\psi_i^{\epsilon,N}(\by), \qquad \forall \bx,\by \in \mathcal{M},\label{Mercer}
\end{eqnarray}
where $\psi_i^{\epsilon,N} = (\lambda_i^{\epsilon,N})^{-1} S_{\epsilon,N}\varphi_i^{\epsilon,N}$ is the Nystr\"om extension of $\varphi_i^{\epsilon,N}$. This kernel representation aligns with practical implementation and it is not the same as a truncation of the Mercer's representation of the CDM in \eqref{analyticdmalgebra}.  

We note that $\left\{\sqrt{\lambda_{i}^{\epsilon,N}}\psi_i^{\epsilon,N}\right\}_{i=0}^{N-1}$ forms an orthonormal basis of a finite dimensional space $\mathcal{H}_{\epsilon,N}$. This implies that the RKHS space $\mathcal{H}_{\epsilon,N}$  can be characterized as:
\begin{equation}
\mathcal{H}_{\epsilon,N} = \left\{f = \sum_{i= 0}^{N-1} a_i\sqrt{\lambda_i^{\epsilon,N}} \psi_i^{\epsilon,N}, \sum_{i= 0}^{N-1} a_i^2 < \infty\right\} = \left\{f = \sum_{i= 0}^{N-1} c_i \psi_i^{\epsilon,N},\, \sum_{i= 0}^{N-1} \frac{c_i^2}{\lambda_i^{\epsilon,N}}<\infty\right\},\label{RKHS}   
\end{equation}
that consists of functions of $\mathcal{M}$. 

One of the main questions in this paper is to characterize the limiting RKHS in \eqref{RKHS} as both $N\to \infty$ and $\epsilon \to 0$ with a convergence rate. {\color{black} Theoretically, we will show that one can characterize the limiting RKHS corresponding to a finite sum up to $K-1$:
 \begin{equation}
\mathcal{H}_{\epsilon,N}^K =  \left\{f = \sum_{i= 0}^{K-1} c_i \psi_i^{\epsilon,N},\, \sum_{i= 0}^{K-1} \frac{c_i^2}{\lambda_i^{\epsilon,N}}<\infty\right\},\label{RKHS_K}   
\end{equation}
for $K\leq N$ since the kernel converges to the heat kernel as $K \to \infty$ with a slower rate compared to $\epsilon \to 0$ and $N \to \infty$ (see Remark~\ref{remarkonKandN}). Furthermore, this will not characterize the limiting RKHS of \eqref{RKHS} as discussed in the next section (see Remark~\ref{negativeresult}).}

\subsection{Spectral convergence results} 
The following lemma is adapted to our notation from Theorem~2 and Remark~4 in \cite{dunson2021spectral})

\begin{lem}\label{lem:spectral}
Let $\nu_j$ and $\varphi_j$ be the $j$th eigenvalue and eigenfunction of the negative definite Laplace-Beltrami operator, $\Delta_g \varphi_j = \nu_j \varphi_j$ and assume that $\nu_j$ are simple. For any fixed $j, \epsilon,$ and $N$, define $\nu_j^{\epsilon,N}:=\frac{1}{\epsilon}\log \lambda_j^{\epsilon,N}$. Let $\epsilon^{1/2} \gg \left(\frac{\log N}{N}\right)^{\frac{1}{4d+13}}$ be valid. For a fixed $K\leq N$, define, $\Gamma_K:=\min_{0\leq i\leq  K} d(\lambda_i, \sigma(\Delta_g)\backslash{\lambda_i})$ as
the 
as the smallest spectral gap between the first $K$ eigenvalues. Assume that,
\BEA
\epsilon^{\frac{1}{2}} \leq G(K,\Gamma_K):= C_1\min \left(\left(\frac{\min(\Gamma_K,1)}{C_2+ (-\nu_K)^{\frac{d}{2}+5}}\right)^2,\frac{1}{(C_3+(-\nu_K)^{\frac{5d+7}{4}})^2} \right),
\EEA
for some constants $C_1,C_2,C_3>1$ that depends on $d$, the volume, the injectivity radius, and the second fundamental form of the manifold.
Then, with probability higher than $1-N^{-2}$,
\BEA
\left| \nu_j - \nu_{j}^{\epsilon,N} \right| = O(\epsilon^{3/4}) + O\left(\frac{1}{\epsilon^{d+\frac{5}{2}}}\left(\frac{\log N}{N}\right)^{\frac{1}{2}}\right).  \label{eig_rate}
\EEA
Likewise,
\begin{equation} \left|\varphi^{DM}_j(\bx_i) - \alpha_j\varphi_j(\bx_i) \right| = O(\epsilon^{1/4}) + O\left(\frac{1}{\epsilon^{d+3}}\left(\frac{\log N}{N}\right)^{\frac{1}{2}}\right),\label{eiv_rate}
\end{equation}
where $\alpha_j =\{-1,+1\}$,
for all $\bx_i \in X$. Balancing \eqref{eig_rate} and \eqref{eiv_rate}, we achieve orders $\left(\frac{\log N}{N}\right)^{\frac{3}{8d+26}}$ and $\left(\frac{\log N}{N}\right)^{\frac{1}{8d+26}}$, respectively. 
\end{lem}

Recall from Remark~\ref{DMconnection} that $\varphi_j^{DM}= \hat{q}_{\epsilon,N}^{-1/2}\varphi_j^{\epsilon,N}$, where $\varphi_j^{\epsilon,N}$ is the eigenvector of $\frac{1}{N}\mathbf{K}_{\epsilon,N}$. For uniform distribution $\mu$, we shall see that $q^{\epsilon,N}$ is approximately constant with error smaller than the error rate in \eqref{eiv_rate}, so \eqref{eiv_rate} becomes,
\begin{equation} \left|\varphi^{\epsilon,N}_j(\bx_i) - \alpha_j\varphi_j(\bx_i) \right| = O(\epsilon^{1/4}) + O\left(\frac{1}{\epsilon^{d+3}}\left(\frac{\log N}{N}\right)^{\frac{1}{2}}\right),\label{eiv_rate2}
\end{equation}
where $\left\{\varphi_j^{\epsilon,N}\right\}$ forms an orthonormal basis of $L^2(\mu_N)$.

\begin{assu}\label{assumption1}
For our analysis that uses the spectral convergence result we will choose a kernel bandwidth,
\BEA
h\asymp \left(\frac{\log N}{N}\right)^{\frac{1}{d}} \ll \left(\frac{\log N}{N}\right)^{\frac{1}{4d+13}} \ll \epsilon^{1/2} \leq G(K,\Gamma_K),\label{eps_N}
\EEA
and refers to \eqref{eig_rate} as the spectral error bound. This kernel bandwidth also larger than the fill distance, $h$.
\end{assu}
We should point out that while faster convergence rate, $\left(\frac{\log N}{N}\right)^{\frac{1}{d+4}}$, has been reported in literature \cite{calder2022improved,calder2022lipschitz}, their result is not directly related to our approach since their result is based on a compactly supported kernels and they consider the unnormalized Graph-Laplacian, not the DM normalization. A closely related study \cite{Peoples2025} that normalizes Gaussian kernel also achieves this fast convergence rate for the eigenvalues, but with a slower convergence rate for eigenvectors in $L^2(\mu_N)$. 

The simplicity assumption on the eigenvalues in Lemma~\ref{lem:spectral} is only needed to state
eigenfunction convergence mode-by-mode, since a repeated eigenspace has no
canonical orthonormal basis. In the presence of eigenvalue multiplicities, the
same spectral convergence should instead be interpreted as convergence of the
corresponding eigenspaces with constants depending on the gap between this eigenspace and the rest of the spectrum. In particular, a repeated eigenvalue does not affect the limiting heat kernel.

\section{Limiting {\color{black}truncated} DM kernel under uniform sampled data}\label{sec3}

Let $H$ be the heat kernel of the Laplace--Beltrami operator given by
\BEA
H(\bx,\by;t)=\sum_{j=0}^{\infty}e^{\nu_jt}\varphi_j(\bx)\varphi_j(\by), \label{eq:heatkernel}
\EEA
where the eigenvalues are ordered as
$\nu_0=0\ge \nu_1\ge \nu_2\ge\cdots$ and $\nu_j\to -\infty$.
For every fixed $t>0$, define the positive-time empirical diffusion kernel as
\BEA
H_{\epsilon,N}^{\color{black}K}(\bx,\by;t)
=\sum_{j=0}^{\color{black}K-1}\left(\lambda_j^{\epsilon,N}\right)^{t/\epsilon}
\psi_j^{\epsilon,N}(\bx)\psi_j^{\epsilon,N}(\by),\label{eq:empiricalheatkernel}
\EEA
{\color{black}for $K\leq N$.} Here $\lambda_j^{\epsilon,N}$ denotes the eigenvalue of the self-adjoint
heat-step diffusion maps operator, $T_{\epsilon,N}$ as discussed in Section~\ref{sec2}, ordered as
\[
1=\lambda_0^{\epsilon,N}\ge \lambda_1^{\epsilon,N}\ge\cdots\ge
\lambda_{\color{black}K-1}^{\epsilon,N}\ge 0.
\]
The main result in this section is reported in the following theorem, that is, ${\color{black}H_{\epsilon,N}^K} \to H$ for large enough $t>0$ under as $K,N\to \infty, \epsilon \to 0$ the uniform sampling assumption, extending the convergence in $L^\infty(\bx)$ reported in \cite{dunson2021spectral}. 

\comment{
\begin{theo}\label{thm1}
Assume that $\mathcal M$ is closed, that the data are sampled uniformly on
$\mathcal M$, and that $\epsilon$ satisfies
Assumption~\ref{assumption1}. Suppose also that, for every fixed $s>0$,
\[
    \sup_{\bx\in\mathcal M}H_{\epsilon,N}(\bx,\bx;s)\le C_s
\]
with probability tending to one. Then, for every fixed $t>0$,
\[
    \|H_{\epsilon,N}(\cdot,\cdot;t)-H(\cdot,\cdot;t)\|_{L^\infty(\mathcal M\times\mathcal M)}
    \longrightarrow 0
\]
in probability. 
\end{theo}
}

\begin{theo}\label{thm1}
Assume that $\mathcal M$ is closed, that the data are sampled uniformly on
$\mathcal M$, and that $\epsilon$ satisfies
Assumption~\ref{assumption1}. In addition, we let,
\begin{equation}
C\left( \frac{\log N}{N}\right)^s \leq \epsilon^{1/2} \leq {\color{black}\min \left\{G(K,\Gamma_K),K^{-4}(-\nu_K)^{-\frac{d-1}{2}}\right\}},\label{validregime}
\end{equation}
where $s = \min\left\{\frac{1}{4d+13},\frac{2}{d^2+3d}\right\}$. For $ {\color{black}t \geq \frac{8\log K}{cK^{2/d}}}$, with probability higher than $1-\frac{22}{N}-\frac{5}{N^2}$,
\[
    \|H_{\epsilon,N}^{\color{black}K}(\cdot,\cdot;t)-H(\cdot,\cdot;t)\|_{L^\infty(\mathcal M\times\mathcal M)}
    \leq \frac{C_1}{\color{black}K}(t + 1)e^{t\nu_1} + C_2 \epsilon^{1/4} + C_3  {\color{black}\frac{1}{K}e^{-\frac{cK^{2/d}t}{2}}},
\]
where the constants $C_1, C_2, C_3 >0$ are independent of $N,K, t,$ and $\epsilon>0$, but they depend on the dimension $d$, the Ricci curvature, diameter and volume of $\mathcal{M}$.
\end{theo}

\begin{remark}
Let us comment on the additional assumptions in the Theorem above. The lower bound $\epsilon^{1/2} \geq C\left(\frac{\log N}{N}\right)^{\frac{2}{d^2+3d}}$ corresponds to the Nystr\"om interpolation error of the estimation of eigenfunctions, which dominates the error when the intrinsic dimension $d>9$ (see Remark~\ref{rem:error_analysis}). The upper bound {\color{black}$\epsilon^{1/2}\leq K^{-4}(-\nu_K)^{-\frac{d-1}{2}}$} is to ensure that the error of {\color{black}$H_{\epsilon,N}^K$} in estimating the leading {\color{black}$K$} terms of the heat kernel $H$ is uniformly bounded as {\color{black}$K\to \infty$}. Finally, the lower bound $t\geq {\color{black}\frac{8\log K}{cK^{2/d}}}$ is to ensure that the tail of infinite series (beyond the leading $K$ terms) of the heat kernel remains bounded as {\color{black}$K \to \infty$}.
\end{remark}

{\color{black}
\begin{remark}\label{remarkonKandN} 
Balancing the upper bound $K^{-4}(-\nu_K)^{-\frac{d-1}{2}}$ with the lower bound  in \eqref{validregime} and using the fact that $-\nu_K \propto K^{2/d}$ \cite{colbois2015eigenvalues}, we have 
\[
K^{-4-\frac{d-1}{d}} \propto \left( \frac{\log N}{N}\right)^s \qquad \Longleftrightarrow \qquad K \propto \left( \frac{N}{\log N}\right)^{\frac{s}{5-\frac{1}{d}}}.
\]
For $s = \min\left\{\frac{1}{4d+13},\frac{2}{d^2+3d}\right\}$, it is clear that $K$ grows slower than $N$ for any $d\geq 1$.
\end{remark}

\begin{remark}\label{negativeresult}
The convergence requires that $t\geq {\color{black}\frac{8\log K}{cK^{2/d}}}$ and the validity of \eqref{validregime}. for sufficiently large \(K\), assuming fixed \(d\ge1\) and \(c>0\). Using the fact that $-\nu_K \propto K^{2/d}$ \cite{colbois2015eigenvalues}, the second inequality in \eqref{validregime} suggests that $\epsilon\le K^{-10+2/d}$.
Thus, requiring \(\epsilon\ge t\) would force
\[
\frac{8\log K}{cK^{2/3}}
\le t\le\epsilon\le K^{-10+2/d},
\]
which requires
\[
\frac8c(\log K)\,K^{28/3-2/d}\le1.
\]
so the left-hand side tends to infinity. Therefore, the upper bound on \(\epsilon\) is eventually smaller than the lower bound on \(t\). That is, any admissible choices must satisfy \(\epsilon<t\) which implies the theorem does not apply for $t=\epsilon$. This also suggests that $k_{\epsilon,N} = H_{\epsilon,N}^N(\cdot,\cdot,\epsilon)$ does not necessarily converge to the heat kernel $H$. 
\end{remark}
}

The following two subsections~\ref{sec31}-\eqref{sec32} provide intermediate propositions that are needed to prove this theorem. We will report the proof of this Theorem in Section~\ref{sec33}.

\subsection{DM kernel approximates a Gaussian kernel} \label{sec31}In this subsection, we first show that under the uniform sampling assumption, the DM kernel approximates a scaled Gaussian kernel. 

First, we derive the relation between $q_\epsilon$ (and $\hat{q}_\epsilon$), both defined in \eqref{analyticdmalgebra}, and the underlying sampling density $q$. To carry this step, we recall the following asymptotic expansion \cite{coifman2006diffusion}: For $f\in C^3(\mathcal{M})$  and $\bx\in \mathcal{M}$, where $\mathcal{M}\subset \mathbb{R}^n$ is a $d$-dimensional smooth manifold,
\BEA
G_\epsilon f(\bx):=\epsilon^{-d/2}\int_{\mathcal M}\tilde k_\epsilon(\bx,\by)f(\by)\,d\mathrm{Vol}(\by)
=m_0 f(\bx)+\epsilon m_2\big(\omega(\bx)f(\bx)+\Delta_g f(\bx)\big)+O(\epsilon^2),\label{eq:Ge}
\EEA
where $m_0= \int_{\BR^d}h(z)dz , m_2=\frac{1}{2}\int_{\BR^d}z_1^2h(z)dz$ are constants with $h(z) = \exp\left(-\|z\|^2\right)$ with $z\in \BR^d$ and $\omega$ depends on the induced geometry of $\mathcal{M}$.

If we increase the regularity, $f\in C^4(\mathcal{M})$ and bounded geometry coefficients (in particular,
$\omega\in L^\infty(\mathcal M)$), the remainder can be taken uniformly in $x$ on
compact sets:
\BEA
\sup_{\bx\in\mathcal M}\left|G_\epsilon f(\bx)-m_0 f(\bx)-\epsilon m_2\big(\omega(\bx)f(\bx)+\Delta_g f(\bx)\big)\right|
=O(\epsilon^2).\label{uniform_DM_asymptotic_bound}
\EEA

Then it is immediately clear that
\[
\epsilon^{-d/2}q_\epsilon(\bx)=G_\epsilon q(\bx)= m_0q(\bx)+\epsilon m_2\big(\omega(\bx)q(\bx)+\Delta_g q(\bx)\big)+O(\epsilon^2),
\]
hence
\begin{eqnarray}
q_\epsilon(\bx)=\epsilon^{d/2}\left(m_0\,q(\bx)+O\!\left(\epsilon\right)\right).    \label{eq:qe_asymp}
\end{eqnarray}
Furthermore,
\[
\hat q_\epsilon(\bx)=\frac{1}{q_\epsilon(\bx)}\int_{\mathcal M}\tilde k_\epsilon(\bx,\by)\frac{q(\by)}{q_\epsilon(\by)}\,d\mathrm{Vol}(\by)
=\frac{\epsilon^{d/2} G_\epsilon\!\left(q/q_\epsilon\right)(\bx)}{q_\epsilon(\bx)},
\]
and (see Appendix~\ref{app:A} for details),
\[
\frac{q(\bx)}{q_\epsilon(\bx)}
=\epsilon^{-d/2}\left[\frac{1}{m_0}-\epsilon\frac{m_2}{m_0^2}\left(\omega(\bx)+\frac{\Delta_gq(\bx)}{q(\bx)}\right)+O(\epsilon^2)\right].
\]
Applying $G_\epsilon$ to this function and dividing by $q_\epsilon$ yields
\begin{eqnarray}
G_\epsilon(q/q_\epsilon)(\bx)
=\epsilon^{-d/2}
\left[1-\epsilon\frac{m_2}{m_0}\frac{\Delta_gq(\bx)}{q(\bx)}+O(\epsilon^2)\right] = \epsilon^{-d/2}\left( 1 - O\!\left(\epsilon\right)\right),
\end{eqnarray}
then
\begin{eqnarray}
\hat q_\epsilon(\bx)
= \frac{\epsilon^{d/2}G_\epsilon(q/q_\epsilon)(\bx)}{q_\epsilon(\bx)} = \epsilon^{-d/2}\left( \frac{1}{m_0q(\bx)}+O\!\left(\epsilon\right)\right).\label{eq:qhate_asymp} 
\end{eqnarray}

Next, we employ the standard KDE uniform estimate following \cite{gine2002rates}  to deduce errors between $q_{\epsilon,N}$ and $q_\epsilon$ as well as between $\hat{q}_{\epsilon,N}$ and $\hat{q}_\epsilon$.

\begin{lem}[Uniform concentration for $q_{\epsilon,N}$ and $\hat{q}_{\epsilon,N}$]\label{lem:3.1}
Assume $q\in C^2(\mathcal M)$, $q\ge q_{\min}>0$, and \eqref{eps_N}, that is, $\frac{\log N}{N \epsilon^{d/2}}\to 0$.

Then with probability higher than $1-\frac{1}{N}$,
\begin{eqnarray}
\sup_{\bx\in\mathcal M}\big|q_{\epsilon,N}(\bx)-q_\epsilon(\bx)\big|
=O\!\left(\epsilon^{d/2}\sqrt{\frac{\log N}{N\epsilon^{d/2}}}\right),\label{eq:error_q_en}\\
\sup_{\bx\in\mathcal M}\big|\hat q_{\epsilon,N}(\bx)-\hat q_\epsilon(\bx)\big|
= O\!\left(\epsilon^{-d/2}\sqrt{\frac{\log N}{N\epsilon^{d}}}\right). \label{eq:error_qhat_en}
\end{eqnarray}
\end{lem}

See Appendix~\ref{proof:lemma3.1} for the proof. With this lemma, we conclude the following result for this subsection:

\begin{prop}\label{prop:Gaussapprox}
Let the assumption~\ref{assumption1}, specifically, $\epsilon^{1/2}\gg h$, be valid. Also, let the sampling density $q:\mathcal{M}\to \BR$ be uniform, $q(\bx)=1(\bx)/\textup{Vol}(\mathcal{M}),$ for all $\bx\in \mathcal{M}$. With probability higher than $1-\frac{4}{N}$,
\BEA
\sup_{\bx,\by\in\mathcal M}
\left|k_{\epsilon,N}(\bx,\by)-\frac{\mathrm{Vol}(\mathcal M)}{m_0 \epsilon^{d/2}}\tilde{k}_{\epsilon}(\bx,\by)\right|
\le C\epsilon^{-d/2}\left(\epsilon+\sqrt{\frac{\log N}{N\epsilon^{d}}}\right),\label{eq:unifconv}
\EEA
where $C>0$ is independent of $N$ and $\epsilon$.
\end{prop}

\begin{proof}
Combining results in Lemma~\ref{lem:3.1} with \eqref{eq:qe_asymp} and \eqref{eq:qhate_asymp}, we conclude that w.p.h. $1-N^{-1}$, 
\BEA
q_{\epsilon,N}(\bx) &=& \epsilon^{d/2}\left(m_0q(\bx) + O(\epsilon) + O\left(\sqrt{\frac{\log N}{N \epsilon^{d/2}}}\right)\right), \\
\hat{q}_{\epsilon,N}(\bx) &=& \epsilon^{-d/2}\left(\frac{1}{m_0q(\bx)} + O(\epsilon) + O\left(\sqrt{\frac{\log N}{N \epsilon^{d}}}\right)\right).
\EEA
By the uniform sampling density assumption, $q(\bx)=1(\bx)/\textup{Vol}(\mathcal{M})$, with probability higher than $1-\frac{4}{N}$,
\[
d_{\epsilon,N}(\bx,\by):=\frac{1}{q_{\epsilon,N}(\bx)q_{\epsilon,N}(\by)\sqrt{\hat q_{\epsilon,N}(\bx)\hat q_{\epsilon,N}(\by)}}
=\frac{\mathrm{Vol}(\mathcal M)}{m_0 \epsilon^{d/2}}
\left[1+O(\epsilon)+O\!\left(\sqrt{\frac{\log N}{N\epsilon^{d}}}\right)\right],\]
which implies that the DM kernel is approximately a normalized Gaussian kernel,
\[
k_{\epsilon,N}(\bx,\by)
=\frac{\mathrm{Vol}(\mathcal M)}{m_0\epsilon^{d/2}}\tilde{k}_{\epsilon}(\bx,\by)
\left[1+O(\epsilon)+O\!\left(\sqrt{\frac{\log N}{N\epsilon^{d}}}\right)\right].
\]
The uniform bound is immediate due to \eqref{uniform_DM_asymptotic_bound}, \eqref{eq:error_q_en}, and \eqref{eq:error_qhat_en}.
\end{proof}

\begin{remark}\label{rem:smoothbound}
Here, we discuss the continuity and boundedness of the DM kernel.
\begin{enumerate}
\item[(a)] The DM kernel is smooth, $k_{\epsilon,N}\in C^\infty(\mathcal{M}\times\mathcal{M})$ on a compact manifold $\mathcal{M}$. This is due to the following reason. First, we note that the Gaussian kernel  $\tilde{k}_\epsilon \in C^\infty(\mathcal{M}\times\mathcal{M})$. Its normalization factor $q_{\epsilon,N}>0$ is also smooth since it is a finite sum of a smooth kernel that is also strictly positive on a compact manifold $\mathcal{M}$. Following the same argument for the second normalization, $\hat{q}_{\epsilon,N}$, the DM kernel is smooth on $\mathcal{M}$. In fact, $d_{\epsilon,N}$ in the proof above gives an estimate of a lower bound for the normalization factor of the DM kernel for uniform measure.
\item[(b)] By Proposition 3.1 and the fact that the Gaussian kernel $\tilde{k}_\epsilon \leq 1$, the DM kernel is uniformly bounded,
\end{enumerate}
\end{remark}

\subsection{Convergence of the orthonormal basis of the RKHS  $\mathcal{H}_{\epsilon,N}$}\label{sec32}

In this subsection, we deduce the convergence of the orthonormal basis for $\mathcal{H}_{\epsilon,N}$ to the corresponding orthonormal basis for the RKHS induced by the heat kernel.

Before we prove such a result, we deduce the local Lipschitz bound for the DM kernel. 

\begin{lem}\label{lem:localerror_k}
    Let the assumption in Proposition~\ref{prop:Gaussapprox} to be valid. For any $\bx,\bx'\in B_{h}(\by)\in \mathcal{M}$ and any fixed $ \by\in \mathcal{M}$, where $B_h(\by)$ denotes the geodesic ball with center at $\by$ and radius $h$,  with probability higher than  $1-\frac{5}{N}$,
\[
\left|  k_{\epsilon,N}(\bx,\by)-  k_{\epsilon,N}(\bx',\by)  \right| = O\left(\left(\frac{\log N}{N}\right)^{\frac{2}{d}}\frac{1}{\epsilon^{1+\frac{d}{2}}}\right) .
\]
This bound includes $\bx =\by$ or $\bx'=\by$.
\end{lem}

See Appendix~\ref{proof:lemma3.2} for the proof. With this result, we conclude this section as follows:

\comment{\color{red}
If we apply Lemma~\ref{lem:localerror_k} with $x,x'\in\mathcal M$ such that
$d_g(x,x')\le h$, once with the second
argument fixed at $x$ and once with the second argument fixed at $x'$, we obtain
\[
|k_{\epsilon,N}(x,x)-k_{\epsilon,N}(x',x)|
\le C\epsilon^{-(3d/4+1)}h^2,\qquad |k_{\epsilon,N}(x',x')-k_{\epsilon,N}(x,x')|
\le C\epsilon^{-(3d/4+1)}h^2.
\]
Since $k_{\epsilon,N}$ is symmetric, $k_{\epsilon,N}(x',x)=k_{\epsilon,N}(x,x')$,
and therefore
\[
\left|k_{\epsilon,N}(x,x)+k_{\epsilon,N}(x',x')
-2k_{\epsilon,N}(x,x')\right|
\le C\epsilon^{-(3d/4+1)}h^2.
\]
This justifies a diagonal estimate used for
$\|k_{\epsilon,N}(\cdot,x)-k_{\epsilon,N}(\cdot,x')\|_{\mathcal H_{\epsilon,N}}$ that we will use later.
}

\begin{prop}\label{prop:err_eigfun}
Let $\psi^{N,\epsilon}_j$ be defined as in \eqref{RKHS} and $\nu_j<0$ and $\varphi_j$ denote the $j$th eigenvalue and eigenfunction of the Laplace-Beltrami operator on $\mathcal{M}$. Under the Assumption~\ref{assumption1} and the assumptions in Lemma~\ref{lem:localerror_k},
for all $\bx\in \mathcal{M}$ and $t>0$, with probability higher than $1- \frac{11}{N}-\frac{2}{N^2}$,
\[
\left|e^{\frac{t}{2}\nu_j^{\epsilon,N}}\psi^{\epsilon,N}_j(\bx) - e^{\frac{t}{2}\nu_j}\varphi_j(\bx) \right| \leq  e^{\frac{t}{2}\nu_1 }\left( C_1 \frac{1}{\epsilon^{1/2+ d/4}}\left(\frac{\log N}{N} \right)^{\frac{1}{d}} + C_2 \epsilon^{\frac{1}{4}}+ C_3\frac{1}{\epsilon^{d+3}}\left(\frac{\log N}{N}\right)^{\frac{1}{2}} + C_4 \left(\frac{\log N}{N} \right)^{\frac{1}{d}}\right),
\] 
and $C_1,C_2, C_3, C_4>0$ are independent of $\epsilon$ and $N$. Here, $\left\{e^{\frac{\epsilon}{2}\nu_j^{\epsilon,N}}\psi^{\epsilon,N}_j\right\}$ forms an orthonormal basis of $\mathcal{H}_{\epsilon,N}$.\end{prop}

\begin{proof}
We first note that  for every $\bx_i \in B_h(\bx) \subset \mathcal{M}$ 
where $h$ is the fill distance as defined in Lemma~\ref{samplingassumption},
\BEA
\left|e^{\frac{t}{2}\nu_j^{\epsilon,N}}\psi^{\epsilon,N}_j(\bx) - e^{\frac{t}{2}\nu_j}\varphi_j(\bx) \right| &\leq& \underbrace{\left|e^{\frac{t}{2}\nu_j^{\epsilon,N}}\left(\psi^{\epsilon,N}_j(\bx) - \psi^{\epsilon,N}_j(\bx_i) \right)\right|}_{\text{(I)}} + \underbrace{\left|e^{\frac{t}{2}\nu_j^{\epsilon,N}}\psi^{\epsilon,N}_j(\bx_i) -e^{\frac{t}{2}\nu_j} \varphi_j(\bx_i) \right|}_{\text{(II)}}\notag \\ && + \underbrace{\left|e^{\frac{t}{2}\nu_j}\left(\varphi_j(\bx_i) - \varphi_j(\bx)\right) \right|  }_{\text{(III)}}
\EEA 
First, we can bound term (II) immediately based on the result \eqref{eiv_rate2} (using \eqref{eiv_rate} in Lemma~\ref{lem:spectral},  \eqref{eq:error_qhat_en}, and uniformly sampling assumption). That is, writing
\BEA
\textup{(II)} \leq e^{\frac{t}{2}\nu_j^{\epsilon,N}}\left| \psi^{\epsilon,N}_j(\bx_i) - \varphi_j(\bx_i) \right| + |e^{\frac{t}{2}\nu_j^{\epsilon,N}} - e^{\frac{t}{2}\nu_j}| \left| \varphi_j(\bx_i) \right|.
\EEA
With probability higher than $1-N^{-2}$,
\begin{equation}
    \left|\psi^{\epsilon,N}_j(\bx_i) - \varphi_j(\bx_i) \right| = \left|\varphi^{\epsilon,N}_j(\bx_i) - \varphi_j(\bx_i) \right| \leq C_2 \epsilon^{1/4} + C_3\frac{1}{\epsilon^{d+3}}\left(\frac{\log N}{N}\right)^{\frac{1}{2}} ,\label{convrate_II}
\end{equation}
for some constants $C_2,C_3>0$ that are independent of $N$ and $\epsilon$ as $N\to \infty$, which also satisfied the inequality in \eqref{eps_N}. As for the spectrum, with probability higher than $1-\frac{1}{N^2}$, for $j>1$,
\BEA
\left|e^{\frac{t}{2}\nu_j^{\epsilon,N}} - e^{\nu_j\frac{t}{2}}\right| \leq \hat{C}\epsilon^{3/4} t e^{\frac{t}{2}\nu_1}.\label{eq:error_eigoperator}
\EEA
Since, $e^{\frac{t}{2}\nu_j^{\epsilon,N}} \leq \left|e^{\frac{t}{2}\nu_j^{\epsilon,N}} - e^{\frac{t}{2}\nu_j}\right| +  e^{\frac{t}{2}\nu_j} $, it is clear that,  $e^{\frac{t}{2}\nu_j^{\epsilon,N}}  \leq  e^{\frac{t}{2}\nu_1}\left(1 + \hat{C}\epsilon^{3/4} t \right)$. Since the eigenvector error rate is slower than that of the eigenvalues, we have
\[
\textup{(II)} \leq \left( C_2 \epsilon^{1/4} + C_3\frac{1}{\epsilon^{d+3}}\left(\frac{\log N}{N}\right)^{\frac{1}{2}} \right) e^{\frac{t}{2}\nu_1 }\left(1 + \hat{C}\epsilon^{3/4} t  \right)  .
\]


Term (III) is also immediate from the smoothness of eigenfunctions of the Laplace-Beltrami operator, 
\[
e^{\frac{t}{2}\nu_j} \left| \varphi_j(\bx_i) -\varphi_j(\bx) \right| \leq  C_4 e^{\frac{t}{2}\nu_1} h,
\]
where the constant $C_4=\|\nabla\varphi_j\|_{C(B_h(\bx))}$. Here, $h \to 0$ as $N\to \infty$ with a rate in \eqref{filldistance} which is faster than \eqref{convrate_II}.

\comment{\color{red}Next, let us bound the term (I). From the definition in \eqref{Nyst_eigfun}, one can write,
\BEA
\left|\psi^{\epsilon,N}_j(\bx) - \psi_j^{\epsilon,N}(\bx_i) \right| &\leq& \frac{1}{N\lambda^{\epsilon,N}_j} \sum_{\ell=1}^N  \left| \left(k_{\epsilon,N}(\bx,\bx_\ell) -k_{\epsilon,N}(\bx_i,\bx_\ell)\right) \varphi_j^{\epsilon,N}(\bx_\ell)\right| \notag \\
&\leq& \frac{1}{\lambda^{\epsilon,N}_j} \left(\frac{1}{N}\sum_{\ell=1}^N  \left| k_{\epsilon,N}(\bx,\bx_\ell) -k_{\epsilon,N}(\bx_i,\bx_\ell)\right|^2 \right)^{1/2} \underbrace{\left(\frac{1}{N}\sum_{\ell=1}^N  \left| \varphi^{\epsilon,N}_j(\bx_\ell)\right|^2 \right)^{1/2}.}_{=1} \notag
\EEA
}

Now we bound the term (I). Let us denote $\Psi^{\epsilon,N}_j=\sqrt{\lambda_j^{\epsilon,N}}\psi^{\epsilon,N}_j$, which is a feature of the RKHS $\mathcal{H}_{\epsilon,N}$ that satisfies, \BEA\left\langle\Psi^{\epsilon,N}_j,\Psi^{\epsilon,N}_i\right\rangle_{\mathcal{H}_{\epsilon,N}} &=& \left(\lambda_j^{\epsilon,N}\lambda_i^{\epsilon,N}\right)^{1/2}\left\langle\psi^{\epsilon,N}_j,\psi^{\epsilon,N}_i\right\rangle_{\mathcal{H}_{\epsilon,N}} = \frac{\left\langle S_{\epsilon,N}\varphi^{\epsilon,N}_j,S_{\epsilon,N}\varphi^{\epsilon,N}_i\right\rangle_{\mathcal{H}_{\epsilon,N}}}{\left(\lambda_j^{\epsilon,N}\lambda_i^{\epsilon,N}\right)^{1/2}} \notag \\
&=& \frac{\left\langle S_{\epsilon,N}^*S_{\epsilon,N}\varphi^{\epsilon,N}_j,\varphi^{\epsilon,N}_i\right\rangle_{L^2(\mu_N)}}{\left(\lambda_j^{\epsilon,N}\lambda_i^{\epsilon,N}\right)^{1/2}}  = \frac{\left(\lambda_j^{\epsilon,N}\right)^{1/2}}{\left(\lambda_i^{\epsilon,N}\right)^{1/2}} \left\langle \varphi^{\epsilon,N}_j,\varphi^{\epsilon,N}_i\right\rangle_{L^2(\mu_N)}=  \delta_{ij}.\notag
\EEA 

We note that, 
\[
e^{\frac{t}{2}\nu_j^{\epsilon,N}} = e^{\frac{(t-\epsilon)}{2}\nu_j^{\epsilon,N}}e^{\frac{\epsilon}{2}\nu_j^{\epsilon,N}} = e^{\frac{(t-\epsilon)}{2}\nu_j^{\epsilon,N}}\sqrt{\lambda^{\epsilon,N}_j}, 
\]
so, 
\[
e^{\frac{t}{2}\nu_j^{\epsilon,N}}\left(\psi_j^{\epsilon,N}(\bx) - \psi_j^{\epsilon,N}(\bx_i)\right) = e^{\frac{(t-\epsilon)}{2}\nu_j^{\epsilon,N}}\left(\Psi_j^{\epsilon,N}(\bx) - \Psi_j^{\epsilon,N}(\bx_i)\right).
\]
By the reproducing property and orthonormality of $\Psi_j^{\epsilon,N}$,
\BEA
\left|\Psi_j^{\epsilon,N}(\bx) -\Psi_j^{\epsilon,N}(\bx_i)\right|^2 &=& \left|\left\langle \Psi_j^{\epsilon,N},k_{\epsilon,N}(\cdot,\bx)-k_{\epsilon,N}(\cdot,\bx_i) \right\rangle_{\mathcal{H}_{\epsilon,N}} \right|^2\notag \\
&\leq& \left\| \Psi_j^{\epsilon,N}\right\|_{\mathcal{H}_{\epsilon,N}}^2\left\| k_{\epsilon,N}(\cdot,\bx)-k_{\epsilon,N}(\cdot,\bx_i)\right\|_{\mathcal{H}_{\epsilon,N}}^2\notag \\
&=& \left| k_{\epsilon,N}(\bx,\bx) + k_{\epsilon,N}(\bx_i,\bx_i)- 2k_{\epsilon,N}(\bx,\bx_i) \right|
\notag \\
&\leq&  \left| k_{\epsilon,N}(\bx,\bx)- k_{\epsilon,N}(\bx,\bx_i)\right| +\left| k_{\epsilon,N}(\bx_i,\bx_i)- k_{\epsilon,N}(\bx,\bx_i) \right|.
\notag 
\EEA
We conclude that,
\BEA
\text{(I)} \leq e^{\frac{(t-\epsilon)}{2}\nu_j^{\epsilon,N}} \left|\Psi_j^{\epsilon,N}(\bx) -\Psi_j^{\epsilon,N}(\bx_i)\right| \leq e^{\frac{t}{2}\nu_1 }\left(1 + \hat{C}\epsilon^{3/4} t \right) \underbrace{e^{-\frac{\epsilon}{2}\nu_{j}^{\epsilon,N}}}_{1-\frac{1}{2}\nu_j\epsilon + O(\epsilon^{7/4})} \frac{C_1}{\epsilon^{1/2+ d/4}}\left(\frac{\log N}{N} \right)^{\frac{1}{d}},
\EEA
using Lemma~\ref{lem:localerror_k}.  Counting all of the probability,  wph $1-\frac{11}{N}-\frac{2}{N^2}$,
\[
\text{(I)+(II)+(III)} \leq e^{\frac{t}{2}\nu_1}  \left(C_1 \frac{1}{\epsilon^{1/2+ d/4}}\left(\frac{\log N}{N} \right)^{\frac{1}{d}} +C_2 \epsilon^{1/4} + C_3\frac{1}{\epsilon^{d+3}}\left(\frac{\log N}{N}\right)^{\frac{1}{2}} + C_4 \left(\frac{\log N}{N} \right)^{\frac{1}{d}} \right),
\]
for $j\geq 1$. For $j=0$, the same rate also apply without dependent on $t$ because $\nu_0 = \nu_0^{\epsilon,N} = 0$. 
\end{proof}

\begin{remark}\label{rem:error_analysis}
  For any fixed $t>0$, balancing $C_1$ and $C_2$ yields $\epsilon^{1/4} = O\left(\frac{\log N}{N}\right)^{\frac{1}{d(d+3)}}$. 
Balancing $C_2$ and $C_3$ yields the convergence rate of the eigenvector, that is, $\left(\frac{\log N}{N}\right)^{\frac{1}{8d+26}}$.
Since 
\[
\left(\frac{\log N}{N}\right)^{\frac{1}{d}} \ll \left(\frac{\log N}{N}\right)^{\frac{1}{d(d+3)}} \ll \left(\frac{\log N}{N}\right)^{\frac{1}{8d+26}},
\]
for $d<9$, the overall error rate is $\epsilon^{1/4}=O\left(\frac{\log N}{N}\right)^{\frac{1}{8d+26}}$ for $d<9$, which is the error bound of the eigenvector estimation in Lemma~\ref{lem:spectral}. On the other hand, if $d\geq 9$, the overall error rate is even slower,
$\epsilon^{1/4}=O\left(\frac{\log N}{N}\right)^{\frac{1}{d(d+3)}}$.  
\end{remark}

\subsection{Proof of Theorem~\ref{thm1}}\label{sec33}

In this section, we prove Theorem~\ref{thm1}. Namely, we deduce that the limit of the following sequence,
\[
\sup_{\bx,\by\in\mathcal M}
\left|H_{\epsilon,N}^{\color{black}K}(\bx,\by;t)-H(\bx,\by;t)\right|
\]
vanishes, along an admissible sequence $N\to\infty$, $\epsilon\to 0$, where $H$ and $H_{\epsilon,N}$ are defined in \eqref{eq:heatkernel} and \eqref{eq:empiricalheatkernel}, respectively. 

Decompose
\[
\begin{aligned}
H_{\epsilon,N}^{\color{black}K}(\bx,\by;t)-H(\bx,\by;t)
&=
\underbrace{\sum_{j=0}^{\color{black}K-1}\left[
e^{t\nu_j^{\epsilon,N}}
\psi_j^{\epsilon,N}(\bx)\psi_j^{\epsilon,N}(\by)
-e^{t\nu_j}\varphi_j(\bx)\varphi_j(\by)
\right]}_{\color{black}=:A_{K,\epsilon,t}(\bx,\by)} 
-\underbrace{\sum_{\color{black}j=K}^{\infty}e^{t\nu_j}\varphi_j(\bx)\varphi_j(\by)}_{\color{black}=:R_{K,t}(\bx,\by)}.
\end{aligned}
\]

We will employ Proposition~\ref{prop:err_eigfun} to bound {\color{black}$|A_{K,\epsilon,t}(\bx,\by)|$} and follow the proof of Theorem~3 in \cite{dunson2021spectral} to bound the remainder term, {\color{black}$|R_{K,t}(\bx,\by)|$.}

For $j=1,\ldots, {\color{black}K-1}$, using Proposition~\ref{prop:err_eigfun}, each term in  {\color{black}$A_{K,\epsilon,t}$} can be bounded as follows,
\BEA
\left|e^{t\nu_j^{\epsilon,N}}
\psi_j^{\epsilon,N}(\bx)\psi_j^{\epsilon,N}(\by)
-e^{t\nu_j}\varphi_j(\bx)\varphi_j(\by)\right|\leq C_2e^{t\frac{\nu_1}{2}}\epsilon^{1/4}\left(
\sup_{x\in \mathcal{M}} \left|e^{\frac{t}{2}\nu_j^{\epsilon,N}}
\psi_j^{\epsilon,N}(\bx) \right|
+\sup_{x\in \mathcal{M}} \left|e^{\frac{t}{2}\nu_j}\varphi_j(\bx)\right|\right).\label{eq:eq11}
\EEA
Using also \eqref{eq:error_eigoperator}, the right hand side can be bounded as follows:
\BEA
e^{\frac{t}{2}\nu_j^{\epsilon,N}}\sup_{x\in \mathcal{M}} \left|
\psi_j^{\epsilon,N}(\bx) \right|
+e^{\frac{t}{2}\nu_j}\sup_{x\in \mathcal{M}} \left|\varphi_j(\bx)\right| &\leq&
\left(e^{\frac{t}{2}\nu_j^{\epsilon,N}} - e^{\frac{t}{2}\nu_j} \right)\|\psi_j^{\epsilon,N}\|_{\infty} +e^{\frac{t}{2}\nu_j} \left(\|\psi_j^{\epsilon,N}\|_{\infty} + \|\varphi_j\|_{\infty}\right) \notag \\
&\leq& \left(e^{\frac{t}{2}\nu_j^{\epsilon,N}} - e^{\frac{t}{2}\nu_j} \right) \left(\|\varphi_j\|_{\infty} + C_2\epsilon^{1/4}\right) +e^{\frac{t}{2}\nu_1} \left(2\|\varphi_j\|_{\infty}+C_2\epsilon^{1/4}\right) \notag \\
&\leq& (\hat{C}\epsilon^{3/4} t +1) e^{\frac{t}{2}\nu_1} \left(2\|\varphi_j\|_{\infty}+C_2\epsilon^{1/4}\right) \notag \\
&\leq & \tilde{C} (-\nu_j)^{\frac{d-1}{4}} (\hat{C}\epsilon^{3/4} t +1) e^{\frac{t}{2}\nu_1},\notag
\EEA
where we used a uniform bound for eigenfunction of the Laplace-Beltrami operator (see Lemma~1 in \cite{dunson2021spectral}).
Inserting this bound to \eqref{eq:eq11}, 
\[\left|e^{t\nu_j^{\epsilon,N}}
\psi_j^{\epsilon,N}(\bx)\psi_j^{\epsilon,N}(\by)
-e^{t\nu_j}\varphi_j(\bx)\varphi_j(\by)\right| \leq \tilde{C} (-\nu_j)^{\frac{d-1}{4}}e^{t\nu_1}\epsilon^{1/4} (\hat{C}t\epsilon^{3/4} + 1).
\]
For $j=0$, the bound
\[
\left|
\psi_0^{\epsilon,N}(\bx)\psi_0^{\epsilon,N}(\by)
-\varphi_0(\bx)\varphi_0(\by)\right| \leq C_2 \epsilon^{1/4}.
\]
Therefore,
\BEA
|A_{\color{black}K,\epsilon,t}(\bx,\by)| &\leq& \tilde{C} ( {\color{black}-\nu_K})^{\frac{d-1}{4}} \epsilon^{1/4}  {\color{black}K} e^{t\nu_1} (\hat{C}t\epsilon^{3/4} + 1) + C_2\epsilon^{1/4} \notag \\
&\leq& C_1  {\color{black}K} ( {\color{black}-\nu_K})^{\frac{d-1}{4}}   \epsilon^{1/4} (t + 1)e^{t\nu_1} + C_2\epsilon^{1/4} \notag \\
&\leq & \frac{C_1}{ \color{black}K}(t + 1)e^{t\nu_1} + C_2 \epsilon^{1/4},\notag
\EEA
where we have used the assumption 
$
\epsilon^{1/2} \leq  {\color{black}\frac{1}{K^4 (-\nu_K)^{\frac{d-1}{2}}} }
$
and define $C_1 = \tilde{C}\hat{C}>0$.

Following the argument in Section 6 of \cite{dunson2021spectral}, we can bound the reminder.
\[
|R_{\color{black}K,t}(\bx,\by)| \leq C_3 \frac{\color{black}K}{2}A(t) e^{-A(t)}, 
\]
where $C_3$ depends on dimension $d$, Ricci curvature, and $\text{diam}(\mathcal{M})$, and $A(t) = c  {\color{black}K}^{2/d}t$. Take $t \geq  {\color{black}\frac{8\log K}{cK^{2/d}}}$, then,
\[
\frac{A(t)}{2} - \log \frac{A(t)}{2} \geq \frac{A(t)}{4} \geq 2\log  {\color{black}K}, 
\]
which is analogous to $ {\color{black}\frac{K}{2}A(t) \leq \frac{1}{K}e^{\frac{A(t)}{2}}}$. Thus,
\[
|R_{\color{black}K,t}(\bx,\by)| \leq C_3\frac{1}{K}e^{-\frac{A(t)}{2}} = C_3  {\color{black}\frac{1}{ \color{black}K}e^{-\frac{cK^{2/d}t}{2}}}. 
\]

\section{Kernel ridge regression analysis}\label{sec4}

In this section, our goal is to derive an error bound for kernel ridge regression using the diffusion maps kernel, $k_{\epsilon,N}$, where this data-driven kernel is constructed under fixed $\epsilon$ and $N$.
To this end, let us provide an overview of the oracle inequality for supervised learning with kernel ridge regression in \cite{steinwart2009optimal}. Since their result is valid on any compact metric input space $\mathcal{X}$ and bounded kernel $k:\mathcal{X}\times \mathcal{X}\to \BR$, let us rewrite their result in our notation with a $d-$dimensional closed (compact with no boundary) manifold input space, $\mathcal{X}=\mathcal{M}$, and $k$ will be a bounded kernel of $\mathcal{M}$. 

Let $(\bx_1,y_1), \ldots, (\bx_n,y_n)$ sampled from $P$, an unknown distribution on $\mathcal{M}\times [-M,M]$ where $M>0$ is some clipping constant and our goal is to find a function $f:\mathcal{M} \to [-M,M]\subset \BR$ that minimizes,
\[
\mathcal{R}(f)=\int_{\mathcal{M}\times[-M,M]} (y -f(\bx))^2 dP(\bx,y). 
\]
We denote the regression solution as $f^*$ and $\mathcal{R}^*= \mathcal{R}(f^*)$. We also denote the KRR solution as,
\[
f_{\eta} = \arg \min_{f\in \mathcal{H}} \frac{1}{N} \sum_{i=1}^N (y_i -f(\bx_i))^2 + \eta \|f\|^2_{\mathcal{H}},
\]
where $\eta$ denotes a regularization parameter. We also define the clipping function,
\[\hat{f}_{\eta}(\bx) = \begin{cases} -M, & \text{ if } \hat{f}_{\eta}(\bx)< -M \\ 
f_{\eta}(\bx), & \text{ if } \hat{f}_{\eta}(\bx)\in [-M,M]\\
M, & \text{ if } \hat{f}_{\eta}(\bx)> M.
\end{cases}
\]
For the discussion below, we define approximation error,
\[
A_2(\eta) = \inf_{f\in\mathcal{H}}\left(  \mathcal{R}(f) - \mathcal{R}^* + \eta \|f\|^2_{\mathcal{H}}\right) =  \inf_{f\in\mathcal{H}}\left(\| f-f^*\|^2_{L^2(\mu)} + \eta \|f\|^2_{\mathcal{H}}\right).
\]

The right hand term above is sometimes known as the K-functional $K(f^*,\eta)$. More generally, define two Banach spaces $F \subset E$ with continuous embedding $\text{id}:F\to E$, we defined the K-functional of $y\in E$ as, 
\[
K(y,t) = \inf_{x\in F} \left(\| x-y\|^2_{E} + t \|x\|^2_{F}\right).
\]
With this definition, we define the interpolation space
$[E,F]_{\theta,r}$ for $0<\theta<1$, $1\leq r \leq \infty$ as a Banach space of functions  $u\in E$ with bounded norm defined as follows,
\[
\|u\|_{\theta,r} =\begin{cases} 
\left( \int_0^\infty (t^{-\theta} K(x,t))^r t^{-1} dt \right)^{1/r} & \text{ if } r<\infty, \\
\sup_{t>0} t^{-\theta} K(x,t) & \text{ if } r=\infty. 
\end{cases}
\]

With this background, we state the following result:

\begin{prop}[Corollary 3 of \cite{steinwart2009optimal} adapted to our notation]\label{prop:krr_errorbound}
Let $k$ be a bounded measureable kernel on $\mathcal{M}$ and a separable RKHS $\mathcal{H}$.  Assume that the marginal $\mu \propto  1$ is uniform on $\mathcal{M}$ and the integral operator $T_k:L^2(\mathcal{M}) \to L^2(\mathcal{M})$,
\[
T_k f = \int_{\mathcal{M}} k(\cdot,\by) f(\by) d\textup{Vol}(\by), \quad \forall f\in L^2(\mathcal{M}),
\]
has eigenvalues that satisfy
$
\lambda_j(T_k) \leq a j^{-\frac{1}{p}}$, where $a\geq 16M^4$ and $p\in (0,1)$. Assume also,
\BEA
\|f\|_{\infty} \leq C \|f\|_{\mathcal{H}}^p \|f\|_{L^2(\mathcal{M})}^{1-p}, \quad \forall f\in \mathcal{H}.\label{embeddingcond}
\EEA
Then there exists a constant $c=c(p,C)$ such that for all $\eta\in (0,1]$, $\tau>0$, with probability higher than $1-3e^{-\tau}$,
\begin{eqnarray}
\mathcal{R}(\hat{f}_{\lambda}) - \mathcal{R}(f^*) \leq 9 A_2(\eta) + c \frac{a^p M^2\tau}{\eta^pN}.\label{eq:oracle}
\end{eqnarray}
\end{prop}

In \cite{steinwart2009optimal}, the authors deduce an optimal rate with an appropriate choice of regularization parameter $\eta$ based on this proposition and an additional assumption on the approximation error $A_2(\eta)$. In fact, they also argue that the same optimal rate can be achieved when $\eta$ is chosen through training validation, which is the standard practical procedure. To employ this result, we need to verify the conditions in Proposition~\ref{prop:krr_errorbound} for the DM kernel. While the boundedness and spectra decay conditions are intuitive, let us first explain what inequality in \eqref{embeddingcond} means. 

\begin{remark}
The condition in \eqref{embeddingcond} can be explained as follows:
\begin{enumerate}
\item[a)] $[L^2(\mathcal{M}),\mathcal{H}]_{p,1}\hookrightarrow L^\infty(\mathcal{M})$ if and only if \eqref{embeddingcond} is satisfied (see Proposition~2.10 in \cite{bennett1988interpolation}).
\item[b)] If $\mathcal{H}$ is an RKHS with equivalent norm to Sobolev space $H^m(\mathcal{M})$, then the interpolation space is a Besov space $[L^2(\mathcal{M}),H^m(\mathcal{M})]:= B^{pm}_{2,1}(\mathcal{M})$ that is continuously embedded in $L^\infty(\mathcal{M})$ if $p = \frac{d}{2m}$ (see Theorem 7.34 in \cite{adams2003sobolev}). This means that for the Mat\'ern kernel, $k_M$, whose RKHS norm is equivalent to the Sobolev norm $H^m(\mathcal{M})$, the inequality in \eqref{embeddingcond} is satisfied. For convenience of the discussion below, we denote the RKHS corresponding to the Mat\'ern kernel as $\mathcal{H}_{k_M}$. We will specify $k_M$ in the following discussion and the required regularity order $m$ in the analysis below.
\item[c)] Finally, any function in an RKHS that is continuously embedded in $\mathcal{H}_{k_M}$ satisfies inequality in \eqref{embeddingcond}. This suggests that the error rate in the proposition above holds for the Gaussian kernel $\tilde{k}_\epsilon$ if the RKHS $\mathcal{H}_{\tilde{k}_\epsilon}$ is continuously embedded in $\mathcal{H}_{k_M}$, which is what we plan to show in the following.
\end{enumerate}
\end{remark}

Let us set up the notations with the following definition.

\begin{defn}\label{def:kernels}
Define also a kernel on $\BR^n$,
\[
\phi:\BR^n\times \BR^n \to \BR, \quad \phi|_{\mathcal{M}\times \mathcal{M}} = k,
\]
such that $k:\mathcal{M}\times\mathcal{M}\to \BR$ is the restriction of $\phi$ on the manifold $\mathcal{M}$. Particular interests are the extensions of the Gaussian and Mat\'ern kernels,
\[ \tilde{\phi}_{\epsilon}|_{\mathcal{M}\times \mathcal{M}} = \tilde{k}_{\epsilon}, \qquad \phi_{M}|_{\mathcal{M}\times \mathcal{M}} = k_{M},
\]
respectively. The associated RKHSs will be denoted as:
\BEA
\mathcal{H}_{\tilde{\phi}_{\epsilon}}\,&:& \text{ corresponds to the RKHS of the Gaussian kernel on } \mathbb{R}^n, \notag\\
\mathcal{H}_{\tilde{k}_{\epsilon}}\,&:& \text{ corresponds to the RKHS of the Gaussian kernel on $\mathcal{M}$}, \notag\\
\mathcal{H}_{\phi_{M}}\,&:& \text{ corresponds to the RKHS of Mat\'ern kernel on $\BR^n$ with norm equivalent to $H^s(\BR^n)$ with $s>\frac{n-d}{2}$}, \notag\\
\mathcal{H}_{k_{M}}\,&:& \text{ corresponds to the RKHS of Mat\'ern kernel on $\mathcal{M}$}. \notag
\EEA
\end{defn}

With the definition above, our plan is as follows: Prove that $\mathcal{H}_{\epsilon,N}$ and $\mathcal{H}_{\tilde{k}_\epsilon}$ are isometrically isomorphic. Since the RKHS induced by Gaussian kernel is continuously embedded in RKHS induced by any Mat\'ern kernels in $\BR^n$, that is, $ \mathcal{H}_{\tilde{\phi}_\epsilon} \hookrightarrow \mathcal{H}_{\phi_M}$, we will subsequently show that this continuous embedding is still valid when the kernels are restricted on the sub-manifold $\mathcal{M}$, that is,  $\mathcal{H}_{\tilde{k}_{\epsilon}} \hookrightarrow \mathcal{H}_{k_M}$. With these results, we will show that the inequality in \eqref{embeddingcond} is valid for $f\in \mathcal{H}_{\epsilon,N}$. Together with verifing all other assumptions, DMKRR enjoys the oracle inequality in Proposition~\ref{prop:krr_errorbound}.

\begin{lem}\label{lem:isoembedding} Assume that 
\[
\phi(\bx) = \frac{1}{q_{\epsilon,N}(\bx)\sqrt{\hat{q}_{\epsilon,N}(\bx)}} \neq 0,
\]
for all $\bx \in \mathcal{M}$. Then the RKHS induced by the DM kernel is symmetric positive definite. Also, the RKHS $\mathcal{H}_{\epsilon,N}=\left\{Uf := \phi f:\ f\in \mathcal H_{\tilde{k}_\epsilon}\right\}$ induced by the DM kernel, $k_{\epsilon,N}$, is isometrically isomorphic to the RKHS $\mathcal{H}_{\tilde{k}_\epsilon}$ induced by the Gaussian kernel, $\tilde{k}_\epsilon$. That is, $\| \phi f\|_{\mathcal H_{\epsilon,N}}=\|f\|_{\mathcal H_{\tilde{k}_\epsilon}}$.
\end{lem}

\begin{proof}
    By definition, we can write
\BEA
k_{\epsilon,N}(\bx,\by)= \phi(\bx)\phi(\by) \tilde{k}_\epsilon(\bx,\by), \label{DM_Gauss}
\EEA
where $\phi \neq 0$. 

For any \(a_1,\dots,a_m\in\mathbb R\),
\[
\sum_{i,j}a_i a_j k_{\epsilon,N}(\bx_i,\bx_j)
=
\sum_{i,j}(a_i\phi(\bx_i))(a_j\phi(\bx_j))\tilde{k}_\epsilon(\bx_i,\bx_j)\ge0,
\]
since \(\tilde{k}_\epsilon\) is PSD. So \(k_{\epsilon,N}\) is PSD. Next, let us characterize $\mathcal{H}_{\epsilon,N}$ in terms of the RKHS $\mathcal{H}_{\tilde{k}_\epsilon}$ induced by $\tilde{k}_\epsilon$.  Let
\[
\mathcal S_{\tilde{k}_\epsilon}:=\operatorname{span}\{\tilde{k}_\epsilon(\cdot,\bx):\bx\in\mathcal M\},
\quad
\mathcal S_{\epsilon,N}:=\operatorname{span}\{k_{\epsilon,N}(\cdot,\bx):\bx\in\mathcal M\}.
\]
Define \(U:\mathcal S_{\tilde{k}_\epsilon}\to \mathcal S_{\epsilon,N}\) by
\[
(Uf)(\bx):=\phi(\bx)f(\bx),
\]
for all $\bx\in \mathcal{M}$.
If
$
f=\sum_{i=1}^N \alpha_i \tilde{k}_\epsilon(\cdot,\bx_i)$, then
\[
(Uf)(\bx)=\phi(\bx)\sum_{i=1}^N\alpha_i \tilde{k}_\epsilon(\bx,\bx_i)
=\sum_{i=1}^N \frac{\alpha_i}{\phi(\bx_i)}\,k_{\epsilon,N}(\bx,\bx_i),
\]
using the identity in \eqref{DM_Gauss}.

Now take
$
f=\sum_{i=1}^N \alpha_i \tilde{k}_\epsilon(\cdot,\bx_i)$ and
$g=\sum_{i=1}^N \beta_i \tilde{k}_\epsilon(\cdot,\bx_i).
$
Using the RKHS inner product on \(\mathcal H_{\epsilon,N}\),
\[
\langle Uf,Ug\rangle_{\mathcal{H}_{\epsilon,N}}
=
\sum_{i,j=1}^N\frac{\alpha_i}{\phi(\bx_i)}\frac{\beta_j}{\phi(\bx_j)}k_{\epsilon,N}(\bx_i,\bx_j)
=
\sum_{i,j=1}^N\alpha_i\beta_j \tilde{k}_\epsilon(\bx_i,\bx_j)
=
\langle f,g\rangle_{\mathcal{H}_{k_\epsilon}}.
\]
Hence \(U\) is an isometry on \(\mathcal S_G\), by the extension theorem for bounded linear map \cite{conway2019course}, it extends uniquely to an isometry $U:\mathcal H_{\tilde{k}_\epsilon}\to \mathcal H_{\epsilon,N}$.

It remains to show surjectivity. For each \(\bx \in \mathcal{M}\),
\[
k_{\epsilon,N}(\cdot,\bx)=\phi(\bx)\,U(\tilde{k}_\epsilon(\cdot,\bx)),
\]
so \(\mathcal S_{\epsilon,N}\subset \operatorname{Ran}(U)\). Since \(\mathcal S_{\epsilon,N}\) is dense in \(\mathcal H_{\epsilon,N}\) and \(\operatorname{Ran}(U)\) is closed (range of an isometry), \(\operatorname{Ran}(U)=\mathcal H_{\epsilon,N}\). Therefore \(U\) is an isometric isomorphism.
\end{proof}

\begin{remark}
The assumption on $\phi$ is rather generic; otherwise the DM kernel is just trivially zero and not well defined. From \eqref{eq:phi_e}, \eqref{eq:qe_asymp}, \eqref{eq:qhate_asymp}, and the uniform sampling density $q(\bx) = 1/\text{Vol}(\mathcal{M})$, one can see that, 
\[
\phi(\bx):=   \frac{1}{q_{\epsilon,N}(\bx)\sqrt{\hat q_{\epsilon,N}(\bx)}} = \epsilon^{-d/4}\frac{\sqrt{\text{Vol}(\mathcal{M})}}{\sqrt{m_0}}
\left[
1
-\epsilon\frac{m_2}{2m_0}\,\omega(\bx)
+O(\epsilon^2) + O\left(\frac{\log N}{N\epsilon^d}\right),
\right],
\]
which is roughly constant for $\epsilon,N$ that chosen to satisfy Assumption~\ref{assumption1}.Since the expansion above is valid uniformly, $0<\phi^-\leq \|\phi\|_\infty \leq \phi^+$ is uniformly bounded. 
\end{remark}

The next result is the prevalent of the continuously embedding under the restriction on the manifold.

\begin{lem}\label{lem:embedding}
With the Definition~\ref{def:kernels}, that is, let $\phi_M:\BR^n\times\BR^n \to\BR$ denotes the Mat\'ern kernel with RKHS space that has equivalent Sobolev norm, $H^s(\BR^n)$, with $s>\frac{n-d}{2}$. Then, the RKHS corresponding to Gaussian kernel, $\mathcal{H}_{\tilde{k}_\epsilon}$ is continuously embedded in the RKHS space $\mathcal{H}_{k_M}$ with norm equivalent to $H^{s-\frac{n-d}{2}}(\mathcal M)$. In particular, there is $C>0$ such that
\[
\|f\|_{H^{s-\frac{n-d}{2}}(\mathcal M)}\le C\,\|f\|_{\mathcal H_{\tilde{k}_{\epsilon}}(\mathcal M)},\qquad \forall f\in\mathcal H_{\tilde{k}_{\epsilon}}(\mathcal M).
\]
\end{lem}

\begin{proof}
The key part of this proof uses similar idea as the proof of Theorem~5 in \cite{fuselier2012scattered}. By Lemma~4 in \cite{fuselier2012scattered},
for any positive-definite kernel $\phi$ on $\BR^n$ and its restriction $k:=\phi|_{\mathcal M\times \mathcal M}$,
$
\mathcal H_{k}(\mathcal M)=T_{\mathcal M}\big(\mathcal H_\phi(\BR^n)\big),
$
where $T_\mathcal{M}:H_\phi(\BR^n),\ T_{\mathcal M} \to H_{k}(\mathcal M)$ is the trace operator on the manifold $\mathcal{M}$. Here, the RKHS $\mathcal{H}_k$ is defined with
with norm (see Chapter 10 of \cite{wendland2004scattered}) defined as,
\[
\|f\|_{\mathcal H_{k}}=\inf\left\{\|g\|_{\mathcal H_\phi}:\ g\in\mathcal H_\phi(\BR^n),\ T_{\mathcal M}g=f\right\}.
\]
Applying this lemma to the two restricted of Gaussian and Mat\'ern kernels as defined in Definition~\ref{def:kernels}, we have:
$T_{\mathcal M}\big(\mathcal H_{\tilde{\phi}_{\epsilon}}\big)=\mathcal H_{\tilde{k}_{\epsilon}}$ and $T_{\mathcal M}\big(\mathcal H_{\phi_M}\big)=\mathcal H_{k_M}.
$ 

With this background, let us begin the proof. Take $f\in\mathcal H_{\tilde{k}_{\epsilon}}$. By the restricted-kernel trace/extension property (Lemma 4 in \cite{fuselier2012scattered}), there exists an isometric embedding
$
E:\mathcal H_{\tilde{k}_{\epsilon}}\to\mathcal H_{\tilde{\phi}_{\epsilon}}, T_{\mathcal M}(E f)=f,
$
where $Ef\in\mathcal H_{\tilde{\phi}_{\epsilon}}$ and
$\|Ef\|_{\mathcal H_{\tilde{\phi}_{\epsilon}}}=\|f\|_{\mathcal H_{\tilde{k}_{\epsilon}}}.
$
Since RKHS of the Gaussian kernel is continuously embedded in RKHS of the Mat\'ern kernel,  $\mathcal{H}_{\tilde{\phi}_\epsilon} \hookrightarrow \mathcal{H}_{\phi_M}$, it is clear that $Ef \in \mathcal{H}_{\phi_M}$ and
$\|Ef\|_{\mathcal{H}_{\phi_M}} \leq C\|Ef\|_{\mathcal{H}_{\tilde{\phi}_{\epsilon}}}$, for some $C>0$.
Thus,
\BEA
\|f\|_{\mathcal{H}_{k_M}} \leq \|Ef\|_{\mathcal{H}_{\phi_M}}\leq C\|Ef\|_{\mathcal{H}_{\tilde{\phi}_{\epsilon}}} = C \|f\|_{\mathcal{H}_{\tilde{k}_{\epsilon}}}.\label{embedding_proof}
\EEA
where the first inequality is by the definition of restricted kernel, the second inequality is by the fact that RKHS of the Gaussian kernel is continuously embedded in the RKHS of the  Mat\'ern kernel. Hence $\mathcal H_{\tilde{k}_{\epsilon}}\hookrightarrow \mathcal H_{k_M}$ continuously. 

By the norm equivalence assumption, $\mathcal H_{\psi_M}\simeq H^s(\BR^n), s>\frac{n-d}{2},
$
its trace space on $\mathcal M$ is $H^{s-\frac{n-d}{2}}(\mathcal M)$,  with equivalent norms, i.e.
$\mathcal H_{k_M}\simeq H^{s-\frac{n-d}{2}}(\mathcal M).$ Combining this and \eqref{embedding_proof}, the proof is complete.
\end{proof}

With these two lemmas, we can now state our main result:
\begin{theo}\label{thm2}
Let the Assumption~\ref{assumption1} and the assumptions in Lemma~\ref{lem:embedding} to be valid. Then, the oracle inequality \eqref{eq:oracle} is valid for kernel ridge regression solution with the DM kernel, $k=k_{\epsilon,N}$, or equivalently, $\mathcal{H} = \mathcal{H}_{\epsilon,N}$, and $p = \frac{d}{2(s-\frac{n-d}{2})}$ for any $s>n/2$.    
\end{theo}

\begin{proof}
Since Gaussian kernel, $\tilde{k}_\epsilon$, is a bounded measureable kernel on $\mathcal{M}$, then by Proposition~\ref{prop:Gaussapprox}, it is clear that $k_{\epsilon,N}$ is a bounded measurable kernel on $\mathcal{M}$. By the spectral convergence result (Proposition~\ref{lem:spectral}), we have,
$
|\nu_j - \nu_j^{\epsilon,N}| = O(\beta),$ where $\beta = \left(\frac{\log N}{N}\right)^{\frac{3}{8d+26}}$, $\nu_j^{\epsilon,N} = \frac{1}{\epsilon} \log(\lambda_j^{\epsilon,N})$ with $\lambda_{j,N}$ denotes eigenvalues of the corresponding integral operator, $T_{\epsilon,N}$ defined in Section~\ref{sec21}. Here $\nu_j$ denotes the $j$th eigenvalues of the negative definite Laplace-Beltrami operator which scales $\nu_j\propto -j^{\frac{2}{d}}$ (see \cite{colbois2015eigenvalues}). Denote $\nu_j(T_k) = \frac{1}{\epsilon}\log(\lambda_j(T_k))$. From \cite{von2008consistency}, $\nu_j(T_k)  - \nu_{j}^{\epsilon,N} = o(N)$, as $N\to \infty$. 
This implies that, 
\[\lambda_{j}(T_k) = e^{\epsilon\nu_j(T_k)}= e^{\epsilon\left(\nu_j(T_k)-  \nu_j^{\epsilon,N}+\nu_j^{\epsilon,N}- \nu_j +\nu_j\right)} = e^{\epsilon(o(N)+O(\beta) +\nu_j )}=ae^{\epsilon \nu_j} = a e^{-\epsilon j^{2/d}},\]
for some constant $a=e^{\epsilon (o(N)+O(\beta))}>0$ that is finite, which is faster than the algebraic decay rate in the assumption in the Proposition~\ref{prop:krr_errorbound}. 

Take $f\in \mathcal{H}_{\epsilon,N}$. By Lemma~\ref{lem:isoembedding}, $\|f\|_{\mathcal{H}_{\epsilon,N}} = \left\|U^{-1}f\right\|_{\mathcal{H}_{\tilde{k}_\epsilon}}$, where $U^{-1}f := \frac{f}{\phi}\in \mathcal{H}_{\tilde{k}_\epsilon}$. Since $\phi$ is uniformly bounded, we have
\BEA
\|f\|_\infty = \|UU^{-1}f\|_\infty \leq \|\phi\|_{\infty}\|U^{-1}f\|_\infty \leq \phi^+ \|U^{-1}f\|_\infty.\label{eq:proofeq1}
\EEA
By \eqref{lem:embedding}, we have $\mathcal H_{\tilde{k}_{\epsilon}}\hookrightarrow \mathcal H_{k_M}$, where the RKHS $\mathcal{H}_{k_M}$ has a norm equivalent to $H^{s-\frac{n-d}{2}}(\mathcal M)$, for any $s>\frac{n-d}{2}$.  

This means any function in $\mathcal H_{\tilde{k}_{\epsilon}}$ satisfies the inequality in \eqref{embeddingcond}. Therefore,
\BEA
\|U^{-1}f\|_{\infty} \leq C \|U^{-1}f\|_{\mathcal{H}_{\tilde{k}_\epsilon}}^p \|U^{-1}f\|_{L^2(\mathcal{M})}^{1-p} \leq C\|f\|_{\mathcal{H}_{\epsilon,N}}^p \frac{1}{\|\phi\|^{1-p}_\infty} \|f\|_{L^2(\mathcal{M})}^{1-p} \leq \frac{C}{(\phi^-)^{1-p}}\|f\|_{\mathcal{H}_{\epsilon,N}}^p  \|f\|_{L^2(\mathcal{M})}^{1-p},\label{eq:proofeq2}
\EEA
where we have used the isometry and the boundedness of $1/\phi$ in the last inequality. Combining \eqref{eq:proofeq1} and \eqref{eq:proofeq2}, it is clear that 
the inequality in \eqref{embeddingcond} is satisfied for $0<p= \frac{d}{2(s-\frac{n-d}{2})} <1$, for any $s > \frac{n}{2}$ for all functions in the RKHS $\mathcal{H}_{\epsilon,N}$. Therefore, the result in Proposition~\ref{prop:krr_errorbound} for the DM kernel is valid.
\end{proof}

\section{Numerical results} \label{sec5}

In this section, we first numerically verify the convergence of heat kernel estimation. Subsequently, we demonstrate how the normalization factor in the DM kernel improves the regression estimate relative to the Gaussian kernel on manifolds with boundary. We also show the performance comparison of the DMKRR and Gaussian KRR in approximating oscillatory targets and manifolds of high co-dimension. 

\subsection{Convergence of Heat Kernel Estimation}\label{sec:num_heat}

The convergence study of heat kernel using DM is performed on two manifolds: circle and flat torus.
For each manifold, the positive-time empirical diffusion kernel $H_{\epsilon,N}(\bx,\by;s)$, as defined in \eqref{eq:empiricalheatkernel}, is constructed using $N$ sample points, and then computed at a combination of source locations $\bx$ and target locations $\by$ for a fixed diffusion time $s$.

For the convergence study, a series of combinations of $\epsilon$ and $N$ is considered, where $s$ is an integer multiple of $\epsilon$.  For each $\epsilon$ and $N$, the heat kernel is computed as follows.  If $\epsilon=s$, then
$
H_{\epsilon,N}(\bx,\by;s) \approx \mathbf{k}_{\epsilon,N}(\bx,\by).
$
If $\epsilon < s$, then the Gram matrix is formed on the $N$ samples $\mathbf{K}_{\epsilon,N}$, and
$
H_{\epsilon,N}(\bx,\by;s) \approx \mathbf{k}_{\epsilon,N}(\bx,\mathbf{X}) \mathbf{K}_{\epsilon,N}^{p-2} \mathbf{k}_{\epsilon,N}(\mathbf{X},\by),
$
where $p=s/\epsilon$, and $\mathbf{k}_{\epsilon,N}(\bx,\mathbf{X})$ denotes the kernel section evaluated at $\bx$ against the $N$ samples, similarly for $\mathbf{k}_{\epsilon,N}(\mathbf{X},\by)$.

Subsequently, the maximum absolute error (MAE) is computed over the chosen set of sources and tests,
$$
E_\infty = \sup_{\bx,\by} \left| H_{\epsilon,N}(\bx,\by;s) - H(\bx,\by;s) \right|.
$$


The \textbf{unit circle} $\mathbb{S}^1$ is parameterized by the intrinsic coordinates $\theta$ through
$
\bx(\theta) = (\cos\theta, \sin\theta), \theta\in[0,2\pi).
$
The heat kernel on $\mathbb{S}^1$ is
$$
k_{\mathbb{S}^1}(\theta,\theta_0;t)
=
\frac{1}{\sqrt{4\pi t}}
\sum_{m\in\mathbb{Z}}
\exp\!\left(-\frac{(\theta-\theta_0+2\pi m)^2}{4t}\right),
\qquad \theta,\theta_0\in[0,2\pi).
$$
Numerically, the sum is truncated to $m=-5,-4,\cdots,5$, so that the truncation is practically zero.

The \textbf{flat torus} $\mathbb{T}^2 = \mathbb{S}^1\times \mathbb{S}^1$ is parameterized by the intrinsic coordinates $(\theta_1,\theta_2)$ through
$
\bx(\theta_1,\theta_2) = (\cos\theta_1, \sin\theta_1, \cos\theta_2, \sin\theta_2),\theta_1\in[0,2\pi), \theta_2\in[0,2\pi).
$
The heat kernel on $\mathbb{T}^2$ is a composition of the unit circle kernels,
$$
k_{\mathbb{T}^2}([\theta_1,\theta_2],[\theta_{0,1},\theta_{0,2}];t)
=
k_{\mathbb{S}^1}(\theta_1,\theta_{0,1};t)
k_{\mathbb{S}^1}(\theta_2,\theta_{0,2};t).
$$
Numerically, each $k_{\mathbb{S}^1}$ is truncated the same as in the unit circle case.

Due to its scaling, the DM kernel is numerically compared with the heat kernel multiplied by $\mathrm{Vol}(\mathcal{M})$.  For the cases considered,
$
\mathrm{Vol}(\mathbb S^1)=2\pi,
\mathrm{Vol}(\mathbb T^2)=4\pi^2.
$
In our numerical test below, the specific parameters for each manifold are listed below: {\bf Circle:} Tested on 4096 uniformly sampled points with 4 randomly selected sources.  Diffusion time $s=0.01$. {\bf Flat torus:} Tested on 64$\times$64 uniformly sampled points with 5 randomly selected sources.  Diffusion time $t=0.04$.
For each $\epsilon$ and $N$, 8 trials are done and the median error is reported.  Each trial uses a scrambled Sobol sequence to generate samples, which reduces variance in the errors of the 8 trials.

The convergence results are shown in Fig.~\ref{fig:heat_conv}.
Panels (a) and (b) show that, for a fixed $\epsilon$, the heat kernel error decreases as $N$ increases and eventually plateau at a sufficiently large $N$; this is when the $\epsilon^{1/4}$ term dominates the error in Theorem \ref{thm1}.
Panels (c) and (d) show the convergence of errors at the largest $N$ used in each case.  A clear convergence w.r.t. $\epsilon$ is observed, and the convergence rate is faster than the theoretical rate of $1/4$ in Theorem \ref{thm1}.

More details of the heat kernel study are provided in Appendix \ref{app:num_heat}, including the illustration of the heat kernels, detailed convergence plots involving both $\epsilon$ and $N$, and an additional case of heat kernel on disk.

\begin{figure}[htbp]
    \centering
    \includegraphics[width=.8\textwidth]{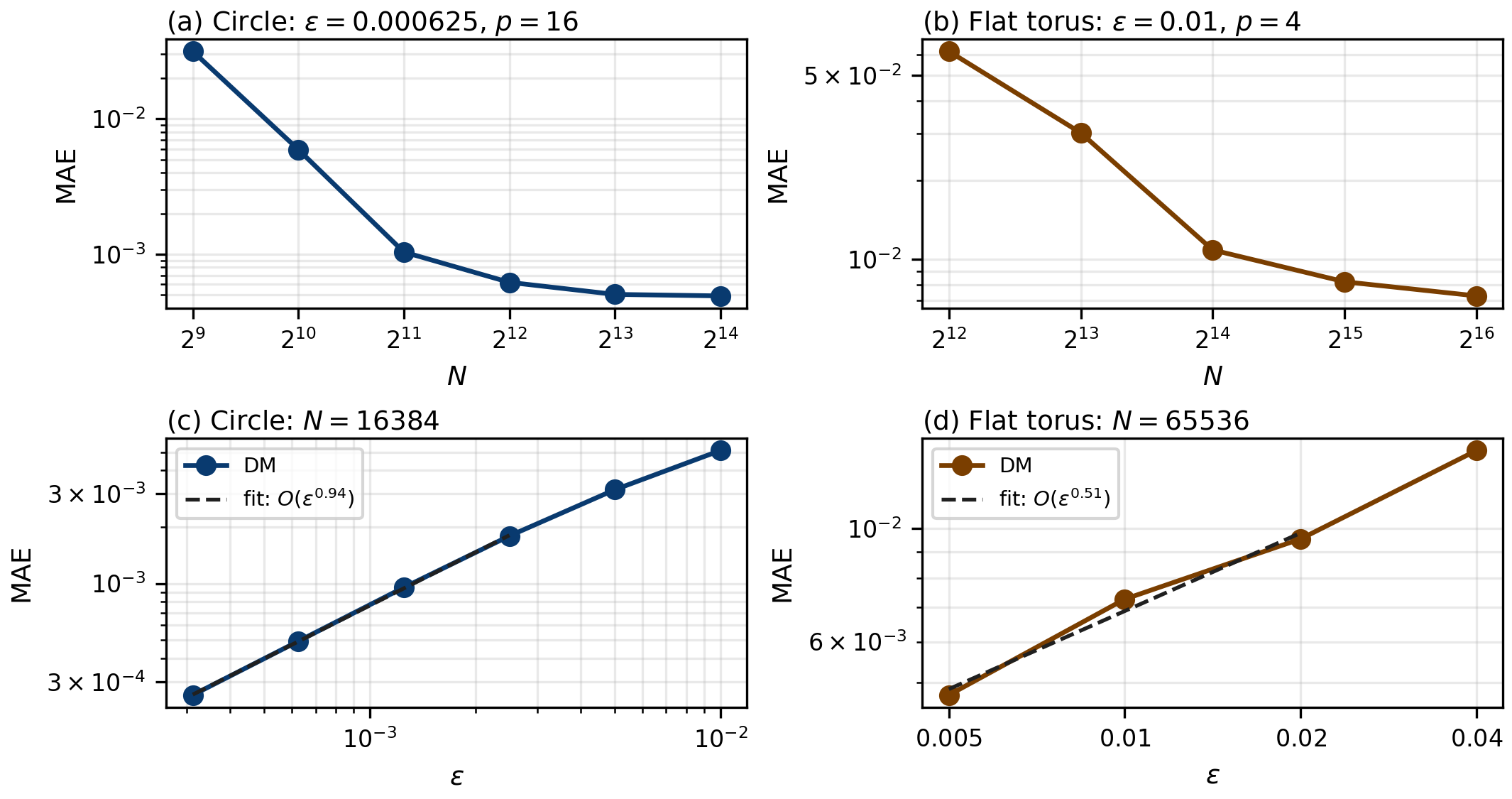}
    \caption{Convergence of heat kernels on circle and flat torus.}
    \label{fig:heat_conv}
\end{figure}

\subsection{Convergence of Kernel Ridge Regression}

In Section~\ref{sec5.2.1}, we compare the performance DM kernel to the Gaussian RBF kernel as a regressor for various types of functions of relevant classes on semicircle and full circle. The aim is to discover the type of functions for which DMKRR is effective. In Section~\ref{sec5.2.2}, we demonstrate the numerical convergence of KRR with these kernels. 

\subsubsection{Semicircle and full circle results}\label{sec5.2.1}

  Two 1D manifolds are considered: (1) semicircle: $\mathbb{S}^{1/2}=\{(\cos t,\sin t)\ |\ t\in[-\pi/2,\pi/2]\}$, with inner product $\langle g,h\rangle_{\rm semi}= \frac{1}{\pi}\int_{-\pi/2}^{\pi/2}g(t)h(t)\,dt$,
and (2) full circle: $\mathbb{S}^{1}=\{(\cos t,\sin t)\ |\ t\in[0,2\pi)\}$, with inner product $\langle g,h\rangle_{\rm full}=
 \frac{1}{2\pi}\int_0^{2\pi}g(\theta)h(\theta)\,d\theta$.  The inner products induce the $L^2$ norms for the respective manifolds which are used below.

\paragraph{Construction of labels.}
For the semicircle, the labels are a mixture of Laplace--Beltrami (LB) eigenfunctions and RBF integral operator eigenfunctions.  Specifically, consider the subspace of the first four odd Neumann LB eigenfunctions
\[
 W_{\rm LB}=\operatorname{span}\left\{\sqrt2\sin t,\sqrt2\sin3t,
 \sqrt2\sin5t,\sqrt2\sin7t\right\},
\]
and the subspace of the first four even RBF eigenfunctions
\[
 W_{\rm RBF}=\operatorname{span} \left\{ \psi_1(t),\psi_3(t),\psi_5(t),\psi_7(t)\right\},
\]
where $\psi_j$ are eigenfunctions of $\epsilon^{1/2}G_\epsilon$ as defined in \eqref{eq:Ge}
\comment{Gaussian integral eigenproblem,
\[
\int_{\mathbb{S}^{1/2}}\tilde{k}_{\epsilon}((\cos t,\sin t),(\cos s,\sin s))\psi_j(s)ds = \lambda_j \psi_j(t),\quad \|\psi_j\|_{L^2}=1
\]}
with $\epsilon=0.02$, and at the chosen $j$'s the eigenfunction is even, i.e., $\psi_j(-t)=\psi_j(t)$.

Subsequently, 12 functions are randomly sampled from each subspace and normalized to unity in the $L^2$ sense: $u_i\in W_{\rm LB}$ and $v_i\in W_{\rm RBF}$, $i=1,2,\cdots,12$.  For each pair of LB and RBF functions, define a function family,
\[
f_i(s) = \cos(\pi s/2)u_i+\sin(\pi s/2)v_i,\quad s\in[0,1].
\]
By design, $W_{\rm LB}$ is orthogonal to $W_{\rm RBF}$ (see Fig.~\ref{fig:circ_label} in Appendix \ref{app:circ}), so are $u_i$ and $v_i$; hence $\|f_i(s)\|_{L^2}=1$ for $s\in[0,1]$.
Twenty-eight values of $\{s_j\}_{j=1}^{28}$, with $s_1=0$ and $s_{28}=1$, are chosen and fixed, so each family contains 28 labels; this totals $12\times 28=336$ labels for the DM and RBF KRR models.

For the full circle, because the Gaussian kernel is translationally invariant over the circle, one can show that the eigenfunctions of the Gaussian integral eigenproblem are sinusoidal functions, which coincides with the LB eigenfunctions on the full circle.  As a result, the KRR labels for full circle are randomly sampled from a LB subspace consisting of the first four periodic sine/cosine pairs, and normalized to unity in the $L^2$ senses. The illustration of the labels are provided in Fig.~\ref{fig:circ_label} in Appendix \ref{app:circ}.

\paragraph{Comparison of KRR errors.}
Figure~\ref{fig:circ_krr}(a) shows the KRR errors of the semicircle families.  The $x$ axis uses a logarithmic resolution near $s=1$, so as to show the crossover of the DM and RBF errors.  The solid lines show the medium $L^2$ error across the 12 families, while the shaded regions are the empirical 10th–90th percentile bands.  The results show a consistent trend across the 12 families.  At $s=0$, when the labels are in $W_{\rm LB}$, DM is nearly two order of magnitude better than RBF, and DM remains significantly better than RBF for most values of $s<1$.  The RBF becomes advantageous only near $s=1$, where the labels consists of nearly only RBF eigenfunctions. Figure~\ref{fig:circ_krr}(b) shows the KRR errors of the full circle labels.  While the labels are linear combinations of LB eigenfunctions on the full circle, this time RBF shows consistent advantage over DM.

\begin{figure}[htbp]
    \centering
    \includegraphics[width=.8\textwidth]{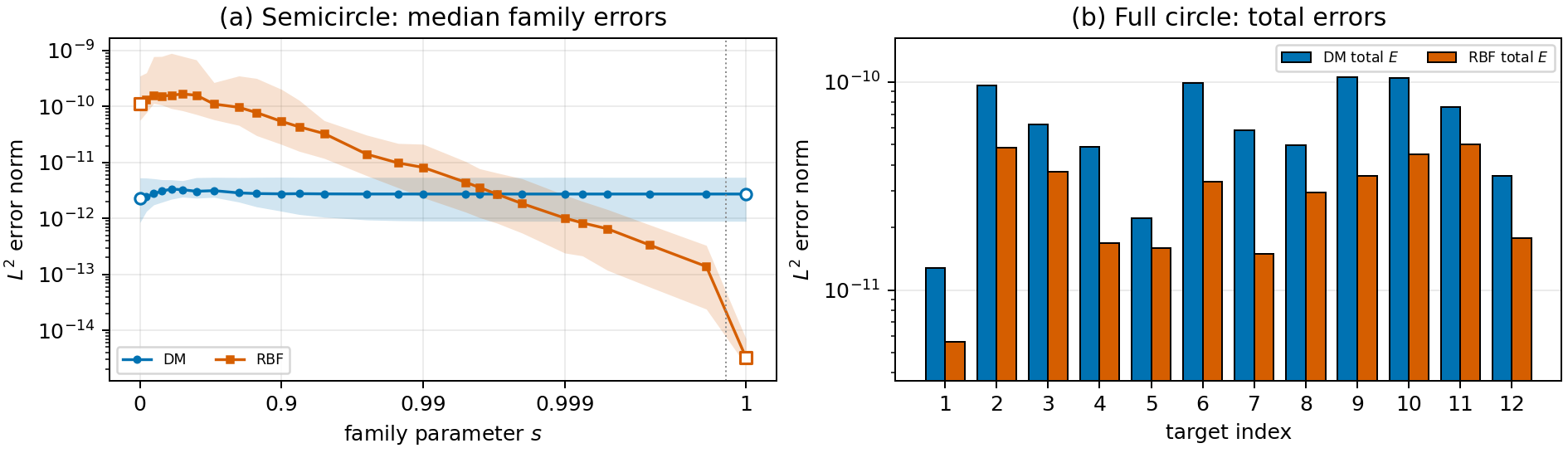}
    \caption{Comparison of DM and RBF KRR errors.}
    \label{fig:circ_krr}
\end{figure}

\paragraph{Subspace alignment analysis.} To understand the differences between DM and RBF in the KRR results, the subspaces containing the labels are compared with the subspaces spanned by the first eight Nystr\"om-extended eigenfunctions of the tuned KRR operators.
The subspaces are compared using the largest principal angle between two subspaces.  The detailed definitions and additional results are provided in Appendix \ref{app:circ}.

First, consider the semicircle cases.  For $s=0$, the labels are in the 4-dimensional LB subspace $W_{\rm LB}$.  The comparison of $W_{\rm LB}$ and each KRR subspace produces one angle, and the angle distributions associated with the DM and RBF models are shown in Fig.~\ref{fig:circ_angles}(a), with labels ``DM-LB'' and ``RBF-LB'' respectively, denoting the kernel being used in KRR - the space of the label function.
Similarly, the subspace comparison is done for the $s=1$, where the labels are in the 4-dimensional RBF subspace $W_{\rm RBF}$.  The angle distributions are labeled ``DM-RBF'' and ``RBF-RBF'' respectively in Fig.~\ref{fig:circ_angles}(a).
Across the semicircle tuned models, the median largest angles are $3.578$ and $22.54$ degrees against $W_{\rm LB}$, and $18.58$ and $0.1763$ degrees against $W_{\rm RBF}$, for DM and RBF respectively.
The small angle in ``DM-LB'' case is attributed to the capability of the DM kernel to approximate the LB heat kernel, while the small angle in ``RBF-RBF'' case is because that the RBF KRR operator is directly a finite-sample approximation of the RBF integral operator. Similar trend is observed with more KRR modes.
The subspace alignment clearly correlates with the KRR performance, that when the KRR subspace aligns well with the label space, the KRR error tends to be low.

The subspace alignment in the semicircle case is further illustrated for $W_{\rm LB}$ in Fig.~\ref{fig:circ_semi_lb}.  The dashed lines show the basis of $W_{\rm LB}$, and the solid lines are the linear combinations of the eight kernel eigenfunctions, best aligned to the LB basis.
Clearly, the LB basis are approximated well by DM eigenfunctions, but not by the RBF ones.  Again, the better approximation of the LB basis is well correlated to lower KRR error.  The basis alignment plots for other cases, showing the same trends, are provided in Appendix \ref{app:circ}.

The same subspace alignment analysis is applied to the full circle case.  The label space, a 8-dimensional LB subspace, is compared to the tuned KRR subspaces, and the angle distributions are shown in Fig.~\ref{fig:circ_angles}(b).  The medians are $2.091\times10^{-6}$ and $1.909\times10^{-6}$ degrees for DM and RBF, respectively.
The near perfect subspace alignment with the label space suggests that both DM and RBF are suitable for the LB labels on full circle.  Furthermore, the slightly lower angles for RBF do correlate to the smaller RBF KRR errors than the DM errors.

Overall, the 1D cases of semicircle and full circle show three findings.  First, DM and RBF kernels are suitable for fitting labels that align the best with their respective eigenfunctions.  Second, when both the kernels align well with the label space, the RBF kernel tends to perform better. Third, on closed manifolds, due to translational invariance, the eigenfunctions of RBF integral operator coincide with the eigenfunctions of LB eigenfunctions on the manifold. In this case, the DM kernel does not have any advantage over the RBF kernel in the KRR fitting. With the above findings, we consider manifolds with boundary in the rest of the numerical study.

\begin{figure}[htbp]
    \centering
    \includegraphics[width=0.8\textwidth]{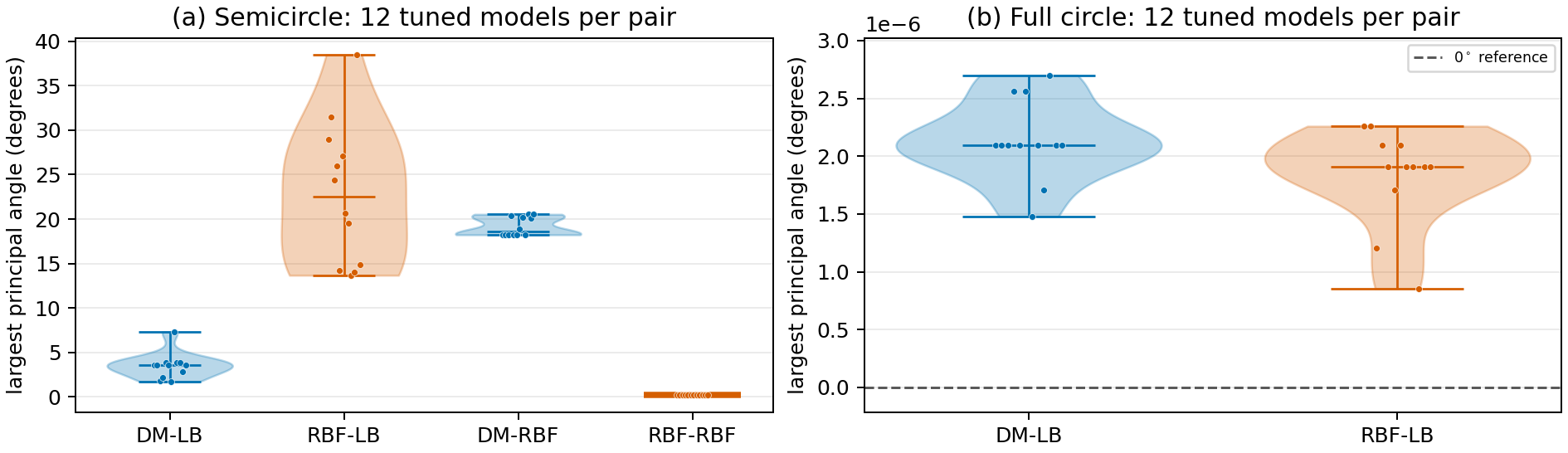}
    \caption{Distributions of subspace angles.}
    \label{fig:circ_angles}
\end{figure}

\begin{figure}[htbp]
    \centering
    \includegraphics[width=0.8\textwidth]{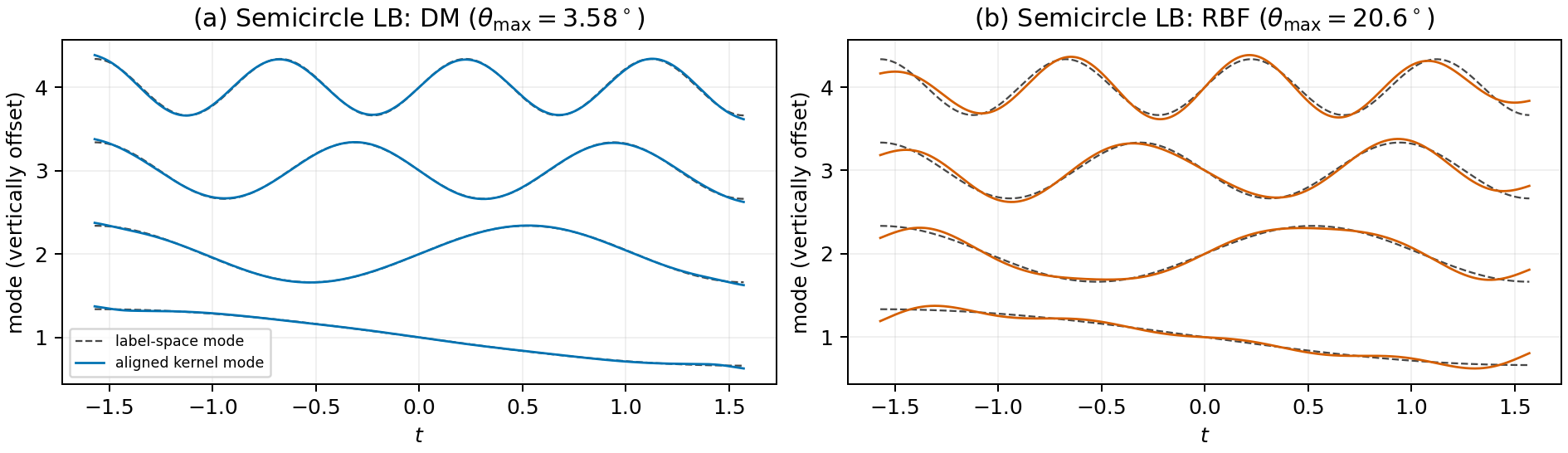}
    \caption{Subspace alignment for the LB basis.}
    \label{fig:circ_semi_lb}
\end{figure}

\subsubsection{The semi torus results}\label{sec5.2.2}

The convergence of KRR with DM and RBF kernels are compared on a semi-torus. In this section, we show only the results with Neumann eigenfunctions corresponding to
$(m,j)\in\left\{(1,0),(3,1),(6,3)\right\}$, representing low-, moderate-, and high-frequency targets  (Fig.~\ref{fig:torus_neumann})
as labels.
\comment{Three families of regression targets are considered for the semi-torus:
\begin{enumerate}
    \item Neumann Laplace--Beltrami eigenfunctions corresponding to
    \[
    (m,j)\in\{(1,0),(3,1),(6,3)\},
    \]
    representing low-, moderate-, and high-frequency targets
    (Fig.~\ref{fig:torus_neumann});

    \item Eigenfunctions of the Euclidean Laplacian in $\mathbb R^3$
    restricted to the semi-torus, with wavevectors
    \[
    \bk\in\{(1,0,0),(2,2,1),(4,3,3)\}
    \]
    (Fig.~\ref{fig:torus_fourier});

    \item One Dirichlet Laplace--Beltrami eigenfunction fixed at
    $(m,j)=(3,1)$ and represented through nonlinear embeddings in
    \[
    n\in\{3,7,11,15\}
    \]
    ambient dimensions (Fig.~\ref{fig:torus_ambient}).
\end{enumerate}
}
The precise definitions and numerical construction of these targets are given in Appendix~\ref{app:torus}.  All targets are normalized using a fixed surface-weighted root-mean-square factor, so that the reported RMSE values are comparable across target functions and experimental settings.  Appendix~\ref{app:torus} also provides details in KRR fitting and validation.

In Fig.~\ref{fig:torus_neumann}, we show the RMSEs as functions of $N$. Compare to RBF KRR, the DMKRR generally achieves lower test error and a more favorable error as $N$ increases. For the high-frequency label, the smallest training size exhibits substantial variability, indicating that $N=1024$ does not consistently resolve the oscillatory target. This variability decreases as the number of training samples
increases. We observe similar trends for labels drawn from an ambient Laplacian restricted on the semi-torus (see Appendix~\ref{app:torus}).

Finally, Fig.~\ref{fig:torus_ambient} presents the RMSE convergence for two-dimensional semitorus embedded in four ambient dimensions, $n$. The DMKRR generally achieves lower RMSE than RBF KRR and exhibits a faster decay of the RMSE with increasing $N$. However, the performance gap between the two methods tends to narrow as the ambient dimension increases. For the largest case shown, $n=15$, the two methods show comparable RMSE at smaller training sizes, whereas a clear advantage of DMKRR emerges from $N=2^{13}$ onward.  This result indicates that, when sufficiently many training samples are available, DMKRR remains effective even in the case of increasingly high co-dimension. 

\begin{figure}[htbp]
    \centering
    \includegraphics[width=.8\textwidth]{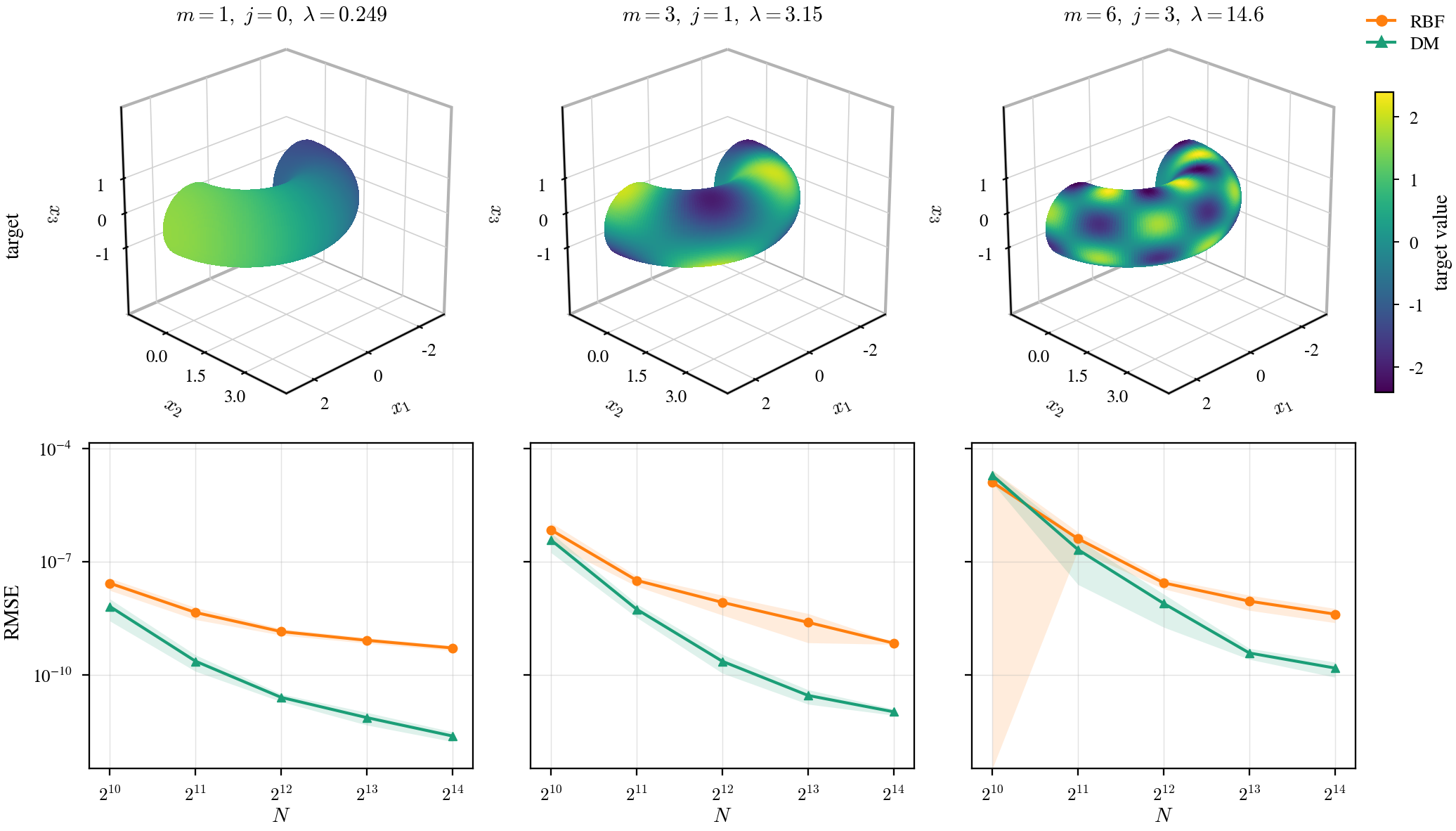}
    \caption{LB eigenfunctions with Neumann BC on semi-torus.}
    \label{fig:torus_neumann}
\end{figure}


\begin{figure}[htbp]
    \centering
    \includegraphics[width=\textwidth]{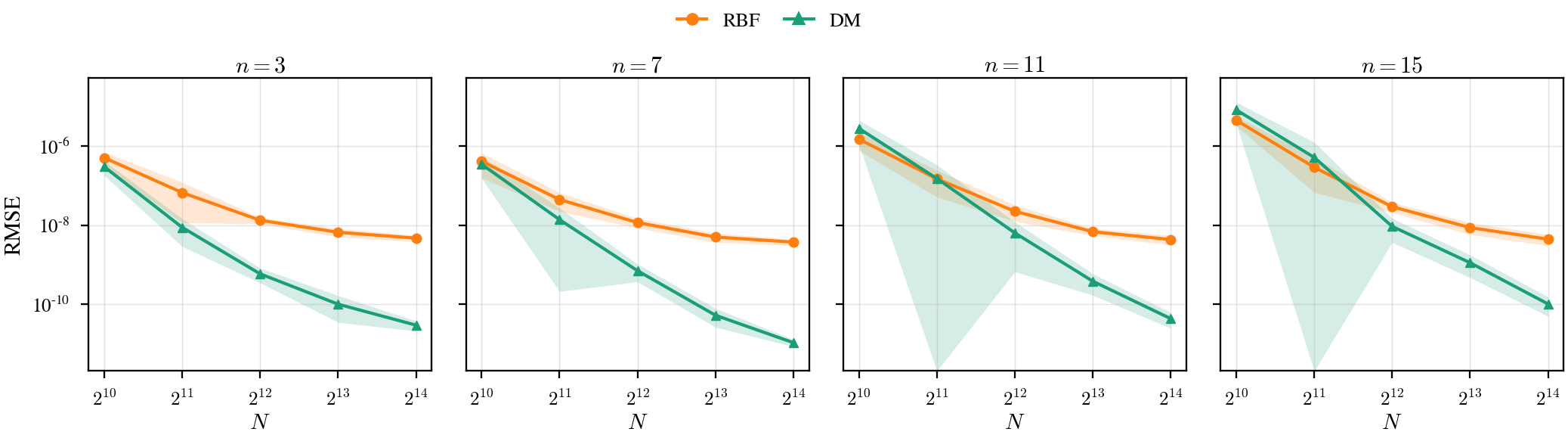}
    \caption{LB eigenfunctions with Dirichlet BC on semi-torus, embedded in higher ambient dimensions.}
    \label{fig:torus_ambient}
\end{figure}

\comment{
\begin{figure}[htbp]
    \centering
    \includegraphics[width=\textwidth]{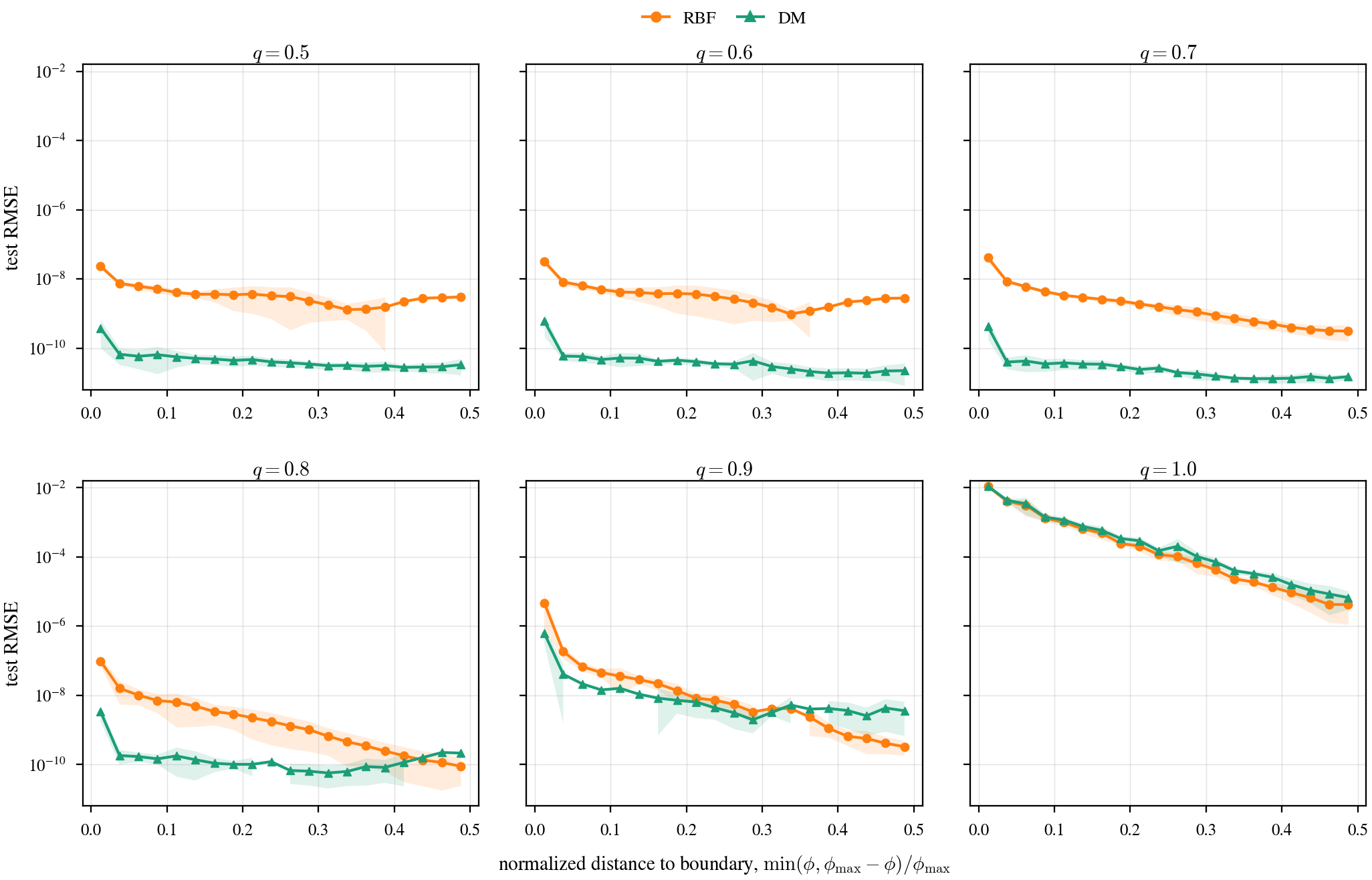}
    \caption{Test RMSE versus normalized distance from the boundary at \(N=2^{13}\).}
    \label{fig:partialtorus_dist}
\end{figure}
\begin{figure}[htbp]
    \centering
    \includegraphics[width=\textwidth]{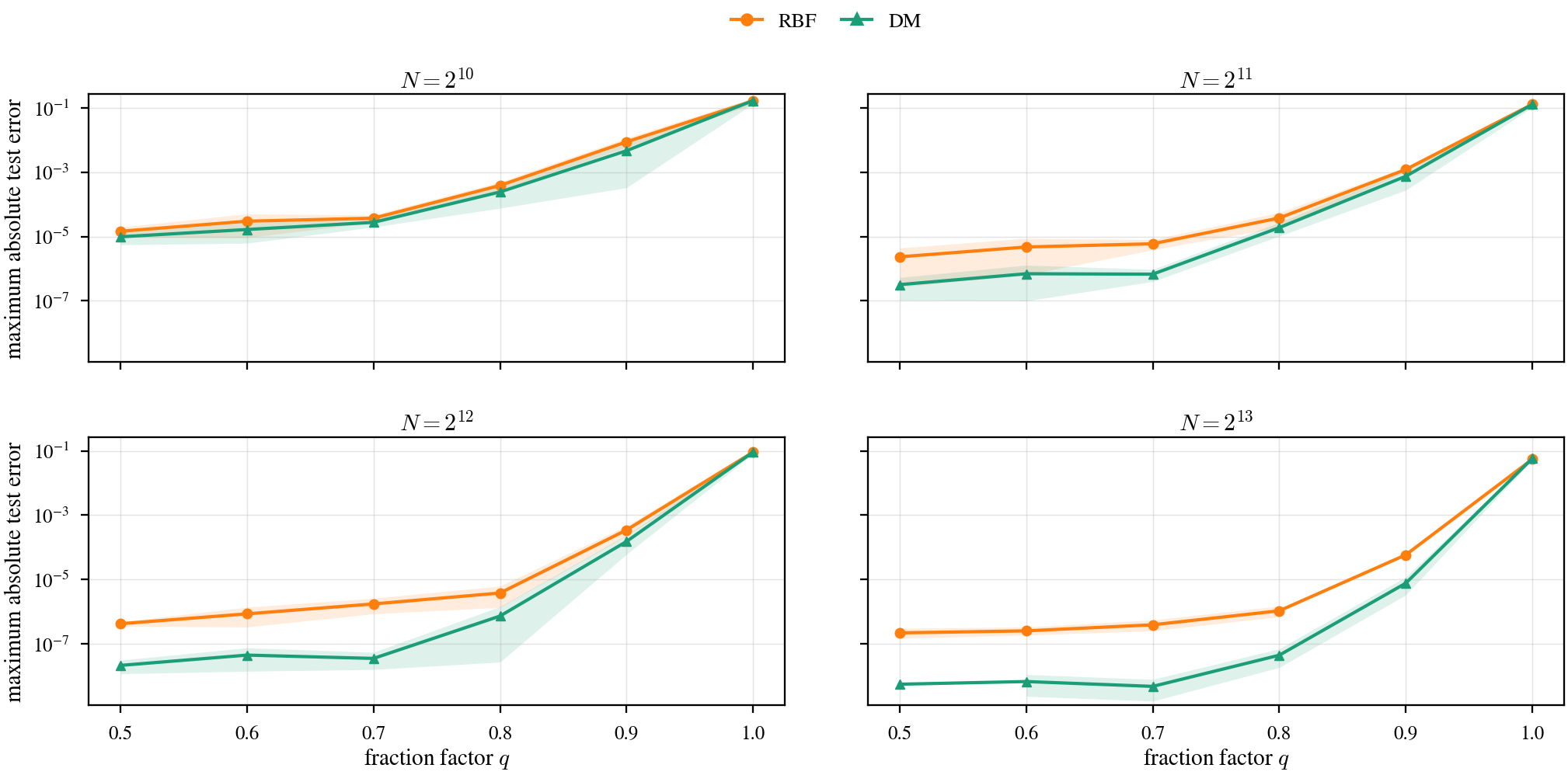}
    \caption{Maximum absolute error versus the fraction factor \(q\) for each $N$ considered.}
    \label{fig:partialtorus_maxerror_q}
\end{figure}
}



\section{Summary} \label{sec6}

This paper provides a theoretical foundation for kernel ridge regression based on the data-driven diffusion maps (DM) kernel. We show that, under suitable assumptions, {a family of appropriately scaled kernels \color{black}induced by the leading orthonormal basis of the DM kernel} converges to the heat kernel on the underlying manifold and that its associated RKHS inherits desirable approximation and generalization properties through its connection to Gaussian, Matérn, and Sobolev spaces. These results justify the use of the DM kernel within the established framework of kernel methods and are supported by numerical experiments demonstrating its advantages over the Gaussian kernel, particularly for learning functions on manifolds with boundary and across a range of frequencies and co-dimensions for large enough data.

Several important questions remain open. Although the normalization factor in the DM kernel consistently improves regression performance in problems involving manifolds with boundary, the theoretical mechanism underlying this improvement is not yet understood. In addition, while the DM kernel has shown promising empirical performance for learning solution operators of dynamical systems with clean data \cite{song2025learning} and noisy data \cite{kreider2026learning}, a rigorous theoretical analysis of its effectiveness in this setting remains an important direction for future work.

\paragraph{\bf Acknowledgments.} This research is partially supported by the National Science Foundation grants DMS-2505605. We thank Alex Townsend for insightful discussions and comments on the manuscript.

\section{Data Availability}

No datasets were generated or analyzed during the current study. Data sharing is not applicable to this article.

\section{Conflict of Interest}

On behalf of all authors, the corresponding author states that there is no conflict of interest.

\bibliographystyle{plain}  
\bibliography{ref}

\begin{thebibliography}{10}

\bibitem{adams2003sobolev}
Robert~A Adams and John~JF Fournier.
\newblock {\em Sobolev spaces}, volume 140.
\newblock Elsevier, 2003.

\bibitem{BachLanckrietJordan2004}
Francis~R. Bach, Gert R.~G. Lanckriet, and Michael~I. Jordan.
\newblock Multiple kernel learning, conic duality, and the {SMO} algorithm.
\newblock In {\em Proceedings of the 21st International Conference on Machine
  Learning (ICML)}, page~41, 2004.

\bibitem{bennett1988interpolation}
Colin Bennett and Robert~C Sharpley.
\newblock {\em Interpolation of operators}, volume 129.
\newblock Academic press, 1988.

\bibitem{calder2022lipschitz}
Jeff Calder, Nicholas Trillos, and Marta Lewicka.
\newblock {Lipschitz regularity of graph {Laplacians} on random data clouds}.
\newblock {\em SIAM Journal on Mathematical Analysis}, 54(1):1169--1222, 2022.

\bibitem{calder2022improved}
Jeff Calder and Nicolas~Garcia Trillos.
\newblock Improved spectral convergence rates for graph laplacians on
  $\varepsilon$-graphs and k-nn graphs.
\newblock {\em Applied and Computational Harmonic Analysis}, 60:123--175, 2022.

\bibitem{christmann2008support}
Andreas Christmann and Ingo Steinwart.
\newblock {\em Support vector machines}.
\newblock Springer, 2008.

\bibitem{coifman2006diffusion}
Ronald~R Coifman and St{\'e}phane Lafon.
\newblock Diffusion maps.
\newblock {\em Applied and computational harmonic analysis}, 21(1):5--30, 2006.

\bibitem{colbois2015eigenvalues}
Bruno Colbois, Ahmad El~Soufi, and Alessandro Savo.
\newblock Eigenvalues of the laplacian on a compact manifold with density.
\newblock {\em Communications in Analysis and Geometry}, 23(3):639--670, 2015.

\bibitem{conway2019course}
John~B Conway.
\newblock {\em A course in functional analysis}.
\newblock Springer, 2019.

\bibitem{debnath2005introduction}
Lokenath Debnath and Piotr Mikusinski.
\newblock {\em {Introduction to Hilbert spaces with applications}}.
\newblock Elsevier, 2005.

\bibitem{dunson2021spectral}
David~B Dunson, Hau-Tieng Wu, and Nan Wu.
\newblock Spectral convergence of graph {Laplacian} and heat kernel
  reconstruction in {$L^\infty$} from random samples.
\newblock {\em Applied and Computational Harmonic Analysis}, 55:282--336, 2021.

\bibitem{fuselier2012scattered}
Edward Fuselier and Grady~B Wright.
\newblock Scattered data interpolation on embedded submanifolds with restricted
  positive definite kernels: Sobolev error estimates.
\newblock {\em SIAM Journal on Numerical Analysis}, 50(3):1753--1776, 2012.

\bibitem{gine2002rates}
Evarist Gin{\'e} and Armelle Guillou.
\newblock Rates of strong uniform consistency for multivariate kernel density
  estimators.
\newblock In {\em Annales de l'Institut Henri Poincare (B) Probability and
  Statistics}, volume~38, pages 907--921. Elsevier, 2002.

\bibitem{gine2016mathematical}
Evarist Gin{\'e} and Richard Nickl.
\newblock {\em Mathematical Foundations of Infinite-Dimensional Statistical
  Models}, volume~40.
\newblock Cambridge University Press, 2016.

\bibitem{harlim2023radial}
John Harlim, Shixiao~Willing Jiang, and John~Wilson Peoples.
\newblock Radial basis approximation of tensor fields on manifolds: from
  operator estimation to manifold learning.
\newblock {\em Journal of Machine Learning Research}, 24(345):1--85, 2023.

\bibitem{jiang2023ghost}
Shixiao~Willing Jiang and John Harlim.
\newblock Ghost point diffusion maps for solving elliptic pdes on manifolds
  with classical boundary conditions.
\newblock {\em Communications on Pure and Applied Mathematics}, 76(2):337--405,
  2023.

\bibitem{kreider2026learning}
Max Kreider, John Harlim, and Daning Huang.
\newblock Learning dynamical systems from noisy data with weak-form kernel
  ridge regression.
\newblock {\em arXiv preprint arXiv:2607.00257}, 2026.

\bibitem{LanckrietEtAl2004}
Gert R.~G. Lanckriet, Nello Cristianini, Peter Bartlett, Laurent~El Ghaoui, and
  Michael~I. Jordan.
\newblock Learning the kernel matrix with semidefinite programming.
\newblock {\em Journal of Machine Learning Research}, 5:27--72, 2004.

\bibitem{OwhadiScovelSuleiman2019}
Houman Owhadi, Clint Scovel, and Gene~Ryan Suleiman.
\newblock Kernel flows: From learning kernels from data into the abyss.
\newblock {\em Journal of Computational Physics}, 389:22--47, 2019.

\bibitem{Peoples2025}
J.~Wilson Peoples and John Harlim.
\newblock Spectral convergence of symmetrized graph {Laplacian} on manifolds
  with boundary.
\newblock {\em Foundations of Data Science}, 8:119--167, 2026.

\bibitem{scholkopfsmola2002}
Bernhard Sch{\"o}lkopf and Alexander~J. Smola.
\newblock {\em Learning with Kernels: Support Vector Machines, Regularization,
  Optimization, and Beyond}.
\newblock MIT Press, Cambridge, MA, 2002.

\bibitem{song2025learning}
Jiwoo Song, Daning Huang, and John Harlim.
\newblock Learning solution operator of dynamical systems with diffusion maps
  kernel ridge regression.
\newblock {\em arXiv preprint arXiv:2512.17203}, 2025.

\bibitem{steinwart2009optimal}
Ingo Steinwart, Don~R Hush, and Clint Scovel.
\newblock Optimal rates for regularized least squares regression.
\newblock In {\em Proceedings of the 22nd Annual Conference on Learning
  Theory}, pages 79--93, 2009.

\bibitem{von2008consistency}
Ulrike Von~Luxburg, Mikhail Belkin, and Olivier Bousquet.
\newblock Consistency of spectral clustering.
\newblock {\em The Annals of Statistics}, pages 555--586, 2008.

\bibitem{wendland2004scattered}
Holger Wendland.
\newblock {\em Scattered data approximation}, volume~17.
\newblock Cambridge university press, 2004.

\end{thebibliography}

\comment{

}

\appendix

{
\section{Continuous DM kernel asymptotic}\label{app:A}

In this Appendix, we report the detailed calculation that reveals the scaling of the DM kernel and shows that the normalization approximates the Laplace-Beltrami operator, supporting the definition of Continuous Diffusion Maps and the calculation in Section~\ref{sec31}.

Consider the following asymptotic expansion \cite{coifman2006diffusion}: For $f\in C^3(\mathcal{M})$ where $\mathcal{M}\subset \mathbb{R}^n$ is a $d$-dimensional smooth manifold,
\[
G_\epsilon f(\bx):=\epsilon^{-d/2}\int_{\mathcal M}\tilde k_\epsilon(\bx,\by)f(\by)\,d\mathrm{Vol}(\by)
=m_0 f(\bx)+\epsilon m_2\big(\omega(\bx)f(\bx)+\Delta_g f(\bx)\big)+O(\epsilon^2).
\]

Let $q:\mathcal{M}\to \BR$ denotes the density defined with respect to the volume form, then it is immediately clear that:
\[
q_\epsilon(\bx)=\epsilon^{d/2} G_\epsilon q(\bx)= \epsilon^{d/2}( m_0q(\bx)+\epsilon m_2\big(\omega(\bx)q(\bx)+\Delta_g q(\bx)\big)+O(\epsilon^2)),
\]
such that,
\BEA
\frac{1}{q_\epsilon(\bx)} &=& \epsilon^{-d/2}\left[\frac{1}{m_0q(\bx)}-\epsilon\frac{m_2}{m_0^2}\left(\frac{\omega(\bx)}{q(\bx)}+\frac{\Delta_gq(\bx)}{q(\bx)^2}\right)+O(\epsilon^2)\right],\notag\\
\frac{1}{q_\epsilon(\bx)q_\epsilon(\by)} &=& \frac{\epsilon^{-d}}{m_0^2q(\bx)q(\by)}\left[1-\epsilon\frac{m_2}{m_0}\left(\omega(\bx)+\frac{\Delta_gq(\bx)}{q(\bx)} +\omega(\by)+\frac{\Delta_gq(\by)}{q(\by)} \right)+O(\epsilon^2)\right], \notag \\
\hat{k}_\epsilon(\bx,\by) &=& \frac{\epsilon^{-d} \tilde{k}_\epsilon(\bx,\by)}{m_0^2q(\bx)q(\by)}\left[1-\epsilon\frac{m_2}{m_0}\left(\omega(\bx)+\frac{\Delta_gq(\bx)}{q(\bx)} +\omega(\by)+\frac{\Delta_gq(\by)}{q(\by)} \right)+O(\epsilon^2)\right].\notag
\EEA
Since
\[
\frac{f(\bx)q(\bx)}{q_\epsilon(\bx)} = \epsilon^{-d/2}\left[\frac{f(\bx)}{m_0}-\epsilon\frac{m_2}{m_0^2}\left(\omega(\bx)f(\bx)+\frac{f(\bx)\Delta_gq(\bx)}{q(\bx)}\right)+O(\epsilon^2)\right],
\]
we have
\BEA
\widetilde{G}_\epsilon f(\bx) &:=&  \int_\mathcal{M} \tilde{k}_\epsilon(\bx,\by) f(\by)\frac{q(\by)}{q_\epsilon(\by)}d\text{Vol}(\by)\notag \\
&=& \epsilon^{d/2}\left[ m_0 \frac{f(\by)q(\by)}{q_\epsilon(\by)} + \epsilon m_2 \left(\omega(\bx)\frac{f(\by)q(\by)}{q_\epsilon(\by)} + \Delta_g\frac{f(\by)q(\by)}{q_\epsilon(\by)}\right) + O(\epsilon^2)\right] \notag \\
&=& f(\bx) -\epsilon \frac{m_2}{m_0}\left(\omega(\bx)f(\bx)+\frac{f(\bx)\Delta_gq(\bx)}{q(\bx)}\right) + \epsilon m_2 \omega(\bx) \frac{f(\bx)}{m_0} + \epsilon \frac{m_2}{m_0} \Delta_g f(\bx) + O(\epsilon^2) \notag \\ &=&
f(\bx) -\epsilon \frac{m_2}{m_0} \frac{f(\bx)\Delta_gq(\bx)}{q(\bx)}+ \epsilon \frac{m_2}{m_0} \Delta_g f(\bx) + O(\epsilon^2) \notag
\EEA
and
\BEA
\frac{\widetilde{G}_\epsilon f(\bx)}{q_\epsilon(\bx)} = \epsilon^{-d/2}
\left[
\frac{f(\bx)}{m_0 q(\bx)}
+\epsilon\left(
\frac{1}{m_0 q(\bx)}
\frac{m_2}{m_0}
\left(
\Delta_g f(\bx)-\frac{f(\bx)\Delta_g q(\bx)}{q(\bx)}
\right)
- f(\bx)\frac{m_2}{m_0^2}
\left(
\frac{\omega(\bx)}{q(\bx)}
+\frac{\Delta_g q(\bx)}{q(\bx)^2}
\right)
\right)
+O(\epsilon^2)
\right].\notag
\EEA

Subsequently,
\[
\hat{q}_\epsilon(\bx) = \frac{\widetilde{G}_\epsilon 1(\bx)}{q_\epsilon(\bx)}=
\epsilon^{-d/2}
\left[
\frac{1}{m_0 q(\bx)}
+\epsilon \frac{m_2}{m_0^2}
\left(
-\frac{2\Delta_g q(\bx)}{q(\bx)^2}
-\frac{\omega(\bx)}{q(\bx)}
\right)
+O(\epsilon^2)
\right].
\]
Since,
\BEA
\frac{1}{\sqrt{\hat q_\epsilon(\bx)}}
&=& \epsilon^{d/4}\sqrt{m_0 q(\bx)}
\left[
1
+\frac{\epsilon}{2}\frac{m_2}{m_0}
\left(
\frac{2\Delta_g q(\bx)}{q(\bx)}+\omega(\bx)
\right)
+O(\epsilon^2)
\right]\notag \\
\frac{1}{\sqrt{\hat q_\epsilon(\bx)\hat q_\epsilon(\by)}}
&=&
\epsilon^{d/2} m_0 \sqrt{q(\bx)q(\by)}
\left[
1+\frac{\epsilon}{2}\frac{m_2}{m_0}
\left(
\frac{2\Delta_g q(\bx)}{q(\bx)}+\omega(\bx)
+\frac{2\Delta_g q(\by)}{q(\by)}+\omega(\by)
\right)
+O(\epsilon^2)
\right].\notag
\EEA
we obtain an estimate to the continuous Diffusion Maps kernel,
\BEA
k_\epsilon(\bx,\by)
=
\frac{\epsilon^{-d/2}\,\tilde{k}_\epsilon(\bx,\by)}{m_0\sqrt{q(\bx)q(\by)}}
\left[
1
-\frac{\epsilon}{2}\frac{m_2}{m_0}
\left(\omega(\bx)+\omega(\by)\right)
+O(\epsilon^2)
\right].\label{CDMexpansion}
\EEA

In the remaining of this Appendix, we want compute,
\[
S_\epsilon f(\bx) = \int_\mathcal{M} k_\epsilon(\bx,\by)f(\by)q(\by)d\text{Vol}(\by). 
\]
To this end, we note that
\[
 \frac{f(\bx)q(\bx)}{q_\epsilon(\bx)\sqrt{\hat q_\epsilon(\bx)}}
=
\epsilon^{-d/4}\sqrt{m_0 q(\bx)}
\left[
\frac{f(\bx)}{m_0}
-\epsilon\frac{m_2}{2m_0^2}\,\omega(\bx)f(\bx)
+O(\epsilon^2)
\right],
\]
and 
\BEA
 \frac{1}{q_\epsilon(\bx)\sqrt{\hat q_\epsilon(\bx)}}
=
\epsilon^{-d/4}\frac{\sqrt{m_0}}{\sqrt{q(\bx)}}
\left[
\frac{1}{m_0}
-\epsilon\frac{m_2}{2m_0^2}\,\omega(\bx)
+O(\epsilon^2)
\right].\label{eq:phi_e}
\EEA
Then,
\[
\int_{\mathcal{M}} \tilde{k}_{\epsilon}(\bx,\by)\,
 \frac{f(\by)q(\by)}{q_\epsilon(\by)\sqrt{\hat q_\epsilon(\by)}}\,d\mathrm{Vol}(\by)
=
\epsilon^{d/4}
\left[
\sqrt{m_0 q(\bx)}\,f(\bx)
+\epsilon\frac{m_2}{\sqrt{m_0}}
\left(
\frac{1}{2}\omega(\bx)\sqrt{q(\bx)}f(\bx)
+
\Delta_g(\sqrt{q}f)(\bx)
\right)
+O(\epsilon^2)
\right].
\]
such that, 
\BEA
S_\epsilon f(\bx) &=& \int_{\mathcal{M}} \frac{\tilde{k}_{\epsilon}(\bx,\by)}{q_\epsilon(\bx)\sqrt{\hat q_\epsilon(\bx)}q_\epsilon(\by)\sqrt{\hat q_\epsilon(\by)}}f(\by)q(\by)\,d\mathrm{Vol}(\by)\notag \\
&=& f(\bx)
+\epsilon\frac{m_2}{m_0}\,
\frac{\Delta_g(\sqrt{q}f)(\bx)}{\sqrt{q(\bx)}}
+O(\epsilon^2).\notag
\EEA

Then, we have
\[
L_\epsilon f(\bx) = \frac{m_0(S_\epsilon -I )f(\bx)} {\epsilon m_2} = \frac{\Delta_g (\sqrt{q}f)(\bx)}{\sqrt{q(\bx)}} +O(\epsilon).
\]
If $q=1/\text{Vol}(\mathcal{M})$, then $L_\epsilon$ approximate the Laplace-Beltrami operator. For nonuniform density $q$, if $\{\lambda,\varphi\}$ is an eigensolution to $L_\epsilon$, then $\{\lambda \sqrt{q}\varphi\}$ is an eigensolution to the Laplace-Beltrami operator, that is,
\[
\Delta_g (\sqrt{q}\varphi)(\bx) = \lambda (\sqrt{q}\varphi)(\bx),
\]
}

\section{Mathematical proofs}\label{app:proof}

\subsection{Proof of Lemma~\ref{lem:3.1}}\label{proof:lemma3.1}

The key idea of the proof is to  employ the standard KDE uniform estimate following \cite{gine2002rates} to deduce errors between $q_{\epsilon,N}$ and $q_\epsilon$ as well as between $\hat{q}_{\epsilon,N}$ and $\hat{q}_\epsilon$.
Specifically, we will make use of the following concentration bound.

\begin{lem}[Talagrand/Bernstein-type inequality, see Chapter 3 of \cite{gine2016mathematical}]\label{lem:talagrand}
Let $X_1,\dots,X_N$ be i.i.d. with law $P$, and let $\mathcal F$ be a measurable class of real-valued functions on the sample space. Define
\[
Pf:=\mathbb E[f(\bx)],
\qquad
P_Nf:=\frac1N\sum_{i=1}^N f(X_i).
\]
Assume that $|f|\le b$ for all $f\in\mathcal F$, and set
\[
\sigma^2:=\sup_{f\in\mathcal F}\mathrm{Var}(f(\bx)).
\]
If $\mathcal F$ is VC-type (equivalently, has suitable entropy control), then there exist constants $C,c>0$ such that for all $t\ge 0$,
\[
\Pr\!\left(
\sup_{f\in\mathcal F}|(P_N-P)f|
> C\left(\sigma\sqrt{\frac{\log N+t}{N}}+b\frac{\log N+t}{N}\right)
\right)
\le e^{-ct}.
\]
\end{lem}

With this Lemma, we have:

\begin{proof}[Proof of Lemma~\ref{lem:3.1}]
Define $f_\bx(\by):=\tilde k_\epsilon(\bx,\by)$ and $\mathcal F_\epsilon:=\{f_\bx:\bx\in\mathcal M\}$. Then
\[
\sup_{\bx\in\mathcal M}\big|q_{\epsilon,N}(\bx)-q_\epsilon(\bx)\big|
=\sup_{f\in\mathcal F_\epsilon}\big|(P_N-P)f\big|.
\]
Since $0\le \tilde k_\epsilon\le 1$, we have
\[
b_\epsilon:=\sup_{f\in\mathcal F_\epsilon}\|f\|_\infty\le 1.
\]
For the variance term, for each fixed $\bx$,
\[
\mathbb E\left[f_\bx(\bx)^2\right]
=\int_{\mathcal M}e^{-\|\bx-\by\|^2/(2\epsilon)}q(\by)\,d\mathrm{Vol}(\by)
=O(\epsilon^{d/2}),
\]
uniformly in $\bx$ by the same local normal-coordinate argument used for $q_\epsilon$.
Hence
\[
\sigma_\epsilon^2:=\sup_{f\in\mathcal F_\epsilon}\mathrm{Var}(f(\bx))\lesssim \epsilon^{d/2}.
\]
Applying Lemma~\ref{lem:talagrand} with $t=c^{-1}\log N$, $\sigma_\epsilon^2\lesssim\epsilon^{d/2}$, and $b_\epsilon\le 1$ gives 
\[
\sup_{\bx\in\mathcal M}\big|q_{\epsilon,N}(\bx)-q_\epsilon(\bx)\big|
\le C_1\sqrt{\frac{\epsilon^{d/2}\log N}{N}}+C_2\frac{\log N}{N}.
\]
with probability at least $1-N^{-1}$ . Under $\log N/(N\epsilon^{d/2})\to 0$, the linear remainder is lower order:
\[
\frac{\log N}{N}=o\!\left(\sqrt{\frac{\epsilon^{d/2}\log N}{N}}\right),
\]
and the proof is  for \eqref{eq:error_q_en}. For \eqref{eq:error_qhat_en}, we can repeat the argument above with $f_\bx(\by) = \frac{\tilde{k}_{\epsilon}(\bx,\by)}{q_{\epsilon,N}(\bx)q_{\epsilon,N}(\by)} = \frac{\tilde{k}_{\epsilon}(\bx,\by)}{q_{\epsilon}(\bx)q_{\epsilon}(\by)} + O\left(\epsilon^{d/2}\sqrt{\frac{\log N}{N\epsilon^{d/2}}}\right)$. Thus, we have $b_\epsilon \leq \epsilon^{-d}\left(1 + \epsilon^{d/2}\sqrt{\frac{\log N}{N\epsilon^{d/2}}}\right)$ and $\sigma_\epsilon^2\lesssim \epsilon^{-2d}\left(1 + \epsilon^{d/2}\sqrt{\frac{\log N}{N\epsilon^{d/2}}} \right )$, and the result follows the same argument as above.
\end{proof}

\subsection{Proof of Lemma~\ref{lem:localerror_k}}\label{proof:lemma3.2}

We note that,
\BEA
\left|  k_{\epsilon,N}(\bx,\by)- k_{\epsilon,N}(\bx',\by)  \right| \leq \|\nabla_{\bx}k_{\epsilon,N}(\tilde{\bx})\| \|\bx- \bx'\|,\label{eq:lips+eq1}
\EEA
where $\tilde{\bx}\in \gamma_{\bx,\bx'} \subset B_{h}(\by) \subset \mathcal{M}$, where $\gamma_{\bx,\bx'}$ denotes the geodesic path connecting $\bx$ and $\bx'$. From Proposition~\ref{prop:Gaussapprox} and the smoothness of the normalization factor (Remark~\ref{rem:smoothbound}), one have, with probability higher than $1-\frac{4}{N}$,
\[
\|\nabla_{\bx} k_{\epsilon,N} (\tilde{\bx},\by)\| = \frac{\text{Vol}(\mathcal{M})}{m_0 \epsilon^{d/2}}\left[\| \nabla_{\bx}\tilde{k}_{\epsilon}(\tilde{\bx},\by)  \| \left(1+ O(\epsilon) +O\!\left(\sqrt{\frac{\log N}{N\epsilon^{d}}}\right)\right)  + \epsilon\|\nabla_{\bx}\omega(\tilde{\bx})\|\right] = O \left(h \epsilon^{-1-d/2}\right),
\]
where an additional order-$\epsilon$ term depends on $\omega$ (see Eq.~\eqref{CDMexpansion} in the Appendix~\ref{app:A}). Here we used the fact that
when 
$\bx'\in B_{h}(\mathbf{y})$, a geodesic ball center at $\by$ and radius $h$, the distance bound $\|\bx'-\by\| = d_g(\bx',\by) +O(h^3) \leq C h$
(e.g., see the proof of Prop 3.1~in \cite{jiang2023ghost})to deduce,
\[
\|\nabla_\bx\tilde k_\epsilon(\bx,\by)\|
=\frac{\|\bx-\by\|}{2\epsilon}\exp\!\left(-\frac{\|\bx-\by\|^2}{4\epsilon}\right)
\le \frac{h}{2\epsilon}.
\]
Inserting the leading order of this bound into \eqref{eq:lips+eq1}, employing Lemma~\ref{samplingassumption}, and counting the probability, the proof is complete. We should note that the estimate continues to hold on the diagonal too.

\section{More on Numerical Results}\label{app:num}

\subsection{Heat Kernel}\label{app:num_heat}

This section provides more detailed results corresponding to the summary presented in Section \ref{sec:num_heat}, including
spatial error for a specific source and detailed convergence plots involving both $\epsilon$ and $N$. In addition, we also demonstrate convergence on a disk with Neumann boundary condition; this is to show that the DM kernel can estimate heat kernel on a manifold with boundary, even though the theory is developed for manifolds without boundary.


The disk $\mathbb{D}$ is parameterized by the polar coordinates $(\rho,\theta)$ through
$$
\bx(\rho,\theta) = (\rho\cos\theta, \rho\sin\theta),\quad \rho\in[0,1], \theta\in[0,2\pi).
$$
The heat kernel on $\mathbb{D}$ with Neumann BC is defined through a series sum
$$
\begin{aligned}
k_{\mathbb{D}}((r,\theta),(r_0,\theta_0);t)
&=
\frac{1}{\pi}
+\sum_{k=1}^{\infty}
e^{-\alpha_{0,k}^2t}
\frac{J_0(\alpha_{0,k}r)J_0(\alpha_{0,k}r_0)}
{2\pi N_{0,k}}
+
\sum_{n=1}^{\infty}\sum_{k=1}^{\infty}
e^{-\alpha_{n,k}^2t}
\frac{J_n(\alpha_{n,k}r)J_n(\alpha_{n,k}r_0)
\cos(n(\theta-\theta_0))}
{\pi N_{n,k}},
\end{aligned}
$$
where $\alpha_{n,k}$ is the $k$-th positive root of the derivative of the $n$-th Bessel function, i.e., $J_n'(\alpha_{n,k})=0$, and
$$
N_{n,k}=\int_0^1 J_n(\alpha_{n,k}\rho)^2\rho\,d\rho
=
\frac{J_n(\alpha_{n,k})^2
-J_{n-1}(\alpha_{n,k})J_{n+1}(\alpha_{n,k})}{2}.
$$
Numerically the sum is truncated to $n\leq 42$ and $k\leq 42$, so that the truncation error is less than $10^{-13}$ and lower than the estimation error of the DM kernel.

Again, due to its scaling, the DM kernel is compared with the heat kernel multiplied by $\mathrm{Vol}(\mathcal{M})$.  For disk, $\mathrm{Vol}(\mathbb{D})=\pi$.

\paragraph{Illustrations.} The heat kernels for circle and torus are evaluated on test points specified in Section \ref{sec:num_heat}.  For the disk, the kernel is evaluated on 4096 randomly sampled points with 5 randomly selected sources; all sources are away from the boundary to avoid additional error due to boundary effects.  A diffusion time of $t=0.04$ is chosen for disk.

The heat kernels are illustrated in Fig.~\ref{fig:heat_repr}, where the heat kernel for one of the sources is shown for each case.  Clearly the DM kernel represents well the heat kernel on all the three manifolds, including the disk.

\paragraph{More convergence plots.} Detailed convergence plots involving both $\epsilon$ and $N$ are shown in Fig.~\ref{fig:hk_conv}, where all cases follow a consistent trend, including the disk case.  For each $\epsilon$, the error decreases with larger $N$ until plateau.  When $N$ is sufficiently large, for a fixed $N$, the error decreases with smaller $\epsilon$, which aligns with the convergence theory.



\begin{figure}
\begin{subfigure}{\textwidth}
    \centering
    \includegraphics[width=\textwidth]{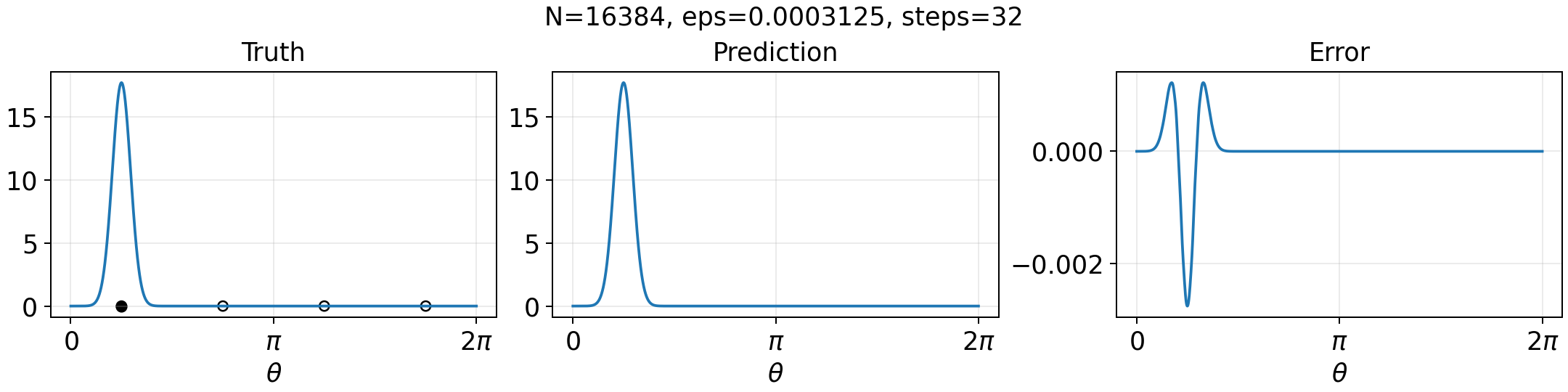}
    \caption{Circle (4 sources in total).}
    \label{fig:hk_circ_repr}
\end{subfigure}
\begin{subfigure}{\textwidth}
    \centering
    \includegraphics[width=\textwidth]{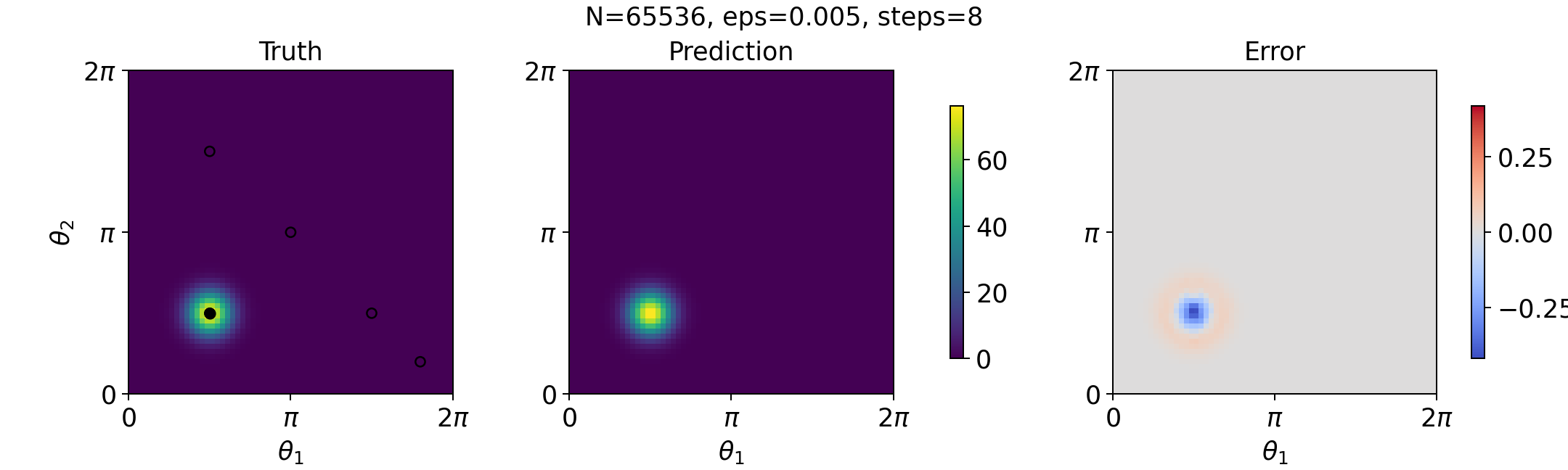}
    \caption{Flat torus (5 sources in total).}
    \label{fig:hk_torus_repr}
\end{subfigure}
\begin{subfigure}{\textwidth}
    \centering
    \includegraphics[width=\textwidth]{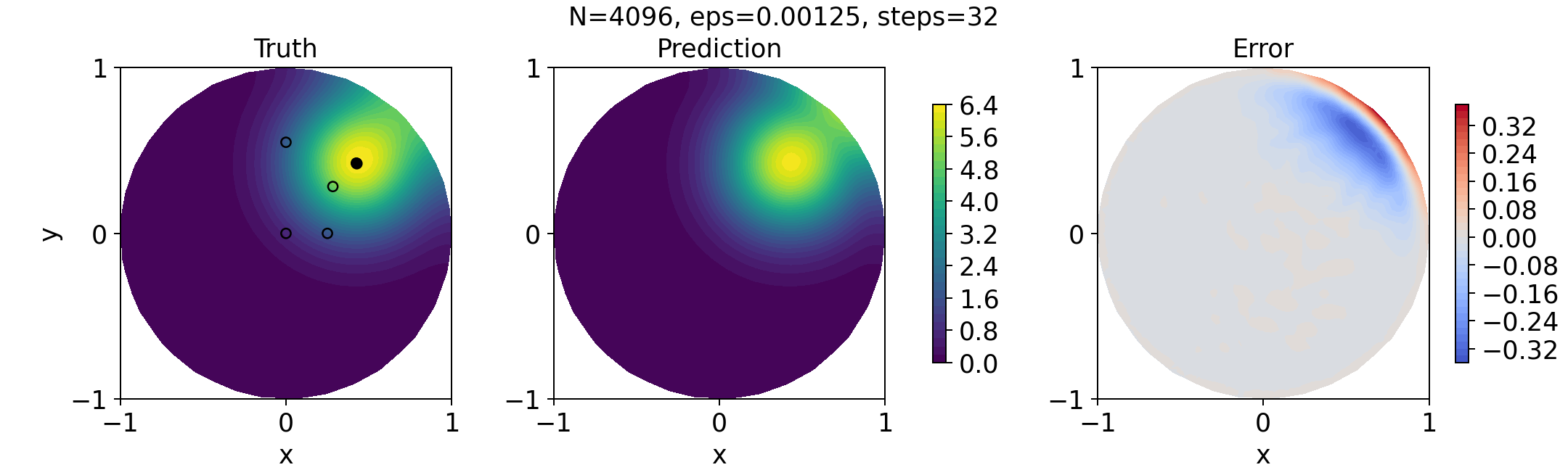}
    \caption{Disk (5 sources in total).}
    \label{fig:hk_disk_repr}
\end{subfigure}
\caption{Illustration of heat kernels on manifolds; heat kernel is shown for the source at the black dot; the convergence results shown in Figs.~\ref{fig:heat_conv} and \ref{fig:hk_conv} are averaged over sources at both the black dot and the circles.}\label{fig:heat_repr}
\end{figure}

\begin{figure}
\begin{subfigure}{0.32\textwidth}
    \centering
    \includegraphics[width=\linewidth]{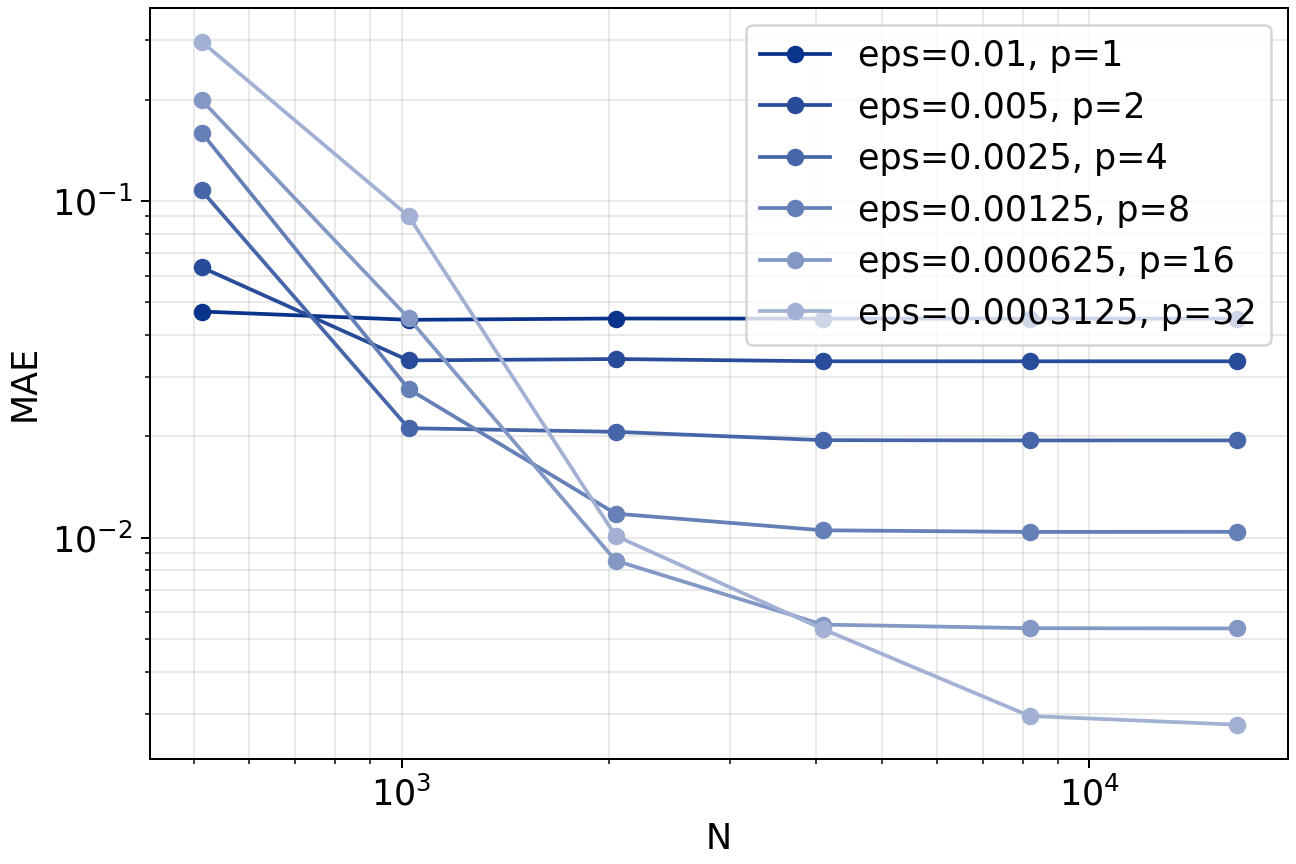}
    \caption{Circle, diffusion time $t=0.01$.}
    \label{fig:hk_conv_circ}
\end{subfigure}
\hfill
\begin{subfigure}{0.32\textwidth}
    \centering
    \includegraphics[width=\linewidth]{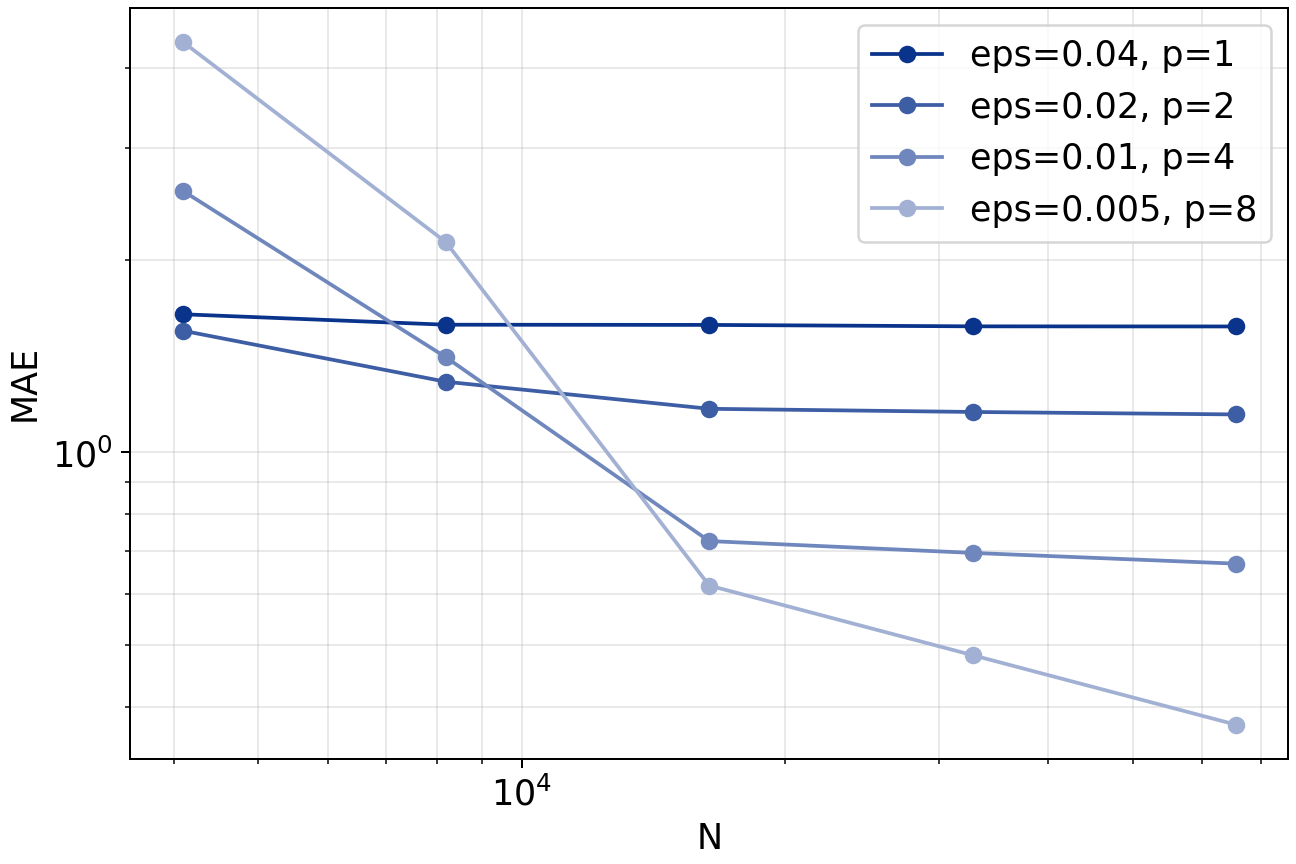}
    \caption{Flat torus, diffusion time $t=0.04$.}
    \label{fig:hk_conv_torus}
\end{subfigure}
\hfill
\begin{subfigure}{0.32\textwidth}
    \centering
    \includegraphics[width=\linewidth]{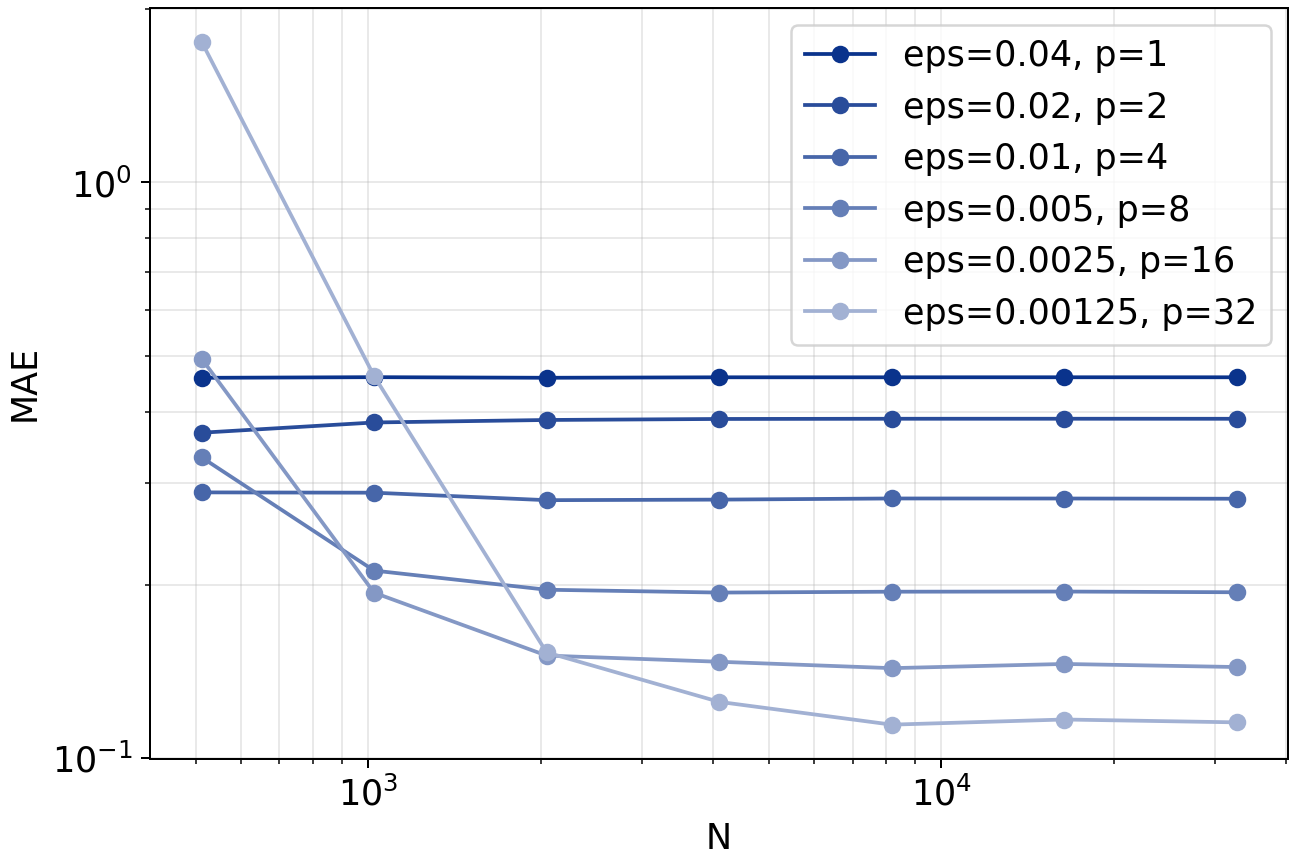}
    \caption{Disk, diffusion time $t=0.04$.}
    \label{fig:hk_conv_disk}
\end{subfigure}
\caption{Convergence of heat kernels on three manifolds.}
\label{fig:hk_conv}
\end{figure}


\subsection{KRR on semicircle and full circle}\label{app:circ}

This section provides more detailed discussion corresponding to the summary presented in Section \ref{sec5.2.1}.

\paragraph{Illustration of the labels.}
Figures \ref{fig:circ_label} and \ref{fig:circ_semi_mode} illustrate the KRR labels and relevant function basis, respectively. Fig.~\ref{fig:circ_label} (a)-(b) illustrate the 12 functions sampled from $W_{\rm LB}$ and $W_{\rm RBF}$, respectively.  Fig.~\ref{fig:circ_semi_mode}(a)-(b) illustrate the basis of the two subspaces, respectively.
Three representative functions from one family are illustrated in Fig.~\ref{fig:circ_label}(c).
Lastly, the full circle LB labels are illustrated in  Fig.~\ref{fig:circ_label}(d).

\begin{figure}[htbp]
    \centering
    \includegraphics[width=\textwidth]{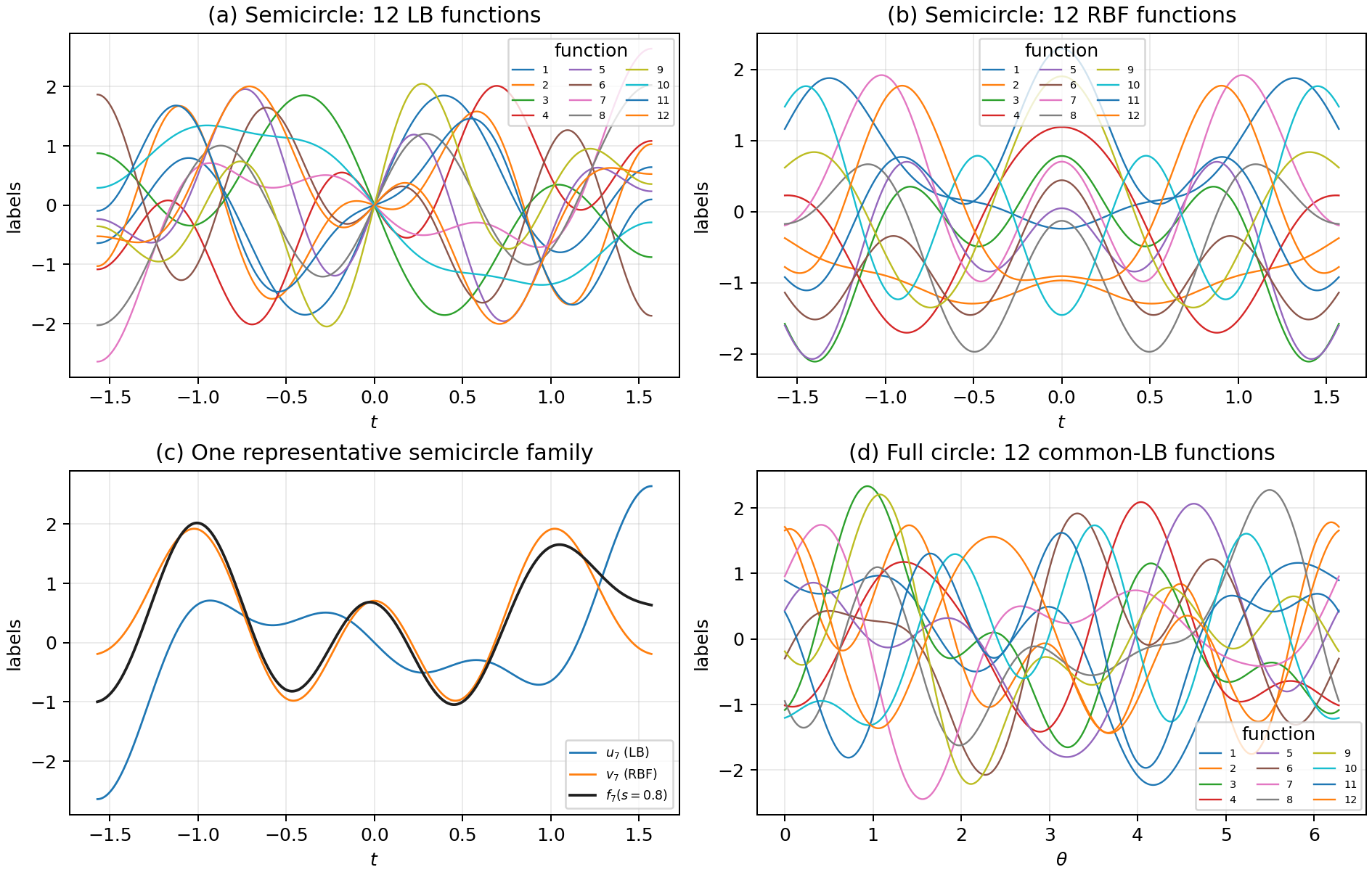}
    \caption{KRR labels for the semicricle and full circle cases.}
    \label{fig:circ_label}
\end{figure}

\begin{figure}[htbp]
    \centering
    \includegraphics[width=\textwidth]{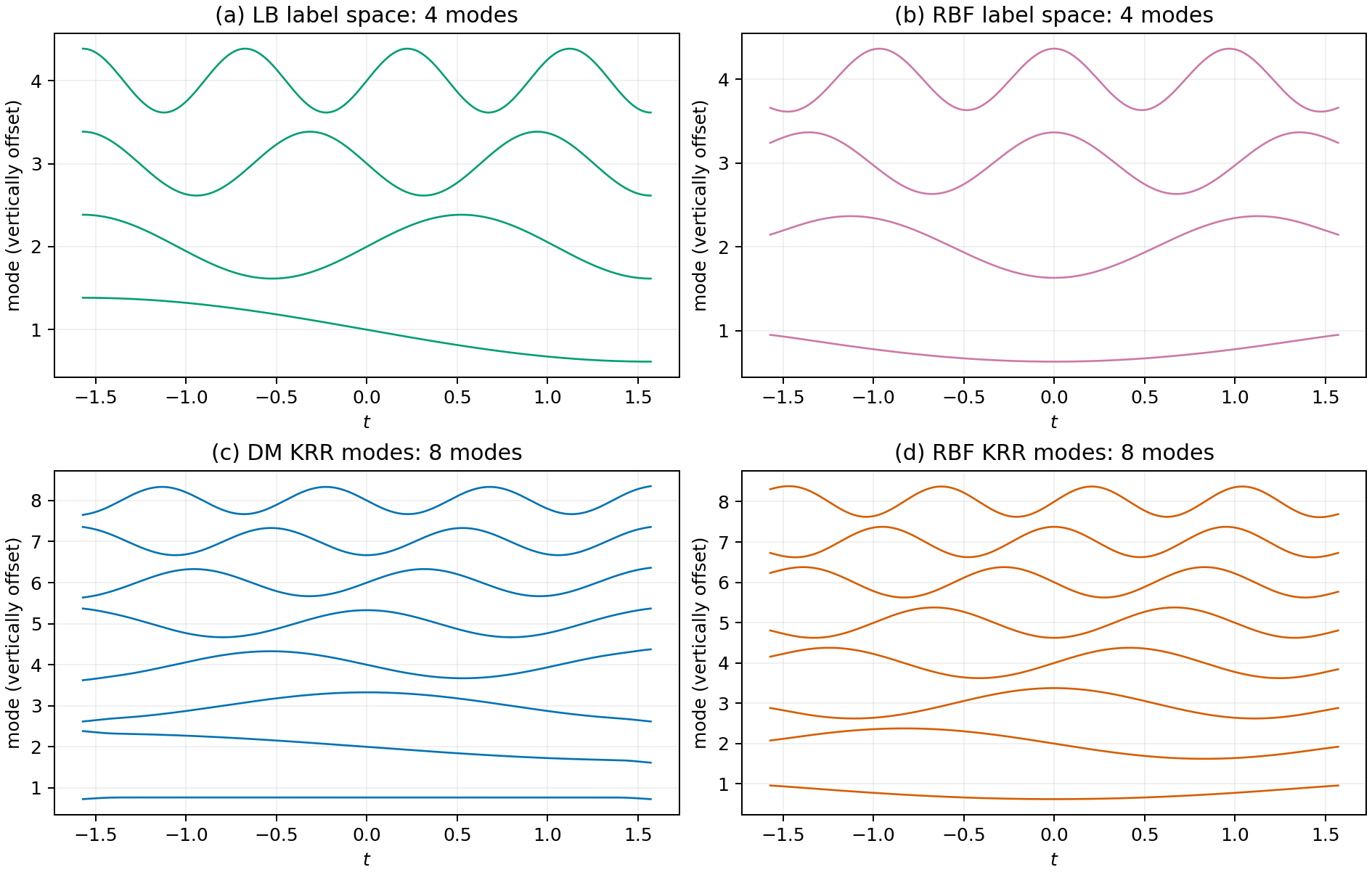}
    \caption{Modes for the semicircle case. The RBF label space (b) are obtained with $\epsilon=0.02$. Panels (c)-(d) show modes obtained from KRR fitting.}
    \label{fig:circ_semi_mode}
\end{figure}

\paragraph{Procedure for KRR fitting.}
The KRR models using DM and RBF kernels are independently tuned for every label using the following procedure.
First, to avoid the variability in random sampling, deterministic equispaced samples are used.
The semicircle uses 1024 equispaced training samples, including endpoints $t=0$ and $t=1$, and 1023 midpoints for validation.
The full circle uses 13 periodic training samples and 13 half-shifted validation samples.
Here, due to the periodicity and translation invariance, the full circle requires much fewer samples to achieve the same error level as in semicircle.
Given the samples, the DM and RBF kernels are evaluated using the ambient coordinates.
Second, the hyperparameters, bandwidth and ridge, are tuned by the average error evaluated at the validation samples.
The optimal hyperparameters are first searched on a $9\times 9$ log-spaced grid over $\epsilon\in[10^{-4},10^2]$ and $\eta\in[10^{-16},10^1]$, and the best candidate is further refined by a Nelder--Mead algorithm.
In addition, for RBF KRR, the hyperparameters are also searched at a fixed $\epsilon=0.02$ and ridge parameter $\eta$ over $[10^{-16}, 10^{-8}]$; this protects against a narrow ridge optimum at the same bandwidth used to construct the RBF target space.
Third, after a KRR model $\hat{f}$ is tuned for a label $f$, the KRR error is quantified using the $L^2$ norm $\|f-\hat{f}\|_{L^2}$.
Numerically, the integral for semicircle $L^2$ norm uses a composite 16-point Gauss–Legendre rule on each of its 1023
training intervals (16368 nodes total), and the integral for full-circle $L^2$ norm uses the 8192-point periodic trapezoidal rule.  The choices of quadrature points ensure machine precision in both cases.

\paragraph{Brief review of subspace alignment analysis.}

Let $\mathcal{U}=\operatorname{span}\{\phi_1,\ldots,\phi_p\}$ and $\mathcal{V}=\operatorname{span}\{\psi_1,\ldots,\psi_q\}$,
where the basis functions are orthonormal under the Hilbert-space inner product $\langle \phi_i,\phi_j\rangle=\delta_{ij}$ and $\langle \psi_i,\psi_j\rangle=\delta_{ij}$.
The relative geometry of the two subspaces is determined from the cross-inner-product matrix $C_{ij}=\langle\phi_i,\psi_j\rangle$, $C\in\mathbb{R}^{p\times q}$.
If $C=L\Sigma R^{\mathsf T}$, then the principal angles are
\[
\theta_i=\arccos(\sigma_i),
\qquad
i=1,\ldots,\min(p,q),
\]
where $\sigma_i$ are the singular values of $C$ in ascending order. These angles characterize the Grassmannian separation between the subspaces, independent of the chosen bases. In this study, the largest principal angle $\theta_{\max}=\theta_1$ is used to characterize the difference between two subspaces.

Procrustes alignment compares the basis functions themselves. Assuming $p\leq q$, the alignment seeks an orthonormal $p$-frame within $\mathcal{V}$ that best matches $\{\phi_i\}_{i=1}^p$:
\[
Q_\star
=
\underset{Q\in\mathbb{R}^{q\times p},\,Q^{\mathsf T}Q=I_p}
{\operatorname{argmin}}
\sum_{i=1}^{p}
\left\|
\phi_i-\sum_{j=1}^{q}\psi_j Q_{ji}
\right\|_{\mathcal{H}}^2.
\]
Using the thin SVD $C=L\Sigma R^{\mathsf T}$, with
$L\in\mathbb{R}^{p\times p}$ and $R\in\mathbb{R}^{q\times p}$, the solution is $Q_\star=RL^{\mathsf T}$.
The aligned functions are therefore
\[
\widetilde{\psi}_i
=
\sum_{j=1}^{q}\psi_j(Q_\star)_{ji},
\qquad i=1,\ldots,p.
\]
This procedure selects and rotates the closest $p$-dimensional frame in
$\mathcal{V}$, while leaving the subspace $\mathcal{V}$ unchanged.

\paragraph{More plots of subspace comparison.}

The typical DM and RBF KRR modes on a semicircle are illustrated in Fig.~\ref{fig:circ_semi_mode}(c) and (d), respectively.  Note that both sets of modes resemble sinusoidal functions, but neither of them satisfy the Neumann boundary conditions.

In the main text, Fig.~\ref{fig:circ_semi_lb} shows that the DM modes can be linearly combined to align well with the LB modes.  Here in Figs.~\ref{fig:circ_semi_rbf} illustrates the opposite case, where the KRR modes are aligned to the RBF modes.  The RBF KRR modes captures the RBF basis, which is expected.  Interestingly, a DM KRR model, tuned against an RBF label, also produces KRR modes that can align well with the RBF modes.

Lastly, for completeness, the alignment of LB basis on a full circle is shown in Fig.~\ref{fig:circ_full_lb}.  Both DM and RBF can align with the true modes with negligible error; this indicates that both are suitable for fitting LB labels on the full circle.

\begin{figure}[htbp]
    \centering
    \includegraphics[width=\textwidth]{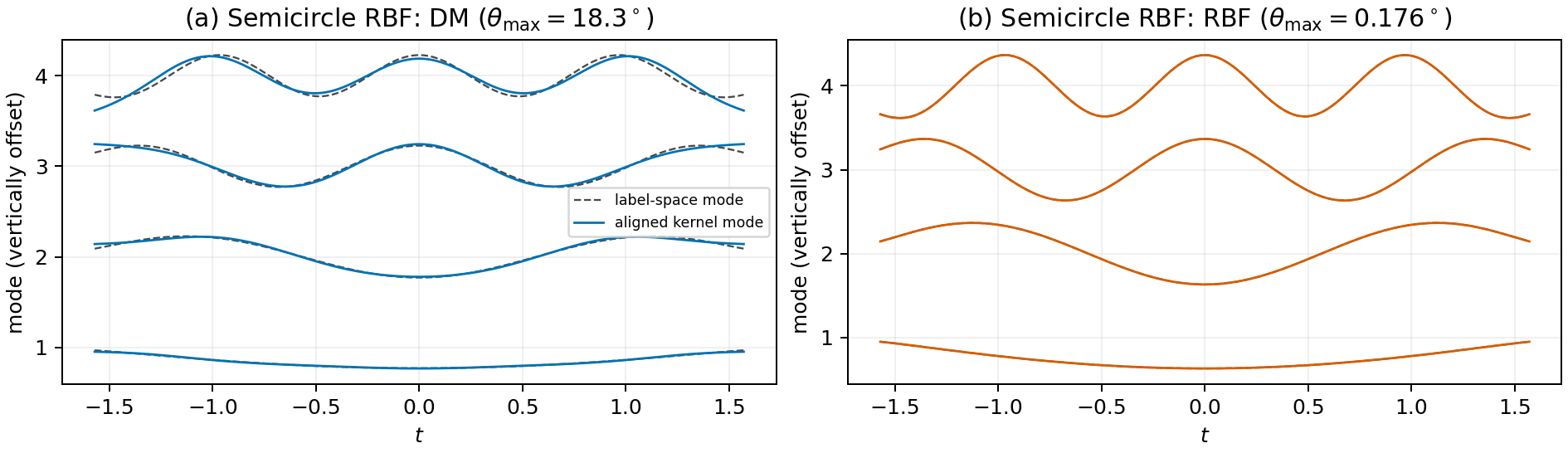}
    \caption{Subspace alignment for the RBF basis.}
    \label{fig:circ_semi_rbf}
\end{figure}

\begin{figure}[htbp]
    \centering
    \includegraphics[width=\textwidth]{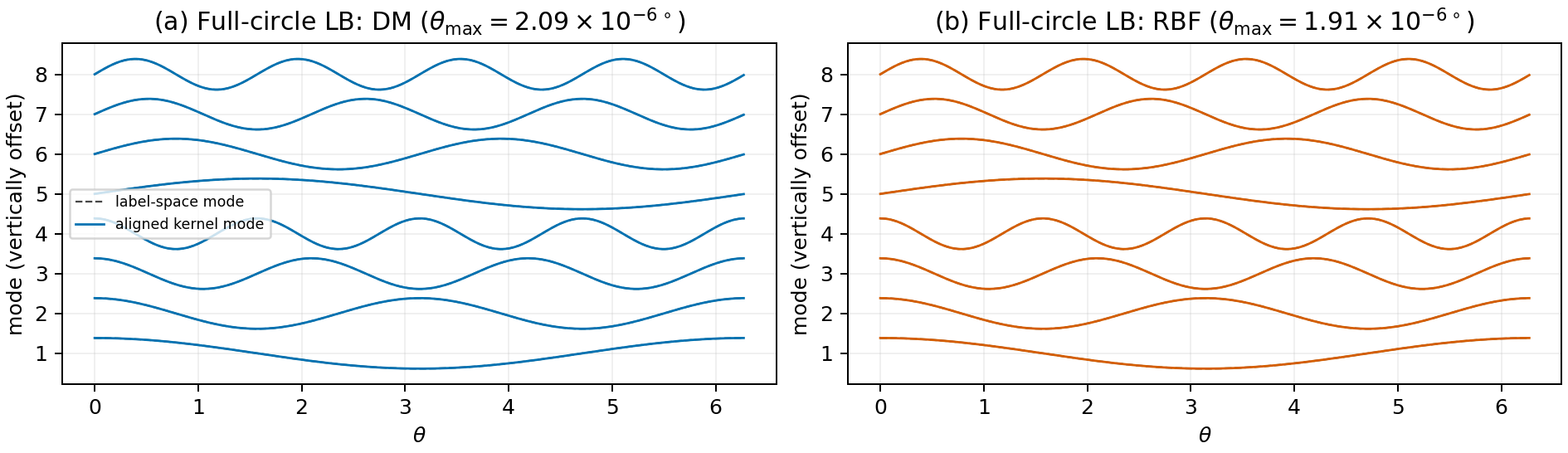}
    \caption{Subspace alignment for the LB basis on full circle.}
    \label{fig:circ_full_lb}
\end{figure}

\subsection{KRR on semi-torus}\label{app:torus}

This section provides more detailed discussion corresponding to the summary presented in Section \ref{sec5.2.2}.

\paragraph{Computation of LB eigenfunctions.}

The reference semi-torus is parameterized by the intrinsic coordinates
$(\theta,\phi)$ through
\[
\bx(\theta,\phi)
=
\bigl(
(a+\cos\theta)\cos\phi,\,
(a+\cos\theta)\sin\phi,\,
\sin\theta
\bigr),
\qquad
\theta\in[0,2\pi),\quad
\phi\in[0,\pi],
\]
where $a=2$. Its surface-area element and total area are
\[
dA(\theta,\phi) = (a+\cos\theta)\,d\theta\,d\phi,
\qquad
\operatorname{Area}(\mathcal M) = 2\pi^2a.
\]
Accordingly, the normalized reference surface-area probability measure used for sampling is
\[
d\omega_{\mathrm{ref}}(\theta,\phi) = \frac{a+\cos\theta}{2\pi^2a}\,d\theta\,d\phi.
\]
Equivalently, the intrinsic coordinates are sampled independently according to
\[
\phi\sim\operatorname{Unif}(0,\pi),
\qquad
p_\theta(\theta) = \frac{a+\cos\theta}{2\pi a},
\qquad \theta\in[0,2\pi).
\]
In the ambient dimension experiment, this measure is also held fixed for every $d$ and is not recomputed from the metric induced by the higher dimensional embedding.

The semi-torus LB eigenfunctions are constructed by separation of variables,
\[
f(\theta,\phi)
=
\Theta(\theta)\Phi(\phi).
\]
For each $\phi$-mode number $m$, the periodic $\theta$-dependent factor satisfies
\[
\Theta''(\theta) - \frac{\sin\theta}{a+\cos\theta}\Theta'(\theta) - \frac{m^2}{(a+\cos\theta)^2}\Theta(\theta)
= -\lambda\Theta(\theta),
\qquad
\Theta(\theta+2\pi)=\Theta(\theta).
\]

For the Neumann and Dirichlet boundary conditions, the eigenfunctions are, respectively, 
\[
f_{m,j}^{\mathrm{N}}(\theta,\phi) = \Theta_{m,j}(\theta)\cos(m\phi),
\qquad
m=0,1,2,\ldots
\]
and
\[
f_{m,j}^{\mathrm{D}}(\theta,\phi) =
\Theta_{m,j}(\theta)\sin(m\phi),
\qquad
m=1,2,\ldots.
\]
They satisfy
\[
\partial_\phi f_{m,j}^{\mathrm{N}}(\theta,0) = \partial_\phi f_{m,j}^{\mathrm{N}}(\theta,\pi) =
0,
\qquad
f_{m,j}^{\mathrm{D}}(\theta,0) = f_{m,j}^{\mathrm{D}}(\theta,\pi) = 0.
\]

The periodic $\theta$-eigenproblem is discretized using $n_\theta=1024$ equally spaced points on $[0,2\pi)$ and second-order centered finite differences with periodic wrapping. For each fixed $m$, the numerical eigenpairs closest to zero are retained and ordered by increasing positive LB eigenvalue. The index $j=0,1,\ldots$ denotes this ordering within a fixed value of $m$. The sign of each numerical eigenvector is fixed deterministically, and periodic cubic-spline interpolation is used to evaluate $\Theta_{m,j}(\theta)$ away from the finite-difference grid.

Because the scaling of a numerical eigenvector is arbitrary, each raw eigenfunction is normalized using its surface-weighted root-mean-square value with respect to $d\omega_{\mathrm{ref}}$. The Neumann and Dirichlet normalization factors, collectively denoted by
\[
\mathcal N_{m,j}
\in
\left\{
\mathcal N_{m,j}^{\mathrm{N}},
\mathcal N_{m,j}^{\mathrm{D}}
\right\},
\]
are defined, respectively, by
\[
\begin{aligned}
\mathcal N_{m,j}^{\mathrm{N}}
&=
\left[
\int_{\mathcal M}
\left|
\Theta_{m,j}(\theta)\cos(m\phi)
\right|^2
\,d\omega_{\mathrm{ref}}(\theta,\phi)
\right]^{1/2},
&
\mathcal N_{m,j}^{\mathrm{D}}
&=
\left[
\int_{\mathcal M}
\left|
\Theta_{m,j}(\theta)\sin(m\phi)
\right|^2
\,d\omega_{\mathrm{ref}}(\theta,\phi)
\right]^{1/2}.
\end{aligned}
\]

The normalized regression targets are therefore
\[
\widetilde f_{m,j}^{\mathrm{N}}(\theta,\phi) = \frac{
\Theta_{m,j}(\theta)\cos(m\phi)
}{
\mathcal N_{m,j}^{\mathrm{N}} 
},
\qquad
\widetilde f_{m,j}^{\mathrm{D}}(\theta,\phi) = \frac{
\Theta_{m,j}(\theta)\sin(m\phi)
}{
\mathcal N_{m,j}^{\mathrm{D}}
}.
\]

Each normalization factor is approximated once using a fixed $256\times256$ intrinsic-coordinate grid and the reference surface-area weight $a+\cos\theta$.  This computation is independent of the sampled training data. The resulting normalization factor is held fixed across all training set sizes $N$, independent trials, and kernel methods.  Consequently, changing $N$ changes only the available training samples and does not change the target eigenfunction or its scale.

\paragraph{Ambient Laplace eigenfunctions}

We also consider regression targets obtained by restricting eigenfunctions of the ambient Euclidean Laplacian in $\mathbb{R}^3$ to the embedded semi-torus. For a wavevector
\[
\bk=(k_1,k_2,k_3)\in\mathbb{R}^3,
\]
the raw target is defined as
\[
f_{\bk}(\bx)
=
\prod_{\ell=1}^{3}
T_\ell(k_\ell x_\ell),
\]
where each $T_\ell$ is independently chosen as either $\sin$ or $\cos$. Since
\[
-\Delta_{\mathbb{R}^3}f_{\bk} = \|\bk\|_2^2f_{\bk},
\]
the corresponding positive Euclidean Laplacian eigenvalue is
\[
\lambda_{\bk} = \|\bk\|_2^2
= k_1^2+k_2^2+k_3^2.
\]

The target is evaluated only on points of the semi-torus, $\bx=\bx(\theta,\phi)$. To make the RMSE values comparable among targets with different wave numbers, each eigenfunction is normalized using the fixed surface-weighted RMS factor
\[
\mathcal{N}_{\bk}
=
\left[
\frac{
\displaystyle
\int_0^{2\pi}\!\!\int_0^\pi
\left|f_{\bk}\bigl(\bx(\theta,\phi)\bigr)
\right|^2
(a+\cos\theta)\,d\phi\,d\theta
}{
\displaystyle
\int_0^{2\pi}\!\!\int_0^\pi
(a+\cos\theta)\,d\phi\,d\theta
}
\right]^{1/2}.
\]
The regression label associated with
$\bx_i=\bx(\theta_i,\phi_i)$ is therefore
\[
y_i
=
\widetilde f_{\bk}(\bx_i)
=
\frac{f_{\bk}(\bx_i)}{\mathcal{N}_{\bk}}.
\]
The normalization factor is approximated once on a fixed $256\times256$ intrinsic-coordinate grid and is held fixed across all training set sizes, independent trials, and kernel methods.

The three ambient space targets used in the numerical experiments are
\[
\begin{aligned}
f_{\mathrm{low}}(\bx)
&= \cos(x_1),
& \bk&=(1,0,0),
& \lambda_{\bk}&=1,
\\
f_{\mathrm{medium}}(\bx)
&= \sin(2x_1)\cos(2x_2)\sin(x_3),
& \bk&=(2,2,1),
& \lambda_{\bk}&=9,
\\
f_{\mathrm{high}}(\bx)
&= \cos(4x_1)\sin(3x_2)\cos(3x_3), & \bk&=(4,3,3), & \lambda_{\bk}
&=34.
\end{aligned}
\]




\paragraph{Effect of the ambient representation dimension.}

To examine learning from high dimensional observations, we keep the latent semi-torus coordinates, target function, and sampling distribution fixed while varying the ambient representation supplied to the kernels. The Dirichlet target is
\[
\widetilde f_{3,1}^{\mathrm{D}}(\theta,\phi)
=
\frac{\Theta_{3,1}(\theta)\sin(3\phi)}
{\mathcal N_{3,1}^{\mathrm{D}}},
\]
and is unchanged for every ambient dimension
\[
n\in\{3,7,11, 15\}.
\]

Given a point $[z_1,z_2,z_3]$ on a unit sphere $S^2$, we first compute its spherical coordinate representation as
\[
\theta=\operatorname{atan2}(z_2,z_1),
\qquad
\phi=\arccos(z_3).
\]

We then use $(\theta, \phi)$ to map onto a general $n$-dimensional ambient space by the following formula:
\begin{equation}\label{torus-formula}
(\theta, \phi) \mapsto
\begin{pmatrix}
  (2+\cos \theta) \cos \phi \\
  (2 + \cos \theta) \sin \phi \\
  \vdots\\
  \frac{2}{n-1}(2+\cos \theta) \cos\frac{n-1}{2}\phi \\
  \frac{2}{n-1}(2+\cos \theta) \sin\frac{n-1}{2}\phi \\
  \sqrt{\sum_{i=1}^{(n-1)/2} i^{-2}}\sin \theta
\end{pmatrix},
\end{equation}
where $n>1$ is an odd integer. Note that, we only consider the range $(\theta, \phi)=[0,2\pi]\times[0, \pi]$ to ensure that the transformed vector field on the torus is continuous.


For every \(n\), samples are drawn from the same normalized reference surface-area measure,
\[
d\omega_{\mathrm{ref}}(\theta,\phi)
=
\frac{a+\cos\theta}{2\pi^2a}\,d\theta\,d\phi.
\]
For each independent trial and training-set size \(N\), the intrinsic coordinates of the training, validation, test samples, and the same randomly sampled hyperparameter pairs are held fixed across all ambient dimensions \(n\). Hyperparameter selection is then performed independently for each ambient dimension and each kernel by choosing the sampled pair that minimizes the corresponding validation error.  Therefore, the selected hyperparameters may vary with \(n\). Our objective is to evaluate the sensitivity of the regression to high dimensional ambient representations of an intrinsically fixed two-dimensional learning function.

\paragraph{Fitting procedure.}

Training, validation, and test samples are drawn independently from the normalized surface-area measure of the reference semi-torus.  The same sampling measure is used for the Neumann, Euclidean Laplacian, and ambient dimension experiments.

For the Neumann and Euclidean Laplacian studies, the training sizes are
\(
N\in\{1024,2048,4096,8192,16384\}.
\)
For each target, five independent trials are performed. In each trial, a training pool of \(16384\) samples is generated and randomly permuted. The training set of size $N$ is obtained from the first $N$ points of this permutation, so the training sets remain nested as \(N\) increases.   Separate validation and test sets containing \(1000\) and \(8000\) samples, respectively, are also generated for each trial and held fixed across all values of \(N\).

For each target, independent trial, and training-set size \(N\), RBF KRR and DMKRR use exactly the same training, validation, and test points. They also use the same collection of candidate hyperparameter pairs. Specifically, $4096$ pairs are sampled independently on logarithmic scales according to
\[
\epsilon = 10^{u_\epsilon},
\qquad
u_\epsilon\sim\operatorname{Unif}[-4,2],
\]
and
\[
\eta = 10^{u_\lambda},
\qquad
u_\lambda\sim\operatorname{Unif}[-16,-10],
\]
where these ranges are selected arbitrarily.  For each kernel, the pair $(\epsilon,\eta)$ that minimizes the validation RMSE is selected.  A new KRR model is then fitted on the corresponding $N$-point training set using the selected parameters and evaluated on the test data.

Evaluating a large number of randomly sampled hyperparameter pairs can be computationally expensive.  Therefore, in practice, a more economical strategy may begin with a coarse search in $(\log_{10}\epsilon,\log_{10}\eta)$ space and subsequently refine the most promising candidates using a gradient-based methods or derivative-free optimization methods, such as the Nelder--Mead algorithm.

In the ambient dimension experiment, the intrinsic samples
$(\theta_i,\phi_i)$, target values, and hyperparameter candidates are also held fixed across all values of $d$. Only the ambient representation $X_d(\theta_i,\phi_i)$ supplied to the kernel changes.  Consequently, this experiment measures sensitivity to increasingly high dimensional nonlinear representations of a fixed latent learning problem.

\paragraph{An additional numerical result.}

Figure~\ref{fig:torus_fourier} shows the target semi-torus labels, which are the ambient Laplace eigenfunctions restricted on the semi-torus, and the corresponding RMSE convergence as a function of $N$. The RMSE values of DM in the 3 target cases are generally lower than those of RBF, and its convergence is faster than that of RBF.  This empirically shows that the advantage of the DM kernel is not restricted to targets that are only the Laplace--Beltrami eigenfunctions of the semi-torus. Also, these targets are not necessarily Dirichlet or Neumann at the semi-torus boundary.

\begin{figure}[htbp]
    \centering
    \includegraphics[width=\textwidth]{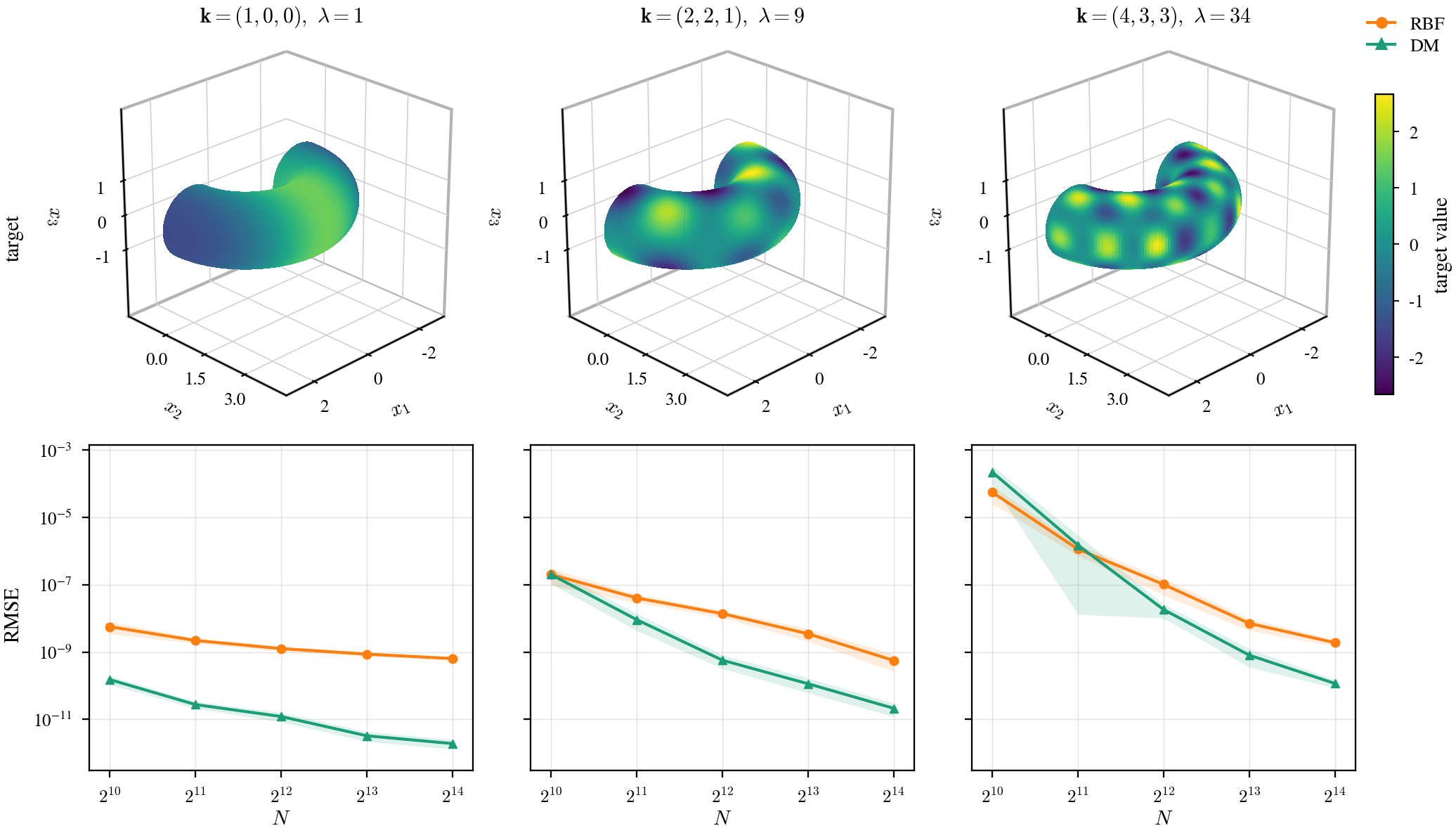}
    \caption{Laplace eigenfunctions in ambient space constrained to semi-torus.}
    \label{fig:torus_fourier}
\end{figure}

\end{document}